\documentclass[mnsc]{informs3}
\RequirePackage{tgtermes}
\RequirePackage{newtxtext}
\RequirePackage{newtxmath}
\RequirePackage{endnotes}

\OneAndAHalfSpacedXI 

\usepackage{algpseudocode}
\usepackage{tikz}

\usepackage{natbib}
 \bibpunct[, ]{(}{)}{,}{a}{}{,}%
 \def\bibfont{\small}%
\usepackage{amsmath,amsfonts,cancel}
\usepackage{thmtools}
\usepackage{pgfplots} 
\pgfplotsset{compat=1.18} 
\usepackage{adjustbox}
\usepackage[shortlabels]{enumitem}
\usepackage{mathtools}
\usepackage{dsfont}
\usepackage{boldline} 
\usepackage{tcolorbox}
\usepackage{booktabs}
\usepackage{xspace}
\usepackage{subcaption}
\usepackage{tensor}
\usepackage{tikz}
\usetikzlibrary{patterns}
\usepackage{natbib}
\usepackage{rotating} 

\usepackage[ruled]{algorithm2e} 

\SetAlFnt{\small}
\SetAlCapFnt{\small}
\SetAlCapNameFnt{\small}
\SetAlCapHSkip{0pt}
\IncMargin{-\parindent}

\usepackage{placeins}
\usepackage[colorlinks=true,breaklinks=true,bookmarks=true,urlcolor=magenta,citecolor=magenta,linkcolor=magenta,bookmarksopen=false,draft=false]{hyperref}
\usepackage[capitalize]{cleveref}

\theoremstyle{TH}%

\newtheorem{fact}{Fact}

\usepackage[suppress]{color-edits}
\addauthor{srs}{orange}
\addauthor{ch}{magenta}
\addauthor{rev}{red}
\addauthor{min}{blue}

\newcommand{\Nest}{N_{\text{est}}}
\newcommand{\Msep}{\widetilde{M}}

\newcommand{\ALG}{\textsf{RCN}\xspace}
\newcommand{\ALGWS}{\textsf{RCN}$^{\textsf{ws}}$\xspace}
\newcommand{\ALGplus}{\textsf{RCN}$^+$\xspace}

\newcommand{\IgnorantSAA}{\textsf{Naive SAA}\xspace}
\newcommand{\SubsampleSAA}{\textsf{Subsample SAA}\xspace}
\newcommand{\TrueSAA}{\textsf{True SAA}\xspace}

\newcommand{\KM}{\textsf{KM}\xspace}
\newcommand{\RKM}{\textsf{Robust KM}\xspace}

\newcommand{\CensoredSAA}{\textsf{Censored SAA}\xspace}
\newcommand{\U}{\mathcal{U}}
\newcommand{\Uest}{\U_{\textsf{est}}}

\newcommand{\Ind}[1]{\mathds{1}\left\{ #1 \right\}}

\renewcommand{\Pr}{\mathbb{P}}

\newcommand{\E}{\mathbb{E}}

\DeclarePairedDelimiter{\abs}{\lvert}{\rvert}

\newcommand{\id}{\textup{id}}
\newcommand{\ui}{\textup{ui}}
\newcommand{\ke}{\textup{ke}}

\newcommand{\Ber}{\text{Ber}}

\newcommand{\GKE}{\mathcal{G}^{\textup{ke}}}
\newcommand{\GI}{\mathcal{G}^{\textup{id}}}
\newcommand{\GUI}{\mathcal{G}^{\textup{ui}}}
\newcommand{\Gminus}{G^-}
\newcommand{\Gminushat}{\widehat{G}^{-}}
\newcommand{\Gminushatk}[1]{\widehat{G}^{-}_{#1}}

\newcommand{\F}{\mathcal{G}}
\newcommand{\CDF}{G}

\newcommand{\eCDF}{\widehat{\CDF}}

\renewcommand{\d}{d}
\newcommand{\doff}{\d^{\textup{off}}}

\newcommand{\qopt}{q^\star}
\newcommand{\order}{q}
\newcommand{\Event}{\mathcal{E}}
\newcommand{\Regret}{\textup{Regret}}
\newcommand{\UncensoredRegret}{\textup{Vanilla-Regret}}
\newcommand{\qbar}{M}

\newcommand{\B}[1]{\mathcal{B}(#1; G)}
\newcommand{\qcrit}{q^{\dagger}}
\newcommand{\qcrithat}{{q}^{\dagger}_{\offeCDF}}
\newcommand{\risk}{\Delta}
\newcommand{\riskws}{\risk^{\textsf{ws}}}
\newcommand{\width}{\zeta}
\newcommand{\qrisk}{q^{\risk}}
\newcommand{\qriskws}{q^{\riskws}}

\newcommand{\offeCDF}{\widehat{G}}
\newcommand{\offeCDFtrue}{\widehat{G}^{\textup{uc}}}

\newcommand{\inv}[1]{I_{#1}^{\textup{off}}}
\newcommand{\stockoutdays}{\mathcal{T}}
\newcommand{\sales}[1]{\soff_{#1}}

\newcommand{\qalg}{q^{\textup{alg}}}

\newcommand{\func}{Q^\dagger}

\newcommand{\soff}{s^{\textup{off}}}
\newcommand{\qoff}{q^{\textup{off}}}
\newcommand{\lambdasamples}{s^{\textup{off},\lambda}}
\newcommand{\ambig}[1]{\mathcal{F}(#1;G)}
\newcommand{\ber}

\newcommand{\minimaxquant}{minimax optimal ordering quantity}
\newcommand{\criticalquantile}{unidentifiable ordering quantity}

\usepackage{multirow}

\crefname{assumption}{assumption}{assumptions}
\crefname{algocf}{algorithm}{algorithm}

\mathchardef\mhyphen="2D 
\ifdefined \argmin
\else \DeclareMathOperator*{\argmin}{arg\,min}
\fi
\ifdefined \argmax
\else \DeclareMathOperator*{\argmax}{arg\,max}
\fi

\let\originalleft\left
\let\originalright\right
\renewcommand{\left}{\mathopen{}\mathclose\bgroup\originalleft}
\renewcommand{\right}{\aftergroup\egroup\originalright}

\EquationsNumberedThrough    

\TheoremsNumberedThrough     
\ECRepeatTheorems  %

\MANUSCRIPTNO{MNSC-0001-2024.00}

\begin{document}


\RUNAUTHOR{Hssaine and Sinclair}

\RUNTITLE{Data-Driven Censored Newsvendor}

\TITLE{The Data-Driven Censored Newsvendor Problem}

\ARTICLEAUTHORS{%
\AUTHOR{Chamsi Hssaine}
\AFF{Department of Data Sciences and Operations,
University of Southern California, Marshall School of Business, \EMAIL{hssaine@usc.edu}}

\AUTHOR{Sean R. Sinclair}
\AFF{Department of Industrial Engineering and Management Sciences,
Northwestern University, \EMAIL{sean.sinclair@northwestern.edu}}
} 

\ABSTRACT{%
We study a censored variant of the data-driven newsvendor problem, where the decision-maker must select an ordering quantity that minimizes expected overage and underage costs based only on {{offline} {\em censored} sales data}, rather than historical demand realizations.
Our goal is to understand how the degree of historical demand censoring affects the performance of any learning algorithm for this problem. To isolate this impact, we adopt a distributionally robust optimization framework, evaluating policies according to their worst-case regret over an {\em ambiguity set} of distributions. This set is defined by the largest historical order quantity (the {\em observable boundary} of the dataset), and contains all distributions matching the true demand distribution up to this boundary, while allowing them to be arbitrary afterwards.
We demonstrate a {\it spectrum of achievability} under demand censoring by deriving a natural necessary and sufficient condition under which vanishing regret is an achievable goal. In regimes in which it is not, we exactly characterize the information loss due to censoring: an insurmountable lower bound on the performance of any policy, {\it even when the decision-maker has access to infinitely many demand samples}. We then leverage these sharp characterizations to propose a natural {\em robust} algorithm that adapts to the historical level of demand censoring.    
We derive finite-sample guarantees for this algorithm across all possible censoring regimes and show its near-optimality with matching lower bounds (up to polylogarithmic factors). We moreover demonstrate its robust performance via extensive numerical experiments on both synthetic and real-world datasets.
}%




\KEYWORDS{Newsvendor, censored data, data-driven decision-making, distributionally robust optimization, minimax regret, finite-sample guarantees} 

\maketitle


\section{Introduction}\label{sec:intro}
The classical newsvendor problem models settings where a decision-maker makes an inventory decision to fulfill future random demand, with the goal of minimizing expected overage and underage costs once demand is realized. In the most basic version of this problem --- wherein the decision-maker makes a single decision and knows the distribution from which demand is drawn --- the optimal solution is well-known to be the critical quantile, an ordering quantity that balances the likelihood of overage and underage according to their relative costs \citep{zipkin2000foundations}. In more realistic settings, however, the decision-maker does not know the true demand distribution; instead, she has access to data (e.g., historical demand) that allows her to learn good ordering policies despite this {information} gap. This is the so-called {\it data-driven} newsvendor problem. 

In the offline setting, when the decision-maker only needs to make a {\it single} ordering decision given historical demand realizations, the Sample Average Approximation (SAA) heuristic, which computes the sample critical quantile, has been shown to be near-optimal, with extensive efforts devoted to deriving tight performance guarantees for this natural policy \citep{chen2024survey}. This line of work, however, assumes that demand is {\it fully observable}, an assumption that fails to hold in many real-world applications, such as brick-and-mortar retail settings.  In these settings, more often than not, demand is {\it censored}. That is, if the decision-maker only had a fixed amount of inventory to give out over a period of time, she would only observe how many units were {\it sold}, as opposed to how many were actually {\it desired}. Learning in these settings, then, is likely to be more challenging, especially in highly resource-constrained settings where historical demand consistently dwarfs historical inventory levels (e.g., demand for monocultural events such as the Eras Tour, or for hotel rooms associated with conference venues). When this occurs, the decision-maker has likely {\it never} observed the true demand realizations.

We conjecture that this challenge is why the offline, data-driven censored newsvendor problem has been explored to a much lesser extent. Indeed, to the best of our knowledge, the design of algorithms with finite-sample guarantees for this problem has been studied in only two fairly recent works \citep{ban2020confidence, fan2022sample}, both of which rely on the crucial assumption that a constant fraction of historical inventory levels exceed the critical quantile (henceforth referred to as the {\it optimal newsvendor quantity}). Upon first reflection, the necessity of such an assumption is evident: if all historical inventory levels were zero, no demand would ever be observed. In this case, what can the decision-maker possibly hope to accomplish? Still, this assumption remains inherently unsatisfying, for two reasons. The first reason is practical: there exist many highly constrained real-world settings where there is no reason to expect that historical inventory levels were in fact high enough to capture the optimal newsvendor quantity, as in the aforementioned examples. The second reason is more fundamental; in particular, it is not clear that this assumption is even testable, given that it requires knowledge of the true newsvendor quantity. Faced with this fact, the decision-maker has no option but to deploy the proposed algorithms and hope for the best {(a solution we show fails dramatically when the assumption does not hold)}. Our work is motivated by this important gap in the literature. Namely, we seek to answer the following research questions:
\smallskip

\begin{center}
{\em 
{What is the impact of demand censoring on algorithm performance in the {offline} data-driven newsvendor problem? Can we characterize a lower bound on algorithm performance as a function of the level of historical demand censoring? Do there exist simple algorithms that asymptotically achieve this lower bound?
}
}
\end{center}

\subsection{Main Contributions}

\subsubsection*{A flexible modeling framework.} In an attempt to tackle these questions, one of the main contributions of this work is a modeling framework that allows us to formalize the impact of demand censoring on algorithm performance for the newsvendor problem. In particular, while a standard performance metric for the data-driven newsvendor problem is the additive optimality gap relative to the optimal newsvendor cost under complete information, we argue that this benchmark ceases to be meaningful in highly censored settings. 
This is most obvious for the pathological instance where all historical inventory levels are zero, in which case there is clearly no hope for the decision-maker to learn the optimal newsvendor quantity. 
Thus motivated, we turn to the distributionally robust optimization (DRO) framework. Specifically, we consider the {\it ambiguity set} of a given instance to be the set of all possible distributions sharing the same cumulative distribution function (cdf) as the true (unknown) demand distribution, up until the maximum historical inventory level, referred to as the {\it observable boundary} of the dataset and denoted by $\lambda$. After $\lambda$, distributions in the ambiguity set may take an arbitrary shape (\Cref{def:robustness_set}). Given the ambiguity set, we evaluate the performance of any data-driven policy according to its worst-case additive optimality gap compared to the optimal newsvendor quantity (i.e., its worst-case {regret}), {\it over all possible distributions in the ambiguity set} (\Cref{def:min_regret}). (Such an ambiguity set was first defined in \citet{bu2023offline} for the problem of offline pricing under censored demand.)

Introducing this modeling framework for the data-driven censored newsvendor problem directly addresses the benchmark issue described above. When $\lambda$ is low, much of the demand will never be observed. In such information-poor settings, we posit that it is more appropriate for a decision-maker to instead view nature as adversarial. This view is reflected by a large ambiguity set of distributions against which she needs to compete. In historically unconstrained settings, on the other hand, when $\lambda$ is large, existing work hints at the idea that the decision-maker's objective is fundamentally easier \citep{ban2020confidence,fan2022sample}. This other extreme is reflected in a much smaller ambiguity set of distributions against which to compete.

\subsubsection*{Impact of censoring on achievable algorithm performance.} This flexible framework allows us to formalize the existence of a {\it spectrum of achievability} determined by the observable boundary of the dataset. Our first main technical contribution is to exactly characterize this spectrum (\Cref{thm:minimax-risk-identifiable}). In particular, we provide a {\it necessary and sufficient} condition for any algorithm to achieve vanishing regret in the worst case. This condition turns out to be equivalent to the condition assumed in  \citet{ban2020confidence} and \citet{fan2022sample}: the observable boundary $\lambda$ must exceed the optimal newsvendor quantity. We moreover recover that in this {\it identifiable} regime, the minimax optimal ordering quantity is exactly equal to the optimal newsvendor quantity corresponding to the true demand distribution. This result formalizes that the optimality of the critical quantile is robust to a wide range of censoring levels.

In the {\it unidentifiable} setting where this condition fails to hold, however, our result implies that the regret of any algorithm is bounded away from zero, even if the decision-maker has access to infinitely many samples. We show this by providing an exact lower bound on the worst-case regret in such settings, which we refer to as the {\it minimax risk} $\risk$ and which can be interpreted as the information loss due to demand censoring. The minimax risk exhibits the natural property of being decreasing in $\lambda$, reflecting the intuition that the decision-maker's problem becomes easier as the data becomes less censored. We moreover produce an explicit expression for the minimax optimal ordering quantity, which similarly depends on $\lambda$, and at a high level hedges against a worst-case distribution that places weight exclusively on $\lambda$ and a known upper bound on the optimal newsvendor quantity.

\subsubsection*{A robust, near-optimal algorithm.} 
{These characterizations prove to be pivotal in demonstrating that a simple algorithm asymptotically achieves the fundamental lower bound due to demand censoring, as the number of samples grows large.}
Our algorithm, \textsf{Robust Censored Newsvendor} (\ALG, \Cref{alg:newsvendor}), proceeds in two stages. In the first stage, it tests for identifiability by estimating the fraction of demand that lies below $\lambda$. If this estimate exceeds the critical ratio by some appropriately tuned confidence parameter, the algorithm outputs an empirical estimate of the optimal newsvendor quantity. If it falls short of the critical ratio by some confidence term, on the other hand, it classifies the problem as unidentifiable and computes an empirical estimate of the minimax optimal ordering quantity. In between, it is unable to conclusively determine identifiability, and defaults to outputting $\lambda$. 

This practical algorithm lies on the spectrum of achievability characterized in \Cref{thm:minimax-risk-identifiable}. Namely, across {\it all} regimes of identifiability, it guarantees $O(1/\sqrt{N})$ worst-case regret with respect to the minimax risk with probability $1-\delta$, where $N$ is the number of samples associated with ordering quantity $\lambda$ (henceforth referred to as samples {at the boundary}), and $\delta$ determines the algorithm's confidence level for identifiability (\Cref{thm:minmax_regret}). While seemingly intuitive, the above description sweeps under the rug that our algorithm (necessarily) only uses censored data to compute all of its estimates. Hence, it is a priori unclear that any of these would in fact be unbiased, a prerequisite to our algorithm obtaining its strong guarantees. Herein lies a crucial design choice: our algorithm {\it only} uses samples at the boundary. This choice is quite subtle: while censored data in general will produce biased estimates, by subsetting amongst samples at the boundary, we are able to recover unbiasedness of all empirical estimates.

We subsequently complement these upper bounds with matching lower bounds (up to polylogarithmic factors) for all possible censoring regimes (\Cref{thm:lb}). We do so via a unified treatment for all possible instances, reducing the decision-making problem to that of hypothesis testing.
Finally, we demonstrate the strong practical performance of our algorithm via extensive computational experiments, on both synthetic and real-world datasets. We specifically explore the dependence of our algorithm's numerical performance on (i) the number of samples $N$, (ii) the observable boundary $\lambda$, and (iii) the variability of the underlying demand distribution. {We observe that our algorithm asymptotically achieves the minimax risk $\risk$ across all censoring regimes,} as compared to state-of-the-art benchmarks that assume identifiability~\citep{ban2020confidence,fan2022sample}. We moreover propose a practical modification to \ALG (\ALGplus, \Cref{alg:newsvendor_plus_plus}) which leverages the entirety of the dataset, as opposed to only samples at the boundary; we observe that this modification can yield significant performance improvements in the identifiable regime.

\subsubsection*{Paper organization.}  We review the related literature in the remainder of this section. We formally present our model in \cref{sec:preliminary}, and formalize the minimax risk of the censored newsvendor in \cref{sec:minimax-risk}. In \cref{sec:upper-bound} we design and analyze a best-of-both-worlds algorithm which asymptotically achieves the minimax risk. We complement this with a matching lower bound in \cref{sec:lower-bound}.  We demonstrate our algorithm's strong numerical performance in \cref{sec:experiments}, and conclude in \Cref{sec:conclusion}.
\subsection{Related Literature}
\label{sec:related_work}

{The newsvendor model is one of the most foundational models in the operations literature, beginning with the work of \citet{scarf1957min}. In the past few decades, there has been significant focus on incorporating data-driven methods to solve variants of this canonical model, when the decision-maker either has access to {\em offline} data or can collect data adaptively (i.e., {\em online}). In this section we discuss the most closely related literature to the offline censored newsvendor problem we consider. We refer readers to \citet{chen2024survey} for a comprehensive survey of data-driven newsvendor results.}

\subsubsection*{Data-Driven Newsvendor: Non-Parametric Setting.} 
{Our work falls within the large body of literature on the data-driven newsvendor in offline, non-parametric settings. Within this line of work, the most common assumption is that the decision-maker has access to {\it uncensored} demand samples. For this setting, the Sample Average Approximation (SAA) algorithm was first shown to be near-optimal in \citet{levi2007provably}. Their upper bound was later refined by \citet{levi2015data}, who established the dependence of SAA accuracy {on} the weighted mean spread of the demand distribution, and by \citet{lin2022data}, who characterized its dependence on the local flatness of the demand distribution {around} the optimal ordering quantity. \citet{cheung2019sampling} showed that these upper bounds are tight by providing a lower bound on the number of demand samples required for an algorithm to achieve low relative regret with high probability. While the bounds provided by these latter works are most relevant in the asymptotic regime in which the number of samples is large, \citet{besbes2023big} more recently provided an exact analysis of the SAA algorithm across all data sizes, demonstrating that only tens of samples suffice for this algorithm to achieve strong performance. Contrary to this line of work, we assume that the decision-maker only has access to {\it censored} demand samples, a setting in which the classical SAA algorithm is not implementable, {\it sans} modification. We note however that, for the special case where the minimum historical ordering quantity goes to infinity, our setting reduces to the uncensored setting, in which case our results recover near-optimality of the SAA algorithm.}

{While censored demand has been studied in the literature, the vast majority of existing work considers the {\it online} setting, in which the decision-maker makes sequential ordering decisions over a finite horizon. Early work by \citet{godfrey2001adaptive} and \citet{huh2009nonparametric} demonstrated the strong performance of gradient-based methods in these settings. \citet{huh2011adaptive} showed that adaptively using the well-known Kaplan-Meier (KM) estimator \citep{kaplan1958nonparametric}, which reconstructs the empirical cumulative distribution function by uniformly redistributing the mass of censored observations to uncensored observations, converges almost surely to the set of optimal solutions under discrete demand. \citet{besbes2013implications} later characterized the implications of demand censoring in repeated newsvendor problems. Less closely related to our work is the design of sequential decision-making algorithms under censored demand, for more general models with inventory carry-over (e.g., with warehouse capacity constraints~\citep{shi2016nonparametric}, setup costs~\citep{yuan2021marrying,fan2024don}, positive lead times \citep{agrawal2022learning, xie2024vc}, uncertain supply \citep{chen2024learning}), and in nonstationary settings \citep{lugosi2024hardness,keskin2025nonstationary}. 

The online setting differs fundamentally from the offline setting in the presence of censored demand since, in the online case, the decision-maker can adaptively adjust order quantities based on observed sales, thereby influencing future samples. In contrast, the offline setting offers no such control. This distinction partially explains the dearth of literature on the impact of censored data on single-period newsvendor decisions. 
\citet{ban2020confidence} was the first to tackle this question, under the assumption that the maximum historical ordering quantity exceeds the optimal ordering quantity (i.e., in the identifiable regime). They derive an asymptotically consistent estimator for the optimal policy and use this to provide asymptotic confidence intervals. In an early version of their work, \citet{fan2022sample} {studied} the sample complexity of learning the optimal ordering quantity when historical samples are generated from a given data-collecting policy. They show that a variant of the SAA algorithm is near-optimal, assuming that (i) historical inventory levels are independently and identically distributed, and (ii) a constant fraction of historical inventory levels exceed the optimal ordering quantity (an assumption that is analogous to identifiability in our setting).\footnote{Subsequent to our work, \cite{fan2025sample} echoed our results, establishing that under i.i.d. data-collecting policies, this condition is necessary to learn the newsvendor solution to a certain accuracy.} {We make two important contributions relative to these recent works. Firstly, contrary to \citet{ban2020confidence}, we are interested in designing algorithms with {\it finite-sample} guarantees. More importantly, however, the primary motivator behind our work is to characterize the fundamental limits that demand censoring places on algorithm performance, therefore requiring us to relax the identifiability assumptions upon which both of these papers rely. This motivation allows us to design simple algorithms that achieve these limits with strong guarantees {\it across all censoring regimes} (i.e., irrespective of how the historical ordering quantities are generated, and even if {\it no} historical ordering quantity exceeds the optimal newsvendor quantity).}\footnote{\minedit{Similarly subsequent to our work, \cite{kumar2026value} study the problem of exactly characterizing the worst-case regret of classical data-driven algorithms (e.g., Kaplan-Meier) for the newsvendor problem. Their results recover our main insight that any algorithm's performance is fundamentally limited by the level of demand censoring in the dataset. However, in contrast to our work, they do not seek to design algorithms that achieve robust performance across all levels of censoring.}}

Finally, we highlight that while the KM estimator is applicable to our setting, to the best of our knowledge, finite-sample guarantees do not exist for this heuristic. More importantly, however, this estimator only estimates the tail of the distribution at the observable boundary; past the observable boundary, it is not defined. As a result, we do not expect it to perform well in the unidentifiable regime, when the optimal newsvendor quantity is past the maximum historical ordering level. We demonstrate its poor performance in this regime in our numerical experiments (see \Cref{sec:experiments}).

\subsubsection*{Data-Driven Newsvendor: {Robust and Bayesian Settings}.} While our work is concerned with the non-parametric setting, we briefly mention the large body of literature on {robust and Bayesian} newsvendor models, where the decision-maker has access to additional information on the underlying demand distribution (e.g., a parametric form, or its moments). 

Our modeling framework's reliance on the ambiguity set induced by the true demand distribution and the observable boundary $\lambda$ is directly inspired by the distributionally robust newsvendor literature. In this body of work, \citet{scarf1957min} first studied the minimax solution of the newsvendor problem when both the mean and variance of the true demand distribution are known; \citet{gallego1993distribution} later simplified and extended the analysis to variants of the newsvendor model. \citet{perakis2008regret} consider the minimax regret objective when the decision-maker knows various distributional quantities such as the range and median. More recent work has assumed knowledge of the distribution's semi-variance, under demand asymmetry \citep{natarajan2018asymmetry}, and considered the distributionally robust newsvendor problem under a Wasserstein ambiguity set \citep{lee2021data, fiechtner2025wasserstein}.  

Contrary to this line of work, however, our ambiguity set construction exists not because of some exogenously given information about the true demand distribution; rather, it arises from the fact that demand is unobservable past $\lambda$. 
Before the observable boundary, we assume no additional information on the demand distribution.
In light of this, our work can best be viewed as lying at the intersection of the literature on non-parametric and distributionally robust newsvendor models. Recent work by \citet{xu2022robust} and \citet{fu2024distributionally} on the uncensored data-driven newsvendor is philosophically in the same hybrid vein, as they leverage historical samples to construct an ambiguity set defined by the nonparametric characteristics of the true distribution. We note that the notion of ambiguity set we consider is related to the first-order stochastic dominance (FSD) ambiguity set defined in \citet{fu2024distributionally}. Our ambiguity set, however, cannot exactly be cast as a FSD ambiguity set, given that we assume that the mean, as opposed to the support, of the underlying demand distribution is bounded. Moreover, \citet{fu2024distributionally} are interested in the absolute cost, as opposed to the notion of {regret} we consider to characterize the performance of our learning algorithm. Finally, \minedit{\citet{besbes2022beyond} and \citet{zhang2024more} consider robust optimization formulations for the uncensored newsvendor problem, when samples are drawn from a biased distribution.}

We briefly mention the line of work on the {\it Bayesian} newsvendor problem, originating with the works of \citet{scarf1959bayes}, followed by \citet{azoury1985bayes} and  \citet{liyanage2005practical} for the uncensored setting. In this {\it sequential} setting, the decision-maker has a prior belief on the true demand distribution, and updates it in each round as she observes additional demand realizations. Both uncensored \citep{saghafian2016newsvendor} and censored \citep{lariviere1999stalking,ding2002censored,lu2008analysis,bensoussan2009note,bisi2011censored,jain2015demand,mersereau2015demand,besbes2022exploration} settings have been studied under this framework. 

Finally, the {\it feature-based} newsvendor problem --- in which the demand realization depends on some observable features of the data --- has drawn increasing attention in recent years \citep{ban2019big, ding2024feature, zhang2024optimal, fu2024distributionally}. Extending our results to the contextual offline setting would be an interesting and practically relevant future direction.

\subsubsection*{Impact of Demand Censoring: Beyond the Classical Newsvendor.} There exists a large body of work in operations on demand estimation under censored data. A common approach is to assume that the demand distribution has a parametric form (see, e.g., \citet{nahmias1994demand}, \citet{agrawal1996estimating}, \citet{vulcano2012estimating}, \citet{mersereau2015demand}). In contrast, our focus is on the {nonparametric} setting, as previously discussed. Most relevant to our work in this regard is \citet{bu2023offline}, who introduced the concept of problem identifiability for an offline {\it pricing} problem under censored data, and whose hybrid robustness framework we leverage to characterize the impact of demand censoring on the {newsvendor} problem. It is notable that, in contrast to their results, we do not need a well-separatedness condition on the underlying demand distribution, and moreover are able to leverage the structure of the newsvendor model for an {\em exact} characterization of the worst-case regret. We also provide matching lower bounds across all regimes of identifiability.\footnote{Less closely related to our study is the line of work on estimating nonparametric choice models from limited choice data \citep{farias2013nonparametric}. There, the source of censoring comes from the fact that the customer's consideration set of items is unobserved in the data.}

Finally, we highlight the deep connection between the data-driven newsvendor problem and the statistics literature on quantile estimation (see, e.g. \citet{harrell1982new}, \citet{kalgh1982generalized}, \citet{tierney1983space}, \citet{yang1985smooth}, \citet{zielinski1999best} for classical works on nonparametric estimators in the uncensored setting). 
As mentioned above, in the censored setting, the classical benchmark is the KM estimator \citep{kaplan1958nonparametric}. More recent work has considered censored quantile {regression}, in which the censored outcome (in our case, demand) is determined by a set of observed covariates \citep{powell1986censored,portnoy2003censored}. In all of these works, the main point of focus has been deriving consistency and asymptotic normality properties of estimators, as opposed to finite-sample guarantees for some notion of regret. Moreover, as with the KM estimator, these asymptotic guarantees all rely on equivalent notions of identifiability.  

\section{Problem Formulation}
\label{sec:preliminary}

\smallskip

\subsubsection*{Technical notation.} In what follows, for $N \in \mathbb{N}_+$, we let $[N] = \{1,2,\ldots,N\}$. We moreover use $\Pr_G(\cdot)$ and $\E_G[\cdot]$ to respectively denote the probability of an event and the expectation of a random variable when the source of underlying randomness has a \mbox{cumulative distribution function (cdf) $G$}.

\subsubsection*{Model primitives.} We consider the classical single-period newsvendor problem, in which a decision-maker faced with random demand $D \geq 0$ must decide on the number of units to satisfy this demand. We let $G$ denote the cdf of $D$, and assume that $\E_G[D] < \infty$. 

For any ordering decision $q \geq 0$, once demand is realized and fulfilled to the maximum extent possible, the decision-maker incurs a lost sales penalty $b > 0$ for each unsatisfied unit of demand (also referred to as the {\it underage cost}). If all demand is satisfied and there is leftover inventory at the end of the period, the decision-maker incurs a per-unit {\it overage cost} $h > 0$. Given $G$, the decision-maker's goal is to determine an ordering quantity that minimizes the newsvendor cost, given by:
\begin{equation}
\label{eq:cost}
C_\CDF(q) = \E_\CDF\left[b(D - q)^+ + h(q - D)^+\right],
\end{equation}
where $(\cdot)^+ = \max\{\cdot,0\}$ is used to denote the positive part. When $G$ is known, the optimal ordering quantity, denoted by $\qopt_G$, is determined by the {{\it critical ratio}} $\rho = \frac{b}{b+h} \in (0,1)$ \citep{zipkin2000foundations}. Formally, {the so-called {\em critical quantile}} $\qopt_G$ is the $\rho$-th quantile of $G$, i.e., 
\begin{equation}
\label{eq:newsvendor}
\qopt_G = \inf\left\{q \mid G(q) \geq \rho\right\}.
\end{equation}

In the {data-driven} setting we consider, however, the decision-maker does not know $G$, but has access to historical (equivalently, offline) {\it sales} data. {The historical data is defined by $K \in \mathbb{N}^+$ historical ordering quantities, denoted by $\qoff_k$, for $k \in [K]$, with $\qoff_1 < \qoff_2<\ldots<\qoff_K$. We assume that the historical ordering quantities are exogenous and fixed.\footnote{\citet{bu2023offline} similarly consider a setup with finitely many fixed and exogenous historical ordering quantities.} For each ordering quantity $\qoff_k$, $k \in [K]$, the decision-maker observes $N_k$ sales samples \mbox{$\soff_{ki} = \min\left\{\doff_{ki},\qoff_k\right\}$}, $i \in [N_k]$, where $\doff_{ki}$ refers to the true realization of historical demand in period $i$ associated with ordering quantity $\qoff_k$, assumed to be drawn independently and identically (i.i.d.) from $G$. We moreover assume that $\doff_{ki}$ and $\qoff_k$ are independent, for all $k \in [K], i \in [N_k]$.  Let $\doff = \left(\doff_{ki}, i \in [N_k], k \in [K]\right)$ and \mbox{$\soff = \left(\soff_{ki}, i \in [N_k], k \in [K]\right)$}.  
We let $N = N_K$ be the number of samples associated with the largest historical ordering quantity; we moreover denote $\lambda = \qoff_K$ to be this largest historical ordering quantity. As we will later see, this latter quantity, henceforth referred to as the {\it observable boundary} of the dataset, plays a key role in learning the optimal order quantity.

Finally, we assume that the decision-maker has access to a known upper bound $\qbar$ on $\qopt_G$,  and let \ifdefined\arxiv \[\F = \left\{F \mid F \text{ has nonnegative support}, \E_F[D] < \infty, \, \qopt_F \leq \qbar\right\}\] \else \mbox{$\F = \left\{F \mid F \text{ has nonnegative support}, \E_F[D] < \infty, \, \qopt_F \leq \qbar\right\}$} \fi be the set of distributions satisfying the above assumptions on the true underlying cdf $G$.

\begin{remark}
{The assumption that $\qbar$ is known is common in the newsvendor literature (\citet{huh2009nonparametric}, \citet{besbes2013implications}, \citet{ban2020confidence}). For instance, if the decision-maker knows the support of the underlying distribution $G$, she may take $M$ to be an upper bound on this support. Our main results will show that access to a finite upper bound on $\qopt_G$ is in fact {\it necessary} to avoid pathologies in which infinite regret is unavoidable for the decision-maker. This fact will become clear in the subsequent analysis; hence, we defer a counterexample establishing necessity of this assumption to \Cref{remark:well-sep}.}
\end{remark}

\subsubsection*{Objective.} Let $\pi$ denote a policy which takes as input the historical sales data $\soff$, and outputs an ordering quantity $q^\pi \in [0,M]$. We use $\Pi$ to denote the set of all such mappings. 

In the classical data-driven newsvendor problem, the decision-maker observes the true realizations of historical demand, rather than censored sales data. In this uncensored setting, a standard metric to evaluate the performance of a policy $\pi$ is the additive optimality gap of $q^\pi$ relative to the optimal ordering quantity $\qopt_G$. We refer to this optimality gap as the {{\it vanilla regret}} of policy $\pi$, formally defined as:
\begin{align}\label{eq:uncensored-regret}
\UncensoredRegret(q^\pi) = C_G(q^\pi)-C_G(\qopt_G).
\end{align}

In the censored setting we consider, however, this notion of regret may or may not be meaningful, depending on the level of demand censoring. To see this, consider an extreme case where $\lambda = +\infty$, and demand is bounded. In this case, the decision-maker always observes the true demand when $\lambda$ was the ordering quantity. This setting reduces to the classical uncensored data-driven newsvendor over these demand samples, in which case competing against $\qopt_G$ is indeed a meaningful metric.

On the other hand, consider the pathological case in which $\lambda = 0$ (i.e., $\qoff_k = 0 \ \forall \ k \in [K]$). In this case, the decision-maker never observes {any} demand information! Since the dataset is completely uninformative in this case, it is philosophically more natural to consider an {adversarial} framework in which the policy should aim to perform well against {\it any} distribution chosen by nature.

{To interpolate between these two pathological cases, we draw from the distributionally robust optimization (DRO) literature, and consider the {\it ambiguity set} of demand distributions induced by $\lambda$.
}

\smallskip

\begin{definition}[Ambiguity set]
\label{def:robustness_set}
Given $\lambda$, the \emph{ambiguity set} associated with $G$ is the set of all distributions $F \in \F$ that share the same cdf as $G$, for $x < \lambda$.{\footnote{The philosophical dependence of the performance metric on $\lambda$ motivated the introduction of such an ambiguity set in \citet{bu2023offline}, for the problem of offline pricing with censored data.}} Formally:
\begin{align}
\ambig{\lambda} =\left\{F \in \F \mid F(x) = G(x) \ \forall \ x < \lambda\right\}.
\end{align}
\end{definition}

\smallskip

When $\lambda = 0$, $\ambig{\lambda} = \F$, i.e., the ambiguity set consists of all possible cdfs satisfying our mild distributional assumptions. When $\lambda = +\infty$, however, there is no ambiguity surrounding the underlying demand distribution, and $\ambig{\lambda} = \{G\}$. Now, for the non-pathological cases where $\lambda \in (0, +\infty)$, the ambiguity set captures the idea that, prior to $\lambda$, there may be hope of reconstructing the true cdf $G$, required to compute $\qopt_G$, by \Cref{eq:newsvendor}.  Past $\lambda$, however, the decision-maker never observes any demand realizations. Hence, even with infinitely many samples, she will never be able to estimate the tail of the true demand distribution. In this case, we recover the adversarial framework that was motivated when $\lambda = 0$; namely, for all intents and purposes, the tail of the demand distribution past $\lambda$ can be arbitrary. We illustrate the dependence of the ambiguity set on $\lambda$ in \Cref{fig:robustness_set}. This example moreover pictorially shows that, for $\lambda' > \lambda$, $\ambig{\lambda'} \subset \ambig{\lambda}$.

\begin{figure}[t]
\centering
{
\scalebox{.85}{
\tikzset{every picture/.style={line width=0.75pt}} 
\tikzset{every picture/.style={line width=0.75pt}} 

\begin{tikzpicture}[x=0.75pt,y=0.75pt,yscale=-1,xscale=1]

\draw [line width=1.5]    (78,40) -- (78,248) ;
\draw [shift={(78,36)}, rotate = 90] [fill={rgb, 255:red, 0; green, 0; blue, 0 }  ][line width=0.08]  [draw opacity=0] (11.61,-5.58) -- (0,0) -- (11.61,5.58) -- cycle    ;
\draw [line width=1.5]    (78,248) -- (392.5,248) ;
\draw [shift={(396.5,248)}, rotate = 180] [fill={rgb, 255:red, 0; green, 0; blue, 0 }  ][line width=0.08]  [draw opacity=0] (11.61,-5.58) -- (0,0) -- (11.61,5.58) -- cycle    ;
\draw    (174,237) -- (174,258) ;
\draw  [dash pattern={on 0.84pt off 2.51pt}]  (174,74) -- (174,247.5) ;
\draw    (337,240) -- (337,261) ;
\draw  [dash pattern={on 0.84pt off 2.51pt}]  (337.5,74) -- (337,250.5) ;
\draw [line width=1.5]    (78,248) .. controls (126.5,242) and (114,169) .. (174.5,157) ;
\draw  [dash pattern={on 0.84pt off 2.51pt}]  (81,156) -- (337.25,156) ;
\draw  [dash pattern={on 0.84pt off 2.51pt}]  (174,74) -- (337.5,74) ;
\draw [color={rgb, 255:red, 208; green, 2; blue, 27 }  ,draw opacity=1 ][line width=1.5]    (174.5,157) .. controls (179.5,76) and (147.5,75) .. (337.5,74) ;
\draw [color={rgb, 255:red, 245; green, 166; blue, 35 }  ,draw opacity=1 ][line width=1.5]    (174.5,157) .. controls (178.5,77) and (218.5,84) .. (337.5,74) ;
\draw [color={rgb, 255:red, 74; green, 144; blue, 226 }  ,draw opacity=1 ][line width=1.5]    (174.5,157) .. controls (318.5,154) and (333.5,144) .. (337.5,74) ;
\draw [color={rgb, 255:red, 144; green, 19; blue, 254 }  ,draw opacity=1 ][line width=1.5]    (174.5,157) .. controls (292.5,134) and (313.5,132) .. (337.5,74) ;
\draw [color={rgb, 255:red, 0; green, 0; blue, 0 }  ,draw opacity=1 ][line width=1.5]    (174.5,157) -- (337.5,74) ;
\draw [color={rgb, 255:red, 80; green, 227; blue, 194 }  ,draw opacity=1 ][line width=1.5]    (174.5,157) .. controls (212.5,93) and (248.5,97) .. (337.5,74) ;
\draw [color={rgb, 255:red, 65; green, 117; blue, 5 }  ,draw opacity=1 ][line width=1.5]    (241.5,122) .. controls (263.5,91) and (301.5,93) .. (337.5,74) ;
\draw [color={rgb, 255:red, 184; green, 233; blue, 134 }  ,draw opacity=1 ][line width=1.5]    (241.5,122) .. controls (291.5,123) and (301.5,93) .. (337.5,74) ;
\draw  [dash pattern={on 0.84pt off 2.51pt}]  (242,122) -- (242,249) ;
\draw    (242,238) -- (242,259) ;
\draw  [dash pattern={on 0.84pt off 2.51pt}]  (242,122) -- (78.5,122) ;

\draw (168,263) node [anchor=north west][inner sep=0.75pt]   [align=left] {$\displaystyle \lambda $};
\draw (28,148) node [anchor=north west][inner sep=0.75pt]   [align=left] {$\displaystyle \Gminus(\lambda)$};
\draw (328,262) node [anchor=north west][inner sep=0.75pt]   [align=left] {$\displaystyle M$};
\draw (67,67) node [anchor=north west][inner sep=0.75pt]   [align=left] {$\displaystyle 1$};
\draw (65,241) node [anchor=north west][inner sep=0.75pt]   [align=left] {$\displaystyle 0$};
\draw (400,245) node [anchor=north west][inner sep=0.75pt]   [align=left] {$\displaystyle x$};
\draw (39,26) node [anchor=north west][inner sep=0.75pt]   [align=left] {$\displaystyle F( x)$};
\draw (235,263) node [anchor=north west][inner sep=0.75pt]   [align=left] {$\displaystyle \lambda' $};
\draw (28,112) node [anchor=north west][inner sep=0.75pt]   [align=left] {$\displaystyle \Gminus(\lambda')$};

\end{tikzpicture}}
}
\caption{
{Illustration of the ambiguity set $\ambig{\lambda}$ induced by an observable boundary $\lambda$ and cdf $G$, represented by the black curve. Here, the seven colored curves are cdf's in $\ambig{\lambda}$: they coincide with $G(x)$ for all $x < \lambda$, and are arbitrary afterwards. Note that only the light and dark green curves are contained in $\ambig{\lambda'}$, since all other curves deviate from $G(x)$ for some $x \in [\lambda, \lambda')$.\label{fig:robustness_set}}}
\end{figure}

Given the ambiguity set $\ambig{\lambda}$, our performance metric for any policy $\pi \in \Pi$ is its worst-case optimality gap against the optimal newsvendor cost of {\it any} distribution $F \in \ambig{\lambda}$. We refer to this additive optimality gap as the {\it regret} of a policy, which we formally define below.

\smallskip
\begin{definition}[Regret]\label{def:min_regret}
Given true demand distribution $G$ and {dataset $\soff$} {with observable boundary $\lambda$}, the \emph{regret} of policy $\pi \in \Pi$ with access to censored demand samples $\soff$ is defined as:
 \begin{equation}\label{eq:regret}
\Regret(q^\pi) = \sup_{F \in \ambig{\lambda}} C_F(q^\pi) - C_F(\qopt_F).
\end{equation}
\end{definition}

\medskip

We will equivalently refer to the regret of a policy as its {\it worst-case} regret, in line with the fact that we take the supremum over all $F \in \ambig{\lambda}$.\footnote{While we focus on the additive optimality gap of newsvendor policies, another common metric that has been considered for the data-driven uncensored newsvendor problem is the {\it relative} optimality gap, i.e., $\frac{C_G(q^\pi)-C_G(\qopt_G)}{C_G(\qopt_G)}$. While our exact algorithm and analytical results are tied to the additive regret definition, we conjecture that our techniques still apply to the relative regret metric.}
{Notice that the regret of a policy, as defined in \Cref{eq:regret}, implicitly depends on (i) the true demand distribution $G$, and (ii) the observable boundary $\lambda$. As a result, our guarantees will be instance-dependent, as a function of these two model primitives, in addition to $b$ and $h$}.

The aim is to design a policy with low regret, for any true underlying distribution $G$, as we scale $N$, the number of data samples at the boundary.\footnote{The number of samples associated with the largest ordering level is a common choice for the scaling parameter in censored settings (see, e.g., \citet{bu2023offline} and \citet{fan2022sample}). At a high level, $N$ can be viewed as the ``bottleneck'' of the dataset, in that it is the limiting factor in accurately estimating the true underlying cdf $G$.} In the remainder of our work, we refer to this problem as the {\it data-driven censored newsvendor problem}.

\smallskip 

{Observe that, since {$G \in \ambig{\lambda}$} for all $\lambda$, $\Regret(q^\pi) \geq \UncensoredRegret(q^\pi)$ for any policy $\pi$.} In the same vein, since $\ambig{\lambda'} \subset \ambig{\lambda}$ for $\lambda' > \lambda$, the regret of a policy is (weakly) decreasing in $\lambda$. This formalizes the natural idea that {the censored setting is indeed {at least as hard} as the uncensored setting}, and only becomes more challenging as more demand realizations are censored.  In the following section we formalize the extent to which this is true, and how this depends on the value of $\lambda$.

\begin{remark}
While we make minimal assumptions on the underlying demand distribution for our main set of results, in Appendix \ref{app:well_separated} we explore the value the decision-maker can derive from additional distributional information. In particular, we study the case where the distribution is {\it globally well-separated} (i.e., continuous with a known lower bound $\gamma > 0$ on its probability density function).\footnote{This assumption is commonly made in the literature on the data-driven newsvendor \citep{huh2009nonparametric,besbes2013implications}.}
\end{remark}

\section{Cost of Demand Censoring in the Data-Driven Newsvendor}\label{sec:minimax-risk}

In the uncensored setting, it is well-known that the vanilla regret (defined in \Cref{eq:uncensored-regret}) of any policy is lower bounded by 
\[\inf_{\pi \in \Pi} \mathbb{E}_G\left[C_G(q^\pi)-C_G(\qopt_G)\right] = \Omega(1/\sqrt{N}).\]
This lower bound is achieved by the classical sample average approximation (SAA) algorithm {\citep{levi2007provably}}, which outputs the $\rho$-th empirical quantile of $G$, i.e.,
\begin{align}\label{eq:uncensored-saa}
{\qopt_{\offeCDF}} = \inf\left\{q \mid \frac{1}{\sum_{k\in[K]}N_k}\sum_{k=1}^K\sum_{i=1}^{N_k} \Ind{\doff_{ki} \leq q} \geq \rho \right\}.
\end{align}

We first highlight that this solution is not implementable in the censored setting, since $\doff_{ki}$ is unobserved. While there exist natural adaptations of ${\qopt_{\offeCDF}}$ to the censored setting (e.g., replacing $\Ind{\doff_{ki} \leq q}$ by $\Ind{\soff_{ki} \leq q}$ in \Cref{eq:uncensored-saa}, or conditioning on uncensored samples), we show in \Cref{sec:experiments} that these naive solutions exhibit remarkably poor performance, both on vanilla and worst-case regret. Rather than searching for a policy that exhibits strong performance, however, in this section we ask a more fundamental question: {\it Is vanishing regret even achievable in the censored setting?}

We return to our extreme cases to answer this question neither in the affirmative nor in the pejorative. To see this, consider first the case where $\lambda = +\infty$ and demand is bounded. In this case, $\ambig{\lambda} = \{G\}$, as argued in \Cref{sec:preliminary}. Moreover, samples for which $\lambda$ was the ordering quantity are uncensored; computing ${\qopt_{\offeCDF}}$ over these samples guarantees $O(1/\sqrt{N})$ regret, by existing results for the uncensored setting \citep{chen2024survey}. Consider now the case where $\lambda = 0$. In this case, the ambiguity set $\ambig{\lambda} = \F$, implying that the quantity output by any policy $\pi$ must simultaneously compete against the space of {\it all} possible distributions. For any quantity $q^\pi$, however, it is not difficult to construct a distribution such that $q^\pi$ incurs constant loss relative to the optimal ordering quantity $\qopt_F$, {\it independent of the number of samples $N$.} To illustrate this idea, consider the policy that computes the $\rho$-th empirical quantile with respect to sales. This policy outputs $q^{\pi} = 0$ when $\lambda = 0$. Consider now the atomic distribution $F \in \ambig{\lambda}$ for which $D = M$ with probability 1. For this distribution, $\qopt_F = M$, which implies that $\Regret(q^\pi) = \Omega(M)$, a constant independent of $N$. 

Thus motivated, in this section we investigate the dependence of the optimal achievable performance of any policy on the observable boundary $\lambda$.

\subsection{Algorithm Performance and Minimax Risk}

Understanding whether or not vanishing regret is achievable, for a fixed $\lambda$, is closely related {to} the concept of problem identifiability, first introduced by \citet{bu2023offline}.

\smallskip

\begin{definition}[Problem identifiability]
\label{def:identifiability}
Given true demand distribution $G$ and observable boundary $\lambda$, the data-driven censored newsvendor problem is \emph{identifiable} if there exists a policy $\pi$ such that, for any $\epsilon > 0$:
\begin{equation}
\label{eq:identifiability}
\lim_{N \rightarrow \infty} \Pr_{{G}}\left[ \sup_{F \in \ambig{\lambda}} \Big\{ C_F(q^\pi) - C_F(\qopt_F)\Big\} < \epsilon \right] = 1,
\end{equation}
where the historical ordering quantities $\qoff_k$ are fixed, and the probability is taken over the randomness in the censored demand samples {$\doff_{ki} \sim G$.}  {If no such policy exists, the problem is \emph{unidentifiable}.}
%
\end{definition}

\smallskip

In words, a problem is unidentifiable if, even with infinitely many samples, no policy can achieve zero regret in the worst case. Note the unwieldiness of this definition, however, as it is a statement about the space of {\it all} policies $\pi \in \Pi$. To gain further insight into identifiability, we introduce the more tractable notion of minimax risk, which will be central to all of our results.
\smallskip 
\begin{definition}[Minimax risk]
\label{def:risk}
The \emph{minimax risk} $\risk$ of a data-driven censored newsvendor problem defined by true demand distribution $G$ and {observable boundary $\lambda$} is the minimum achievable regret of any constant $q \in [0,M]$. Formally:
\begin{equation}
\label{eq:risk}
\risk = \inf_{\order \in [0,M]} \sup_{F \in \ambig{\lambda}} \Big\{C_F(\order) - C_F(\qopt_F)\Big\}.
\end{equation}
We refer to the quantity that achieves $\risk$, if it exists, as the {\emph{\minimaxquant}} $\qrisk$.
\end{definition}
\smallskip

At a high level, the minimax risk can be thought of as the {\it cost of demand censoring}. It quantifies the information loss due to the decision-maker never being able to observe demand realizations that exceed $\lambda$. This concept moreover allows us to distinguish between the two sources of information loss in the data-driven censored newsvendor problem: (i) the information loss due to the fact there are finitely many samples of data, which also exists in the uncensored setting and depends on $N$, and (ii) the information loss due to censoring, which is dependent on $\lambda$, but independent of $N$.

The following proposition creates a formal connection between the minimax risk and identifiability. It implies that, if $\risk > 0$, {\it the regret of any policy is lower bounded by a constant.}

\begin{proposition}\label{prop:iff}
{The data-driven censored newsvendor problem is identifiable if and only if \mbox{$\risk = 0$}.}
\end{proposition}


\Cref{prop:iff} thus motivates us to move the goalpost for the success of any given policy to incurring regret that is upper bounded by $\risk + o(1)$. We say that any policy that achieves this is {\it near-optimal}.

We defer the proof of \Cref{prop:iff} to Appendix \ref{apx:iff}. While it is easy to establish that identifiability implies $\risk = 0$, the other direction is more challenging, as it requires exhibiting a policy that achieves vanishing regret when $\risk = 0$. The proof of this direction actually follows from our main algorithmic contribution in \Cref{sec:upper-bound}, in which we design and analyze such a policy. 

We highlight that neither $\risk$ nor $\qrisk$ is computable by the decision-maker, since the supremum in \Cref{eq:risk} is taken over the ambiguity set $\ambig{\lambda}$, which itself depends on $G$, unknown to the decision-maker. As a result, $\risk$ is not an operationalizable quantity. Despite this fact, in the following section, we obtain {\it explicit characterizations} of $\risk$ and $\qrisk$, depending on the problem primitives. Not only will these be important in understanding the 
{cost of demand censoring in the data-driven newsvendor problem}, but they will also be instrumental in designing a simple algorithm 
{that asymptotically achieves the fundamental lower bound $\risk$}.

\subsection{Closed-Form Characterization of Minimax Risk}\label{sec:closed-form}

In our first main contribution, we leverage the more tractable notion of minimax risk by providing a {closed-form characterization} of $\risk$ and $\qrisk$, and {in doing so} obtain a simple necessary and sufficient condition for identifiability.
For ease of notation, we let $\Gminus(\lambda) = \Pr_G(D < \lambda)$ denote the mass of demand observations that lie before $\lambda$. At a high level, $\Gminus(\lambda)$ can be viewed as a proxy for the ``censoring level'' of the data, since $1-\Gminus(\lambda)$ is precisely the fraction of demand that the decision-maker will never observe, even with infinitely many samples. \Cref{thm:minimax-risk-identifiable} establishes that whether or not a problem is identifiable hinges solely on the relative ordering of $\Gminus(\lambda)$ and $\rho$.

{
\begin{theorem}\label{thm:minimax-risk-identifiable}
Consider a data-driven censored newsvendor problem with demand distribution $G$ and {observable boundary $\lambda$}.
\begin{enumerate}
\item If $\Gminus(\lambda) \geq \rho$: 
\[\qrisk = \qopt_G < \lambda \qquad , \qquad \risk = 0.\]
\item If $\Gminus(\lambda) < \rho$: 
\[\qrisk = \frac{bM+h\lambda-(b+h)\Gminus(\lambda)M}{(b+h)(1-\Gminus(\lambda))} \qquad , \qquad \risk = \frac{h\left(b-(b+h)\Gminus(\lambda)\right)(M-\lambda)}{(b+h)(1-\Gminus(\lambda))} \geq 0.\]
\end{enumerate}
\end{theorem}
}

\smallskip
For ease of notation, in the unidentifiable regime 
 we refer to the \minimaxquant \ $\qrisk$ as the \emph{\criticalquantile}, denoted as:
\begin{align}\label{eq:qcrit}
\qcrit_G = \frac{bM+h\lambda-(b+h)\Gminus(\lambda)M}{(b+h)(1-\Gminus(\lambda))}.
\end{align}

{The following result then emerges as a corollary of \Cref{thm:minimax-risk-identifiable} and \Cref{prop:iff}. 

\begin{corollary}\label{cor:risk-to-id}
A data-driven censored newsvendor problem with demand distribution $G$ and {observable boundary $\lambda$} is identifiable if and only if $\Gminus(\lambda) \geq \rho$, or if 
 $\Gminus(\lambda) < \rho$ and $M = \lambda$.
\end{corollary}

{In Appendix \ref{apx:lam-m} we establish that $\Gminus(\lambda) < \rho$ implies $\lambda \leq \qopt_G \leq M$.} Hence, for $\Gminus(\lambda) < \rho$ and $M = \lambda$ to simultaneously hold, it must be that $\qopt_G = \lambda = M$, a condition that we do not expect to hold in practice. {For instance, if $G(x)$ is continuous for all $x < \lambda$, it is not hard to see that $\Gminus(\lambda) < \rho$ implies $\qopt_G > \lambda$, and so $\lambda < M$.} We henceforth make the following mild assumption in the remainder of the paper. 

\begin{assumption}\label{asp:avoid-path}
Suppose $\Gminus(\lambda) < \rho$. Then, $\lambda < M$.
\end{assumption}

Under this assumption, we say that we are in the {\it identifiable regime} if $\Gminus(\lambda) \geq \rho$, and in the {\it unidentifiable regime} if $\Gminus(\lambda) < \rho$.\footnote{{We highlight that \Cref{asp:avoid-path} does not affect any of our results. It is simply introduced to make clear that the main salient driver of identifiability is the relationship between $\Gminus(\lambda)$ and $\rho$, rather than a spurious artifact of the maximum ordering quantity $M$.}}
}

We discuss the implications of \Cref{thm:minimax-risk-identifiable}, deferring a proof sketch to the end of the section. {At a high level, \Cref{thm:minimax-risk-identifiable} can be viewed as characterizing a ``spectrum of achievability'' for the data-driven censored newsvendor problem, based on the additional distributional information of $\Gminus(\lambda)$.} In particular, when $\lambda = +\infty$ and demand is bounded, we have that $\Gminus(\lambda) = 1$. \Cref{thm:minimax-risk-identifiable} then recovers the fact that the optimal
newsvendor cost $C_G(\qopt_G)$ is achievable with infinitely many samples in this case, and that the
problem is identifiable as a result. When $\lambda = 0$, on the other hand, $\Gminus(\lambda) = 0$ \Cref{thm:minimax-risk-identifiable} then
classifies the problem as unidentifiable, with $\risk = \frac{bhM}{b+h} > 0$. This recovers the intuition that, for any policy, there exists a worst-case distribution that forces it to incur constant loss.\footnote{This result subsumes \Cref{thm:minimax-risk-identifiable} of \citet{perakis2008regret}, for the special case where the interval $[\lambda,M]$ is the support of the demand distribution. As we will see at the end of this section, our analysis similarly uncovers the insight that the worst-case regret is achieved by a distribution that places weight on either $\lambda$ or $M$. For general $\lambda, M$, however, the arguments used to derive the minimax optimal ordering quantity and regret critically rely on our specific informational structure (i.e., that $\Gminus(\lambda)$ is fixed).}

\begin{figure}[t]
{
\begin{subfigure}[t]{.45\textwidth}
\centering
\includegraphics[width=\textwidth]{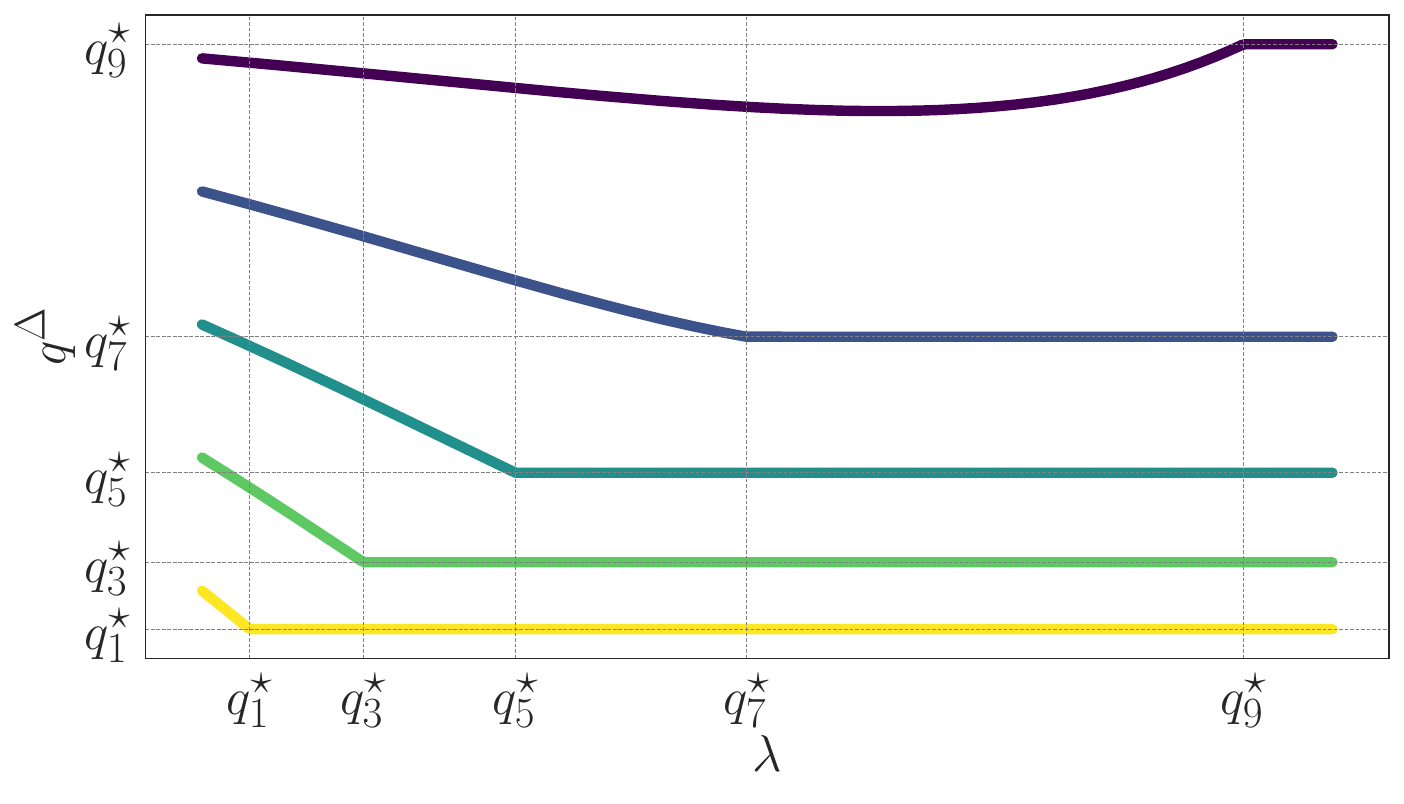}
\caption{\small $\qrisk$ vs. $\lambda$}
\label{fig:risk_a}
\end{subfigure}
\hfill
\begin{subfigure}[t]{.45\textwidth}
\centering
\includegraphics[width=\textwidth]{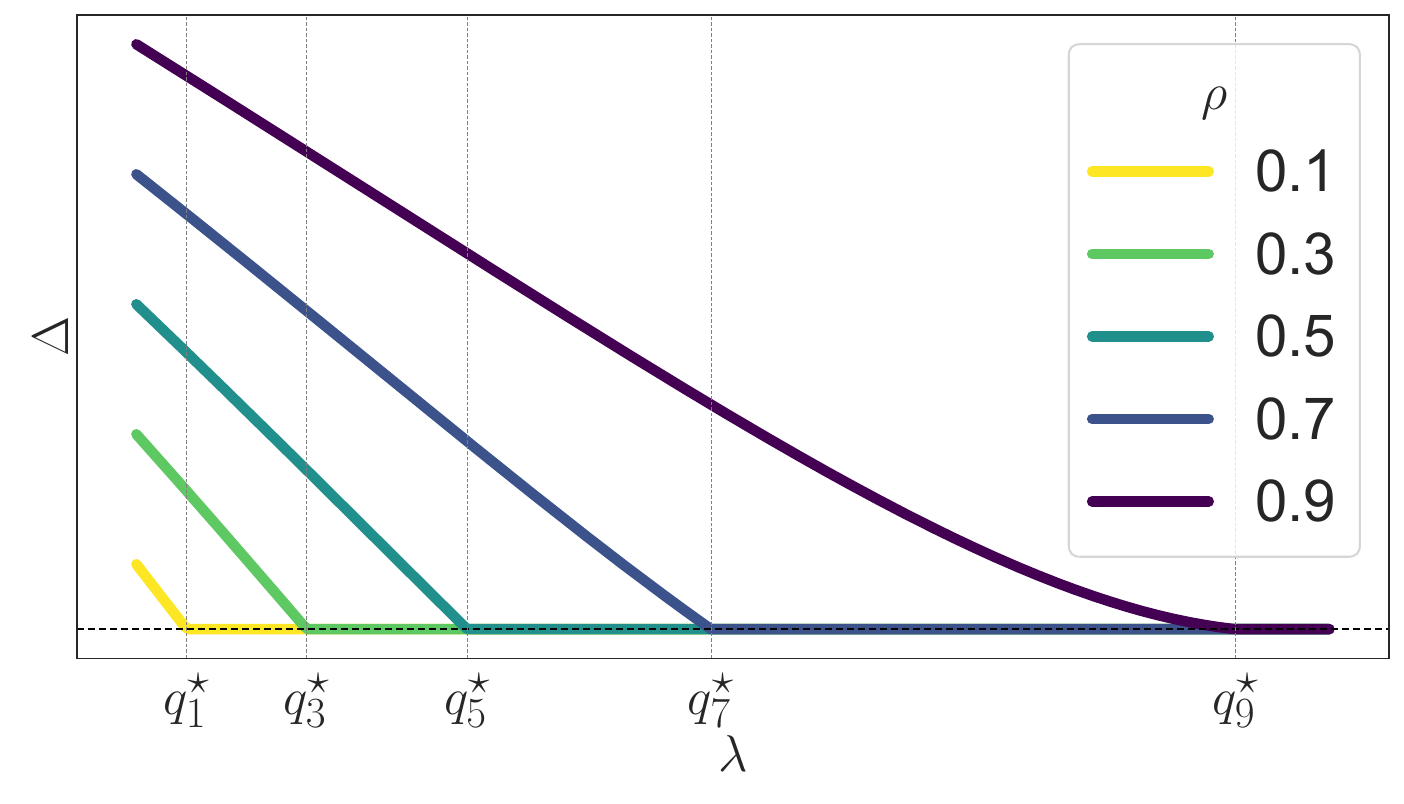}
\caption{\small $\risk$ vs. $\lambda$}
\label{fig:risk_b}
\end{subfigure}
}
\caption{Dependence of $\qrisk$ and $\risk$ on $\lambda$ for $D \sim \text{Exponential}(1/80)$, $\qbar = 200$, $h = 1$, and $\rho \in \{0.1,0.3,0.5,0.7,0.9\}$. We abuse notation and let $\qopt_{10\rho}$ denote the optimal newsvendor quantity associated with $\rho$. By \Cref{thm:minimax-risk-identifiable}, $\lambda \geq \qopt_{10\rho}$ corresponds to the identifiable regime.\label{fig:minimax_risk_figure}}
\end{figure}

{\Cref{fig:minimax_risk_figure} illustrates the dependence of the minimax risk $\risk$ and the minimax optimal ordering quantity $\qrisk$} on $\lambda$, as it interpolates between these two extremes. For $\Gminus(\lambda) < \rho$, even with infinitely many historical samples, for any policy that outputs $q^\pi$, there exists a distribution in the ambiguity set that can make it such that the policy over- (resp., under-) orders, causing constant regret. {This can be seen in \Cref{fig:risk_b}, where $\risk > 0$ for all $\lambda < \qopt_G$.}  As $\lambda$ increases from 0 (and $\Gminus(\lambda)$ increases as a result), $\risk$ strictly decreases. This reflects the phenomenon that, as the observable boundary increases, the size of the ambiguity set decreases, which causes a decrease in information loss. Once $\lambda$ reaches $\qopt_G$ (i.e., $\Gminus(\lambda)\geq \rho$), we have a phase transition, with $\risk = 0$. The problem becomes identifiable, with vanishing regret being achievable throughout this region, despite the fact that demand remains censored by the historical ordering quantities. Hence, \Cref{thm:minimax-risk-identifiable} highlights that, for a wide range of values of $\lambda$, demand censoring is not a barrier to effective decision-making.

We conclude our discussion with an analysis of the \minimaxquant \ $\qrisk$. Perhaps surprisingly, \Cref{thm:minimax-risk-identifiable} establishes that $\qrisk = \qopt_G$, for {\it any} value of $\Gminus(\lambda) \geq \rho$. {This can be seen in \Cref{fig:risk_a}, where $\qrisk$ plateaus for $\lambda \geq \qopt_G$.} So, not only is vanishing regret achievable, but in fact, nature is so limited in this regime that {\it any} distribution $F \in \ambig{\lambda}$ is such that $\qopt_F = \qopt_G$. When $\Gminus(\lambda) < \rho$, on the other hand, $\qrisk$ hedges between over- and under-ordering. This can best be seen when $\Gminus(\lambda) = 0$, in which case $\qrisk$ mixes between the maximum ordering quantity $M$ and the observable boundary $\lambda$ at a rate of $\rho$, by \Cref{thm:minimax-risk-identifiable}. This intuition is also reflected in the fact that, for fixed $\Gminus(\lambda)$, $\risk$ is increasing in $M-\lambda$, since hedging between over- and under-ordering becomes harder as unseen demand samples take on a larger set of possible values. Finally, \Cref{fig:risk_a} and \Cref{fig:risk_b} illustrate that both $\qrisk$ and $\risk$ are increasing in $b$, for fixed values of $h$ and $\lambda$, as under-ordering becomes costlier.

\smallskip 

{We now provide a proof sketch of \Cref{thm:minimax-risk-identifiable},} highlighting important auxiliary results upon which we rely in the rest of the paper. We defer a formal proof of the theorem to Appendix \ref{apx:minimax-risk-identifiable}.

\subsubsection*{Proof sketch.}
Consider first the case where $\Gminus(\lambda) \geq \rho$. \Cref{lem:identifiable_same_opt_quantile} first establishes that, for all \mbox{$F\in\ambig{\lambda}$}, the optimal newsvendor ordering quantity $\qopt_F$ coincides with $\qopt_G$. Its proof can be found in Appendix \ref{apx:identifiable_same_opt_quantile}.

\begin{lemma}\label{lem:identifiable_same_opt_quantile}
Suppose $\Gminus(\lambda) \geq \rho$. Then, for all $F \in \ambig{\lambda}$, $\qopt_F = \qopt_G < \lambda$.
\end{lemma}

Importantly, {\it this fact is independent of the decision-maker's ordering decision $q$}. Letting $q = \qopt_G$, this then implies that $C_F(q) - C_F(\qopt_F) = 0$, for all $F \in \ambig{\lambda}$. Hence, by definition of the minimax risk, $\risk = 0$, with $\qrisk = \qopt_G$. 

\smallskip

Consider now the case where $\Gminus(\lambda) < \rho$.  \Cref{lem:sup_expression_identifiability} is the main workhorse for our result: it provides closed-form expressions for the worst-case regret over $\ambig{\lambda}$, given $q$. Its proof can be found in Appendix \ref{apx:sup_expression_identifiability}.

\begin{lemma}\label{lem:sup_expression_identifiability} Suppose $\Gminus(\lambda) < \rho$. Then, for any $q \in [0,M]$,
    \begin{align}
    {\Regret(q)} & := \sup_{F \in \ambig{\lambda}} C_F(q) - C_F(\qopt_F) \notag 
   \\& = \begin{cases}
        b (M-q) + (b+h)\bigg[\E_G\Big[(q-D)\Ind{D \leq q}] - (M-D)\mathds{1}\{D < \lambda\}\Big]\bigg]  &\text{if }  q < \lambda, \\
         \left(b - (b+h)\Gminus(\lambda)\right)(M-q)&\text{if } q \in \left[\lambda,\qcrit_G \right] \\
        {h(q-\lambda)}  &\text{if } q > \qcrit_G.
    \end{cases} \notag 
    \end{align}
\end{lemma}

\smallskip

 At a high level, when $q < \lambda$ (i.e., the decision-maker orders no more than what she has historically observed), any worst-case demand distribution sets $\qopt_F  = M$, ensuring that $q$ incurs high underage costs. This is in line with the intuition that exploration is important under high censoring levels.

When $q \geq \lambda$, on the other hand, the worst-case demand distribution selects the {worst} of $\qopt_F = M$ (high underage costs) and $\qopt_F = \lambda$ (high overage costs). We show that the former is optimal for nature when $q \in [\lambda,\qcrit_G]$, while the latter is optimal for $q > \qcrit_G$. 

We finalize the argument by comparing the three regimes of $q$, and showing that the minimax {risk} is achieved at $q = \qcrit_G$. At this point,
\[(b-(b+h)\Gminus(\lambda))(M-q) = h(q-\lambda) = \frac{h(M-\lambda)(b-(b+h)\Gminus(\lambda))}{(b+h)(1-\Gminus(\lambda))}.\]
Therefore, $q$ is precisely set such that nature is indifferent between enforcing high underage and high overage costs, given this ordering quantity. The expression for $\risk$ follows from algebra. Finally, it is easy to argue that, when $\Gminus(\lambda) < \rho$, $\qcrit_G \in [\lambda,M]$ and $\risk \geq 0$.
\hfill\Halmos

\subsubsection*{Characterization of worst-case distributions and behavior of worst-case regret.}
The above analysis allows us to gain insight into the worst-case distribution that achieves the minimax risk $\risk$.
To formalize this, we define a distribution $F_p$ parametrized by $p \in \{0,1-\Gminus(\lambda)\}$:
\begin{equation*}
    F_p(x) = \begin{cases}
        G(x) & x < \lambda \\
        \Gminus(\lambda) + p & x \in [\lambda, M)  \\
        1 & x = M
    \end{cases}
\end{equation*}
By construction, $F_p \in \ambig{\lambda}$. Moreover, when $p = 0$, $F_p$ places the entirety of the remaining mass on $M$; when $p = 1-\Gminus(\lambda)$, on the other hand, the entirety of the remaining mass is placed on $\lambda$.  We define the {\em restricted ambiguity set} as the set of all valid distributions $F_p$, i.e.,:
\begin{equation*}
\B{\lambda} = \big\{F_p \mid p \in \left\{0,1 - \Gminus(\lambda)\right\} \big\}.
\end{equation*}
\Cref{prop:ber-is-worst-case} below establishes that $\B{\lambda}$ is indeed the worst-case family of distributions, for any ordering quantity $q$.

\begin{proposition}\label{prop:ber-is-worst-case}
Fix $q \in [0,M]$. Then,
\[\sup_{F\in\ambig{\lambda}}C_F(q) - C_F(\qopt_F) = \sup_{F\in\B{\lambda}}C_F(q)-C_F(\qopt_F).\]
\end{proposition}

We defer the proof of the fact to Appendix \ref{apx:bernoulli-is-worst}. \Cref{prop:ber-is-worst-case} provides useful intuition as to the drivers of regret, as previously discussed. Namely, it precisely reflects the intuition of nature adversarially generating a high- or low-demand distribution to ensure that the decision-maker incurs high underage or overage costs, respectively. \footnote{Similar insights have been derived in the uncensored setting (see, e.g., \citet{besbes2023big,perakis2008regret}).}

\smallskip

\begin{remark}\label{remark:well-sep}
\Cref{thm:minimax-risk-identifiable} and \Cref{prop:ber-is-worst-case} together help to establish why access to an upper bound $M$ on $\qopt_G$ is in fact necessary for the minimax risk $\risk$ to be bounded in the unidentifiable regime. In particular, Part 2 of \Cref{thm:minimax-risk-identifiable} implies that $\risk \to +\infty$ as $M\to+\infty$, achieved by a worst-case distribution in $\B{\lambda}$ setting $\qopt_F = M \to +\infty$.
\end{remark}

\begin{figure}[t]
\centering
{
\includegraphics[width=.75\textwidth]{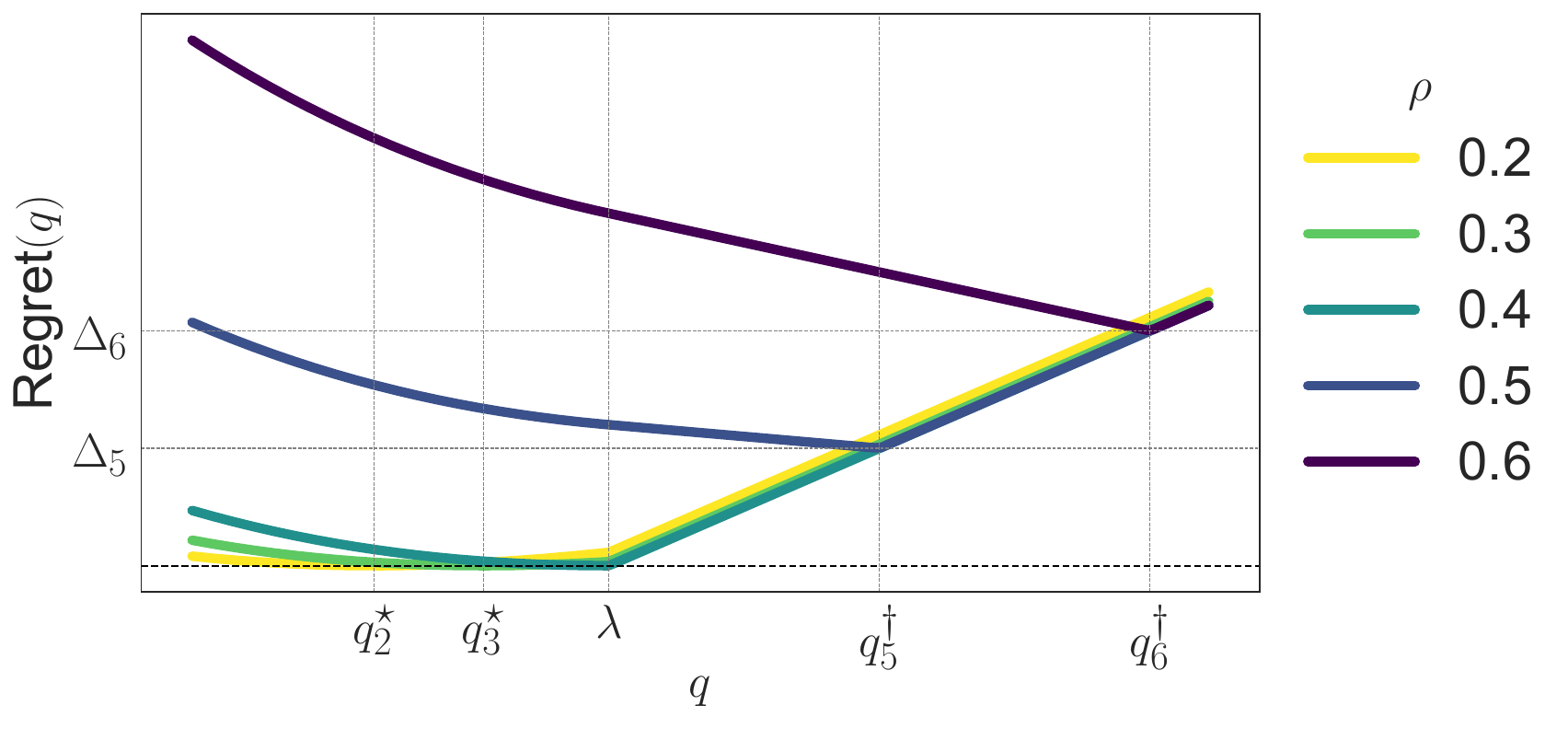}
}
\caption{{$\Regret(q)$ vs. $q$ for $D \sim \text{Exponential}(1/80)$, $\lambda = G^{-1}(0.4)$, $\qbar = 200$, $h = 1$, and $\rho \in \{0.2, 0.3, 0.4, 0.5, 0.6\}$. {We abuse notation and let $\qopt_{10\rho}$ and $\qcrit_{10\rho}$ respectively denote the optimal newsvendor and unidentifiable ordering quantities associated with $\rho$.} {Similarly, $\risk_{10\rho}$ denotes the value of $\risk$ under parameter $\rho$.}  When $\rho = 0.4$, $\Gminus(\lambda) = \rho$ by construction. In this case, we have $\qcrit_4 = \qopt_4 = \lambda$. {Finally, the bottom-most dashed line corresponds to $\Regret(q) = 0$.} 
 \label{fig:regret_sup}}}
{}
\end{figure}

\smallskip

We conclude the section by gaining insights into the structure of the worst-case regret, for a fixed ordering quantity $q$. We provide a plot of $\Regret(q)$ versus $q$ for the special case of exponentially distributed demand in \Cref{fig:regret_sup}. 
In the unidentifiable regime (i.e., $\rho \in \{0.5, 0.6\}$), this figure illustrates that the worst-case regret is decreasing and convex in $q$, for $q < \lambda$. It then linearly decreases for $q \in [\lambda,\qcrit_G]$, and increases for $q > \qcrit_G$. Observe in \Cref{fig:regret_sup} that, in the identifiable regime (i.e., $\rho \in \{0.2, 0.3, 0.4\}$), the worst-case regret is linear in $q$ for $q > \lambda$. This shows that, despite the fact that $\qrisk = \qopt_G$ in this regime, the worst-case regret is still not equivalent to the vanilla regret defined in \Cref{eq:uncensored-regret}, which is nonlinear in $q$. We formalize this in \Cref{prop:worst-case-for-identifiable} below, and defer its proof to Appendix \ref{apx:worst-case-id}.

\smallskip

\begin{proposition}\label{prop:worst-case-for-identifiable}
Suppose $\Gminus(\lambda) \geq \rho$. Then, for any $q \in [0,M]$,
    \begin{align*}
    \Regret(q) &:= \sup_{F \in \ambig{\lambda}} C_F(q) - C_F(\qopt_F) \\
    & = \begin{cases}
        {C_G(q)-C_G(\qopt_G)} \qquad &\text{if } q < \lambda \\
        b(\qopt_G-q) + (b+h)\bigg[(q-\lambda)+\E_G\Big[(\lambda-D)\Ind{D < \lambda} - (\qopt_G - D)\Ind{D \leq \qopt_G}\Big]\bigg] \quad &\text{if } q \geq \lambda.
    \end{cases}
    \end{align*}
\end{proposition}

\smallskip

\Cref{prop:worst-case-for-identifiable} establishes that, while $\Regret(q)$ is indeed equivalent to the vanilla regret when $q < \lambda$, setting $q \geq \lambda$ allows for a distribution in the ambiguity set to kick in, resulting in an increase in cost. In the proof of \Cref{prop:worst-case-for-identifiable} we argue that this worst-case distribution ensures that demand comes in low (specifically, with an atom at $\lambda$), implying that $q$ over-orders in the worst case. In this case, the worst-case regret is linear in $q$, as observed in \Cref{fig:regret_sup}.

\smallskip

In \Cref{sec:upper-bound} we will see that the exact characterization of worst-case regret across all regimes of identifiability, as well as the closed-form expression for $\qrisk$, will allow us to derive a {simple algorithm that achieves the minimax risk $\risk$}.\footnote{This lies in contrast to the pricing problem considered in \citet{bu2023offline}, where a derivation of the minimax optimal price remains out of reach, contributing to a more complicated algorithm that instead optimizes optimistic and pessimistic bounds on the worst-case regret.}

\smallskip

\section{A Near-Optimal Algorithm}\label{sec:upper-bound}

As discussed in \Cref{sec:minimax-risk}, our goal is to design an algorithm whose regret asymptotically achieves the minimax risk $\risk$, {\it across both regimes of identifiability}. \Cref{thm:minimax-risk-identifiable}, however, established that even determining identifiability requires knowledge of $\Gminus(\lambda)$, which the decision-maker does not have. Therefore, performing well across all regimes seems to be a challenging task. In this section, we design a robust algorithm that achieves the optimal order of regret (up to polylogarithmic factors) in both the identifiable and unidentifiable regimes, in spite of this challenge.

\subsection{Algorithm Description}\label{sec:alg-desc}

At a high level, our algorithm, \textsf{Robust Censored Newsvendor} (\ALG), proceeds in {a} hierarchical fashion. Motivated by the above challenge, we first test for identifiability, i.e., whether $\Gminus(\lambda) \geq \rho$. While doing so would be trivial in the uncensored setting, recall that our algorithm only has access to {\it sales}, rather than true demand data. The following fact, however, will allow us to design a test that correctly determines whether $\Gminus(\lambda) \geq \rho$, with high probability.

\begin{fact}\label{key-fact}
Let $\lambdasamples = \left(\soff_{Ki} \ \mid \ i \in [N]\right)$. For all $i \in [N]$, \[\Ind{\lambdasamples_i < \lambda} = \Ind{\min\{\lambda,\doff_i\} < \lambda}= \Ind{\doff_i < \lambda}.\]
\end{fact}
\Cref{key-fact} therefore implies that we can compute an unbiased estimate of $\Gminus(\lambda)$ by restricting our attention to the samples at the boundary. Thus motivated, our algorithm computes the fraction of samples that lie strictly below $\lambda$, i.e.,
\[\Gminushat(\lambda) = \frac{1}{N} \sum_{i \in [N]} \Ind{\lambdasamples_i < \lambda}.
\]
We refer to this quantity as the {\it censored} estimate of $\Gminus(\lambda)$ since it uses the censored sales data, rather than the true (unobserved) demand data that would be required to compute the SAA of $\Gminus(\lambda)$.

Once computed, the censored estimate $\Gminushat(\lambda)$ may fall into one of three cases, depending on {the} confidence parameter $\width$ that our algorithm takes as input. In particular, if $\Gminushat(\lambda) \geq \rho + \width$, our algorithm classifies the problem as {\it likely identifiable}. In this case, motivated by \Cref{thm:minimax-risk-identifiable}, it outputs a censored estimate of $\qopt_G$, denoted by $\qopt_{\offeCDF}$, defined as the $\rho$-th sample quantile of $\lambdasamples$. Formally:
\[\qopt_{\offeCDF} = \inf\left\{x \mid \frac1N\sum_{i\in[N]} \Ind{\lambdasamples_i \leq x} \geq \rho\right\}.\]
For ease of notation, we let 
$\offeCDF(x) = \frac1N\sum_{i\in[N]} \Ind{\lambdasamples_i \leq x}.$ Note again here that $\qopt_{\offeCDF}$ relies only on the samples for which the historical ordering quantity was $\lambda$, as was the case for $\Gminushat(\lambda)$. 

If, on the other hand $\Gminushat(\lambda) < \rho-\width$, the algorithm classifies the problem as {\it likely unidentifiable}, and outputs a censored estimate of $\qcrit_G$, denoted by $\qcrithat$ and defined as:
\[\qcrithat = \frac{bM+h\lambda-(b+h)\Gminushat(\lambda)M}{(b+h)(1-\Gminushat(\lambda))}.\]

Finally, if $\Gminushat(\lambda) \in [\rho-\width,\rho+\width)$, our algorithm cannot determine problem identifiability with confidence, in which case it outputs $\lambda$. We present a formal description of our algorithm in \Cref{alg:newsvendor}.

\begin{remark}
Our algorithm’s first phase can be interpreted through the lens of hypothesis testing, as it seeks to determine whether or not $\Gminus(\lambda) \leq \rho$. Our test, which uses the empirical estimate $\Gminushat(\lambda)$, can be shown to implement the uniformly most powerful test for this hypothesis~\citep{casella2024statistical}.
\end{remark}

\begin{algorithm}[!t]
\DontPrintSemicolon 
\KwIn{Observable boundary $\lambda$, censored demand samples $\lambdasamples$,  confidence term $\width$}
\KwOut{Ordering quantity $\qalg$}
Compute censored SAA of $\Gminus(\lambda)$:
\begin{equation}
\label{eq:g_hat}
\Gminushat(\lambda) = \frac{1}{N} \sum_{i \in [N]} \Ind{\lambdasamples_i < \lambda}.
\end{equation}
\uIf(\tcp*[h]{Likely identifiable}){$\Gminushat(\lambda) \geq \rho+\width$}{
Compute censored SAA of $\qopt_G$:
\begin{align}\label{eq:saa}
\qopt_{\offeCDF} = \inf \left\{x \mid \offeCDF(x) \geq \rho \right\},
\end{align}
where  $\offeCDF(x) = \frac{1}{N}\sum_{i \in [N]} \Ind{\lambdasamples_i \leq x}.$ Let $\qalg = \qopt_{\offeCDF}$.  
}
\uElseIf(\tcp*[h]{Likely unidentifiable}){$\Gminushat(\lambda) < \rho-\width$}{
Compute empirical estimate of $\qcrit_G$:
\begin{align}\label{eq:qcrithat}
\qcrithat= \frac{bM+h\lambda-(b+h)\Gminushat(\lambda)M}{(b+h)(1-\Gminushat(\lambda))}.
\end{align}
Let $\qalg = \qcrithat$.
}
\uElse(\tcp*[h]{Knife-edge case}){Let $\qalg = \lambda$.
}
\Return{$\qalg$}
	\caption{\textsf{Robust Censored Newsvendor} (\ALG)}
	\label{alg:newsvendor}
\end{algorithm}

\subsection{Algorithm Analysis}

\Cref{thm:minmax_regret} establishes that this intuitive algorithm achieves vanishing minimax regret relative to $\risk$ with constant probability, for an appropriately tuned confidence parameter $\width$.

\smallskip

\begin{theorem}
\label{thm:minmax_regret}
Fix $\delta \in (0,1)$, and let $\width = \sqrt{\frac{\log(2 / \delta)}{2N}}$. With probability at least $1 - 2 \delta$, \Cref{alg:newsvendor} outputs an ordering quantity $\qalg$ such that:
\begin{enumerate}[(i)]
\item if $\Gminus(\lambda) \geq \rho + 2 \width$:
\[ \Regret(\qalg) \leq \risk  + \lambda(b+h)\width =  \lambda(b+h)\width \]
\item if $\Gminus(\lambda) \in [\rho, \rho + 2 \width)$:
\[\Regret(\qalg) \leq \risk + 2\lambda(b+h)\width = 2\lambda(b+h)\width \]
\item if $\Gminus(\lambda) \in [\rho - 2 \width, \rho)$:
\[\Regret(\qalg) \leq \risk + {2 \max\left\{\frac{b}{h}, 1\right\}(M - \lambda)}(b+h) \width\]
\item if $\Gminus(\lambda) < \rho - 2 \width$:
\[\Regret(\qalg) \leq \risk + {\max\left\{\frac{b}{h},1\right\}(M-\lambda)}(b+h)\width.\]
\end{enumerate}
\end{theorem}

\smallskip

Letting $\delta = O(1/\sqrt{N})$, we obtain the following bound on the expected worst-case regret of \Cref{alg:newsvendor}.
\begin{corollary}\label{cor:exp-regret}
Let $\delta = c/\sqrt{N}$, for some constant $c > 0$. Then, there exists a constant $c' > 0$ (independent of $M, \lambda, \rho$, and $N$) such that:
\[\E_G\left[\Regret(\qalg)\right] \leq \begin{cases}c'(b+h)\left(\lambda + \max\left\{\rho\lambda,(1-\rho)M\right\}\right) \sqrt{\log N /N} \qquad &\text{if } \Gminus(\lambda) \geq \rho\\ 
\Delta + c'(b+h)\Bigg(\max\left\{\frac{\rho}{1-\rho},1\right\}(M-\lambda) + \max\{\rho M, (1-\rho)(M-\lambda)\}\Bigg)\sqrt{\log N /N} \quad &\text{if } \Gminus(\lambda) < \rho.
\end{cases}
\]
\end{corollary}

\smallskip

We make a few observations on our algorithm's guarantees before providing a proof sketch of the theorem at the end of this section. (We defer a proof of \Cref{cor:exp-regret} to Appendix \ref{apx:exp-regret}.) First of all, in the identifiable regime ($\Gminus(\lambda) \geq \rho$), since $\qrisk = \qopt_G$ by \Cref{thm:minimax-risk-identifiable}, our algorithm recovers existing regret guarantees (up to constant factors) for the uncensored setting \citep{chen2024survey}. We emphasize that this result holds for {\it any} value of $\lambda$ such that $\Gminus(\lambda) \geq \rho$, and for {\it any} set of historical ordering quantities. This further implies that demand censoring is effectively immaterial in the identifiable regime. 

{In the unidentifiable regime ($\Gminus(\lambda) < \rho$), it is also near-optimal, incurring an expected worst-case regret of $\Delta + O\left(\sqrt{\frac{\log N}{N}}\right)$}. Observe that our upper bound on the algorithm's regret in this regime is linear in $M-\lambda$. This is in line with the construction of the worst-case distribution described in \Cref{sec:minimax-risk}, for any quantity $q$: if $q$ is close to $\lambda$, the worst-case distribution $F\in\ambig{\lambda}$ sets the optimal newsvendor quantity $\qopt_F$ to $M$; if $q$ is close to $M$, on the other hand, $F$ sets $\qopt_F$ to $\lambda$. As the gap between $\lambda$ and $M$ increases, it becomes more difficult for the algorithm to hedge against nature, which is reflected in the regret guarantee.

We now provide a sketch of the proof of \Cref{thm:minmax_regret}, deferring a formal proof to Appendix \ref{apx:ub_proof}.

\subsubsection*{Proof sketch.}
We first upper bound our algorithm's regret by two sources of loss: (i) the minimax risk $\risk$, and (ii) the cost difference between $\qalg$ and $\qrisk$, for any distribution $F \in \ambig{\lambda}$. Namely, we have:
\begin{align}\label{eq:regret-decomp-sketch}
\Regret(\qalg) \leq \Delta + \sup_{F \in \ambig{\lambda}} C_F(\qalg) - C_F(\qrisk).
\end{align}

Noting that $\risk$ is independent of $\qalg$, it suffices to upper bound the second term. Recall, by \Cref{thm:minimax-risk-identifiable}, the value of $\qrisk$ depends on whether or not the problem is identifiable. Hence, \Cref{eq:regret-decomp-sketch} highlights that the algorithm incurs regret if: (i) it incorrectly predicts identifiability of the underlying problem (we refer to this as an {\it identifiability error}), in which case $\qalg$ and $\qrisk$ may be far apart, or (ii) it correctly predicts identifiability, but incurs some finite-sample estimation error of $\qrisk$. 

We begin by tackling the first source of loss. At a high level, identifiability errors should be rare if $\Gminushat(\lambda)$ and $\Gminus(\lambda)$ are ``close enough,'' with constant probability. To show this, we leverage \Cref{key-fact}, which implies that {$\Gminushat(\lambda)$ is an unbiased estimate of $\Gminus(\lambda)$}, thereby allowing us to leverage standard concentration arguments to show the following probabilistic bound on the distance between these two quantities:
\[\Pr_G\left(\abs{\Gminushat(\lambda) - \Gminus(\lambda)} {<} \width\right) \geq 1-\delta.\]

Restricting our attention to {the} event $\Event =  \left\{ \abs{\Gminushat(\lambda) - \Gminus(\lambda)} {<} \width \right\}$, then, we conclude that our algorithm correctly predicts identifiability in two regions: (i) $\Gminus(\lambda) \geq \rho+2\width$, and (ii) \mbox{$\Gminus(\lambda) < \rho-2\width$}. 

In the first case, $\qrisk = \qopt_G$ by \Cref{thm:minimax-risk-identifiable}, and our algorithm outputs $\qalg = \qopt_{\offeCDF}$. Hence, by \Cref{eq:regret-decomp-sketch}, it suffices to bound the loss incurred by setting $\qopt_{\offeCDF}$ instead of $\qopt_G$.  In the uncensored setting, standard arguments leverage concentration of the empirical cdf (using {\it uncensored demand} samples) to show that the SAA algorithm produces an estimate of $\qopt_G$ with a loss of $O(1/\sqrt{N})$. However, by \Cref{eq:saa}, $\qopt_{\offeCDF}$ is computed using the empirical cdf $\offeCDF$ of the {censored} sales data which, as discussed above, need not in general converge to $G$. We nonetheless show that when $\Gminus(\lambda) \geq \rho$, $\qopt_{\offeCDF} = \qopt_{\offeCDFtrue}$, where \emph{$\qopt_{\offeCDFtrue}$ is the SAA of $\qopt_G$ with access to the true (uncensored) demand samples!} Again, this is crucially due to the fact that our algorithm exclusively uses samples such that $\qoff_k = \lambda$ in the construction of $\qopt_{\offeCDF}$. Having established this fact, we are able to leverage concentration of the (true) empirical cdf to establish $O(1/\sqrt{N})$ regret for this case.

In the second case, $\qrisk = \qcrit_G$ by \Cref{thm:minimax-risk-identifiable}, and $\qalg = \qcrithat$. Noticing that $\qcrit_G = \func(\Gminus(\lambda))$ and  \mbox{$\qcrithat = \func(\Gminushat(\lambda))$}, where $\func(x) = \frac{bM+h\lambda-(b+h)xM}{(b+h)(1-x)}$, we conclude that the error in this case stems from the gap between $\Gminushat(\lambda)$ and $\Gminus(\lambda)$. Using the fact that $\abs{\Gminushat(\lambda)-\Gminus(\lambda)} < \width$, under event $\Event$, we obtain the $O(1/\sqrt{N})$ regret bound. 

In both of the above regions, our algorithm suffers only from finite-sample estimation error, as it correctly predicts (un)identifiability. When $\Gminus(\lambda) \in [\rho-2\width, \rho+2\width)$, however, $\Gminushat(\lambda)$ may lie in the knife-edge region $[\rho-\width,\rho+\width)$, in which case our algorithm defaults to $\qalg = \lambda$. Despite this, we argue that $\lambda$ is a ``safe'' choice for both regimes in this case. To see why $\lambda$ hedges against both possibilities, observe that, if $\Gminus(\lambda) \in [\rho-2\width,\rho+2\width)$, $\lim_{N\to\infty}\Gminus(\lambda) = \rho$, which implies that $\qopt_G$ approaches $\lambda$ as $N$ grows large. For the same reason, one can argue that $\qcrit_G$ approaches $\lambda$ as $N$ grows large. Hence, $\lambda$ is expected to be close to $\qrisk$ (formally, within $O(1/\sqrt{N})$ of $\qrisk$), whether or not the problem is identifiable. 

Putting these four cases together, we obtain the theorem.
\hfill\Halmos

\smallskip

We conclude the section by reframing \Cref{thm:minmax_regret} in terms of sample complexity, which ties the problem back to the question of identifiability. (Recall, the proof of the fact that $\risk = 0$ implies identifiability in \Cref{prop:iff} relied on the ability to provide a policy $\pi$ satisfying \Cref{def:identifiability}.)

\smallskip 

\begin{corollary}\label{cor:sample-complex}
For any $\epsilon > 0$ and $\delta \in (0,1)$, if $N \geq N(\epsilon,\delta)$, then with probability at least $1-2\delta$, the worst-case regret of \Cref{alg:newsvendor} is no greater than $\risk + \epsilon$, i.e.,
\[\Pr_G\left[ \sup_{F \in \ambig{\lambda}} \Big\{ C_F(\qalg) - C_F(\qopt_F)-\risk\Big\} < \epsilon \right] \geq 1-2\delta,
\]
where $N(\epsilon,\delta) =\frac{2(b+h)^2\log(2/\delta)}{\epsilon^2}\max\left\{\lambda^2,\max\{(b/h)^2,1\}(M-\lambda)^2\right\}$.
\end{corollary}

\smallskip 

Letting $\delta = O(1/\sqrt{N})$, \Cref{cor:sample-complex} implies that, when $\Gminus(\lambda) \geq \rho$, 
\[\lim_{N\to\infty}\Pr_G\left[ \sup_{F \in \ambig{\lambda}} \Big\{ C_F(\qalg) - C_F(\qopt_F)\Big\} < \epsilon \right] = 1,\]
which gives us identifiability, as promised in the proof of \Cref{prop:iff}.

\minedit{
\begin{remark}
A particularly appealing feature of the \ALG framework is its flexibility. In particular, one can first perform the identifiability test and, conditional on the test indicating that the instance is likely identifiable, output {\it any} estimate of $\qopt_G$. For instance, instead of the censored SAA $\qopt_{\widehat{G}}$ defined in \Cref{eq:saa}, one can output the $\rho$-quantile of the well-known Kaplan-Meier estimator of the cumulative distribution function \citep{kaplan1958nonparametric}.\footnote{\minedit{We refer the reader to \Cref{tab:rkm_comparison} in {Appendix \ref{app:sims_robust_km}} for a numerical evaluation of this policy.}}
\end{remark}
}
\section{Lower Bound}\label{sec:lower-bound}

In this section we show that the bounds established in \Cref{thm:minmax_regret} are tight, i.e., that any algorithm necessarily incurs {regret} lower bounded by $\risk + \Omega(1/\sqrt{N})$, in expectation. {While existing lower bounds apply to our problem (i.e., by taking $\lambda = +\infty$), we provide {\it instance-dependent} lower bounds, as a function of the regime of identifiability, which we formally define below.}

\smallskip

\begin{definition}[Identifiability regimes]
The three identifiability regimes of the {data-driven censored newsvendor problem} are defined as follows:
\begin{enumerate}[$(i)$]
\item the \emph{strictly identifiable} regime: the set of all distributions $G$ such that
\[\Gminus(\lambda) > \rho + \frac{c}{\sqrt{N}} \quad \text{for some } \ c > 0.\]
We let $\GI \subset \mathcal{G}$ denote the set of all such distributions.
\item the \emph{knife-edge} regime: the set of all distributions $G$ such that
\[\abs{\Gminus(\lambda)-\rho} \leq \frac{c}{\sqrt{N}} \quad \text{for some } c > 0.\]
We let $\GKE \subset \mathcal{G}$ denote the set of all such distributions.
\item the \emph{strictly unidentifiable} regime: the set of all distributions $G$ such that \[\Gminus(\lambda) < \rho-\frac{c}{\sqrt{N}} \quad \text{for some } \ c > 0.\]
We let $\GUI \subset \mathcal{G}$ denote the set of all such distributions.
\end{enumerate}
\end{definition}

\medskip

\Cref{thm:lb} establishes that our algorithm achieves the optimal regret guarantee (up to polylogarithmic factors) {across the three identifiability regimes}.

\begin{theorem}\label{thm:lb}
Fix $b, h, \lambda$, and $M$, {with $\lambda < M$.} For any policy $\pi \in \Pi$ with access to $N$ censored demand samples $\lambdasamples$:
\begin{enumerate}[(i)]
\item in the \emph{strictly identifiable} regime:
\begin{align*}
\sup_{G \in \GI}\E_{G}\bigg[\Regret(q^\pi)-\Delta\bigg] \geq  \frac{\lambda(b+h)\sqrt{1-\rho}\min\{\rho,1-\rho\}e^{-1/2}}{64\sqrt{N}}
\end{align*}
\item in the \emph{knife-edge} regime:
\begin{align*}
\sup_{G \in \GKE}\E_{G}\bigg[\Regret(q^\pi)-\Delta\bigg] \geq \frac{\lambda(b+h)\sqrt{1-\rho}\min\{\rho,1-\rho\}e^{-1/2}}{32\sqrt{N}}
\end{align*}
\item in the \emph{strictly unidentifiable} regime:
\begin{align*}
\sup_{G \in \GUI}\E_{G}\bigg[\Regret(q^\pi)-\Delta\bigg] \geq \frac{h(M-\lambda)\sqrt{1-\rho}\min\{\rho,1-\rho\}\min\left\{\rho,{3\rho-1}\right\}e^{-1/2}}{64\sqrt{N}}
\end{align*}
\end{enumerate}
\end{theorem}

\medskip  

In the strictly identifiable regime, our bound reflects that the censored setting suffers from the same finite-sample information loss as the uncensored setting (up to constant factors) \citep{chen2024survey}. Moreover, our lower bounds have the same dependence on $\lambda$ and $M-\lambda$, respectively, for the strictly identifiable and unidentifiable regimes, as the upper bounds derived in \Cref{thm:minmax_regret}. Additionally, the lower bound in the knife-edge regime is exactly twice the lower bound for the strictly identifiable regime. Such a phenomenon also appears in the upper bound, and highlights an additional unavoidable source of loss when $\Gminus(\lambda) \approx \rho$, the regime in which identifiability is the hardest to test. 

{Our proof follows similar lines as those used to derive lower bounds for the uncensored setting (see, e.g., \citet{chen2024survey,lin2022data}).} Specifically, for each regime, we identify two ``hard'' distributions with support on two point masses. While the means of these two distributions are within $\Theta(1/\sqrt{N})$ of each other, their respective optimal ordering quantities lie on opposite ends of their support. We then rely on the closed-form characterization of the worst-case regret derived in \Cref{lem:sup_expression_identifiability} to argue that, for these two distributions, the regret incurred by any policy is proportional to the estimation error of the minimax optimal order quantity $\qrisk$. We conclude by showing that the problem of estimating $\qrisk$ reduces to the problem of testing for the identity of the demand distribution, for which the error rate is lower bounded by $\Omega\left(1/\sqrt{N}\right)$. {We refer the reader to \cref{app:lb_proof} for the formal proof of this result.}

\section{Computational Experiments}
\label{sec:experiments}

In this section we demonstrate the practical efficacy of our algorithm via extensive computational experiments.  We test all algorithms on synthetic data in \cref{sec:experiments_synthetic}, and consider a real-world retail dataset in \cref{sec:exp_real_data}.


\subsection{Description of Metrics and Benchmark Policies}\label{sec:experiments_benchmarks}

{We evaluate our algorithm, \ALG, using $\delta = 0.3$, and benchmark it} against {five} other policies from the literature, for all sets of experiments:
\begin{itemize}
    \item {\IgnorantSAA}: This algorithm ignores the potential impact of demand censoring and outputs the $\rho$-th sample quantile, using all samples in the dataset, i.e.:
    \[
    q^{\textsf{naive}} = \inf\left\{ q \ \big{|} \ \frac{1}{\sum_k N_k} \sum_k \sum_{i \in [N_k]} \Ind{\soff_{ki} \leq q} \geq \rho\right\}.
    \]
    \item {\SubsampleSAA}: This algorithm subsamples the dataset to only retain uncensored samples $(k,i)$ such that $\doff_{ki} < \qoff_k$, and outputs the $\rho$-th sample quantile for this subset, i.e.:
    \[
    q^{\textsf{subsample}} = \inf\left\{ q \ \big{|} \ \frac{1}{|\{(k,i)\ \mid \ \doff_{ki} < \qoff_k\}|} \sum_k \sum_{i : \doff_{ki} < \qoff_k} \Ind{\doff_{ki} \leq q} \geq \rho\right\}.
    \]
    If $|\{(k,i)\ \mid \ \doff_{ki} < \qoff_k\}| = 0$, the algorithm outputs $\lambda$.
    \item {\KM}: {This algorithm outputs the $\rho$-quantile of the Kaplan-Meier (KM) estimator of the cdf}~\citep{kaplan1958nonparametric}, using the \textsf{lifelines} package~\citep{davidson_pilon_2024_14007206}.
    \item {\CensoredSAA}: This is the algorithm designed in \citet{fan2022sample} for the identifiable regime, when historical inventory levels are i.i.d. samples from a given distribution. 
    Formally, their algorithm is parameterized by an accuracy parameter $\epsilon > 0$. It then successively computes the following $K$ empirical cdf's, for $k = 1, \ldots,K$:
    \[\hat{F}_-^k(x) = \frac{1}{\sum_{k' \geq k}N_{k'}}\sum_{k'\geq k}\sum_{i\in[N_{k'}]} \Ind{\soff_{k'i}<x},\]
    and terminates at the first value of $k$ for which $\hat{F}_-^k(\qoff_k) \geq \rho - 2\beta$, where $\beta = \frac{\min\{b,h\}}{18(b+h)}\cdot\epsilon$. For this value of $k$, it outputs $q^\CensoredSAA = \inf\left\{q \mid \hat{F}_-^k(q) \geq \rho-2\beta\right\}$.
    If the termination condition is never satisfied, it simply outputs $\lambda$. {We evaluate the algorithm with $\epsilon = 1 / \sqrt{N}$.}
    \item {\TrueSAA}: This is the classical SAA heuristic, which outputs the $\rho$-th sample quantile with respect to the {\it true} demand data, i.e.:
    \[q^{\textsf{SAA}} = \inf\left\{ q \ \big{|} \ \frac{1}{\sum_k N_k} \sum_{k} \sum_{i \in [N_k]} \Ind{\doff_{ki} \leq q} \geq \rho\right\}.\]
    Though this algorithm is not implementable, it serves as a useful benchmark in the identifiable regime (where $\qopt_G$ is learnable), to numerically isolate the impact of demand censoring.
\end{itemize}

\smallskip

In \cref{sec:exp_real_data} we consider a practically motivated variant of \ALG, referred to as \ALGplus, which improves the sample efficiency of \ALG by incorporating {\em all} sales data.

\medskip 

We report the following metrics, depending on the regime of identifiability:
\begin{itemize}
    \item {\em Unidentifiable regime}: When $\Gminus(\lambda) < \rho$, we are interested in all policies' performance relative to the unidentifiable ordering quantity $\qcrit_G$, which is minimax optimal in this regime. Specifically, we report $\Regret(q^\pi)-\risk$, and its relative counterpart:
    \begin{align}\label{eq:rel-ui}
     \mathcal{R}^{{ui}}(q^\pi) = \frac{\Regret(q^\pi) - \risk}{\risk}.
    \end{align}
    In this case, reporting $\Regret(q^\pi)-\risk$ instead of $\Regret(q^\pi)$ allows us to set aside the unavoidable loss that {\it any} policy incurs in the unidentifiable regime. This metric therefore allows us to quantify various policies' ability to learn the minimax optimal ordering quantity.
    \item {\em Identifiable regime}: When $\Gminus(\lambda) \geq \rho$, we are interested {in} all policies' performance relative to the optimal newsvendor quantity $\qopt_G$. Specifically, we report the vanilla regret $C_G(q^\pi)-C_G(\qopt_G)$, and its relative counterpart:
    \begin{align}\label{eq:rel-id}
    \mathcal{R}^{id}(q^\pi) = \frac{C_G(q^\pi) - C_G(\qopt_G)}{C_G(\qopt_G)}.
    \end{align}
\end{itemize}

{
\begin{remark}
Reporting different metrics depending on the identifiability regime allows us to investigate the robustness of various policies, with the desideratum being low regret both when the optimal newsvendor cost is a meaningful benchmark, as well as when it's not. 
\end{remark}
}

\subsection{Synthetic Experiments}
\label{sec:experiments_synthetic}

In this section we investigate how our policy's performance is impacted as we vary (i) the number of samples $N$, (ii) the observable boundary $\lambda$, and (iii) demand variability.

\subsubsection{Experimental setup.} \label{sec:experiment_setup}
Across all experiments, we let $h = 1$ and vary $b \in \{3,9,49\}$, such that $\rho \in \{0.75, 0.9, 0.98\}$.\footnote{In most practical settings we expect the per-unit underage cost to be significantly higher than the overage cost.}
We moreover consider the following demand distributions:
\begin{itemize}
    \item {\em Uniform}: Demand is drawn from a discrete uniform distribution with support $\{0,1,\ldots,100\}$.
    \item {\em Exponential}: Demand is exponentially distributed with mean 80.
    \item {\em Poisson}: Demand is drawn from a Poisson distribution with mean 80.
    \item {\em Normal}: Demand $D = \max\{0, X\}$ where $X \in \mathcal{N}(80, \sigma)$ is a Gaussian random variable with mean $80$ and {standard deviation} $\sigma \in \{20,25,...,40\}$.
\end{itemize}
We let $M = 320$, an upper bound on $\qopt_G$ across all of the above distributions. 

We consider $K = 2$ historical ordering quantities. 
To generate these, we first fix the observable boundary $\lambda$, and define the larger ordering quantity to be $\qoff_2 = \lambda$. For each of the above instances, we sample $N$ demand samples from $G$, and censor them at $\lambda$. We generate the other historical ordering quantity by sampling $\qoff_1 \sim \text{Uniform}[\frac{\lambda}{4}, \frac{3 \lambda}{4}]$. Once we have fixed $\qoff_1$, we sample $N$ true demand samples from $G$, and similarly censor them at $\qoff_1$.
For each distribution, we vary $\lambda \in \left\{\left(\frac12+\frac{k}{7}\right)\qopt_G, k \in [7]\right\}$ (i.e., eight equally spaced points between $\frac{\qopt_G}{2}$ and $\frac{3\qopt_G}{2}$), where $\qopt_G$ is the optimal newsvendor quantity under $\rho = 0.9$. Varying $\lambda$ in this way allows us to generate enough instances for each regime of identifiability. In particular, we will see that our algorithm's performance depends on whether the problem is {\it easily identifiable} ($\Gminus(\lambda) \gg \rho$), {\it easily unidentifiable} ($\Gminus(\lambda) \ll \rho$), or in a {\it difficult regime} ($\Gminus(\lambda) \approx \rho$).

\subsubsection*{Algorithm evaluation.} We run 100 replications for each experiment. \minedit{For each replication, we sample the offline dataset $\soff$ according to the above procedure.  We then compute the relevant regret metric (where $C_G(\cdot)$ is computed via Monte Carlo estimation over $10^7$ samples from the true distribution $G$), and report the average of each metric across all 100 replications.}

\subsubsection{Impact of {$N$}.} 

\begin{figure}[!t]
    \begin{subfigure}{.48\textwidth}
        \centering
        \includegraphics[width=\textwidth]{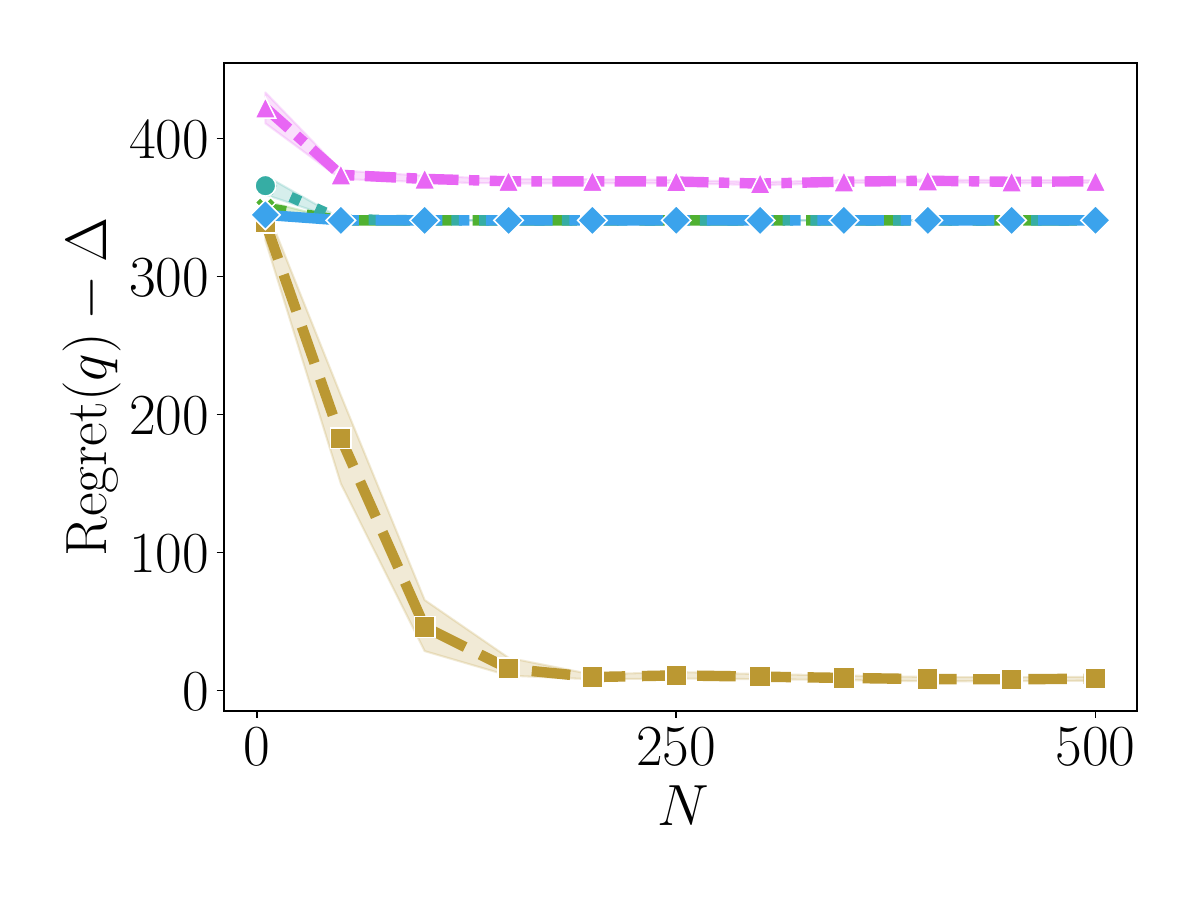}
        \caption{$\lambda = 69.93$ (Unidentifiable)}\label{fig:na}
    \end{subfigure}
     \begin{subfigure}{.48\textwidth}
        \centering
        \includegraphics[width=\textwidth]{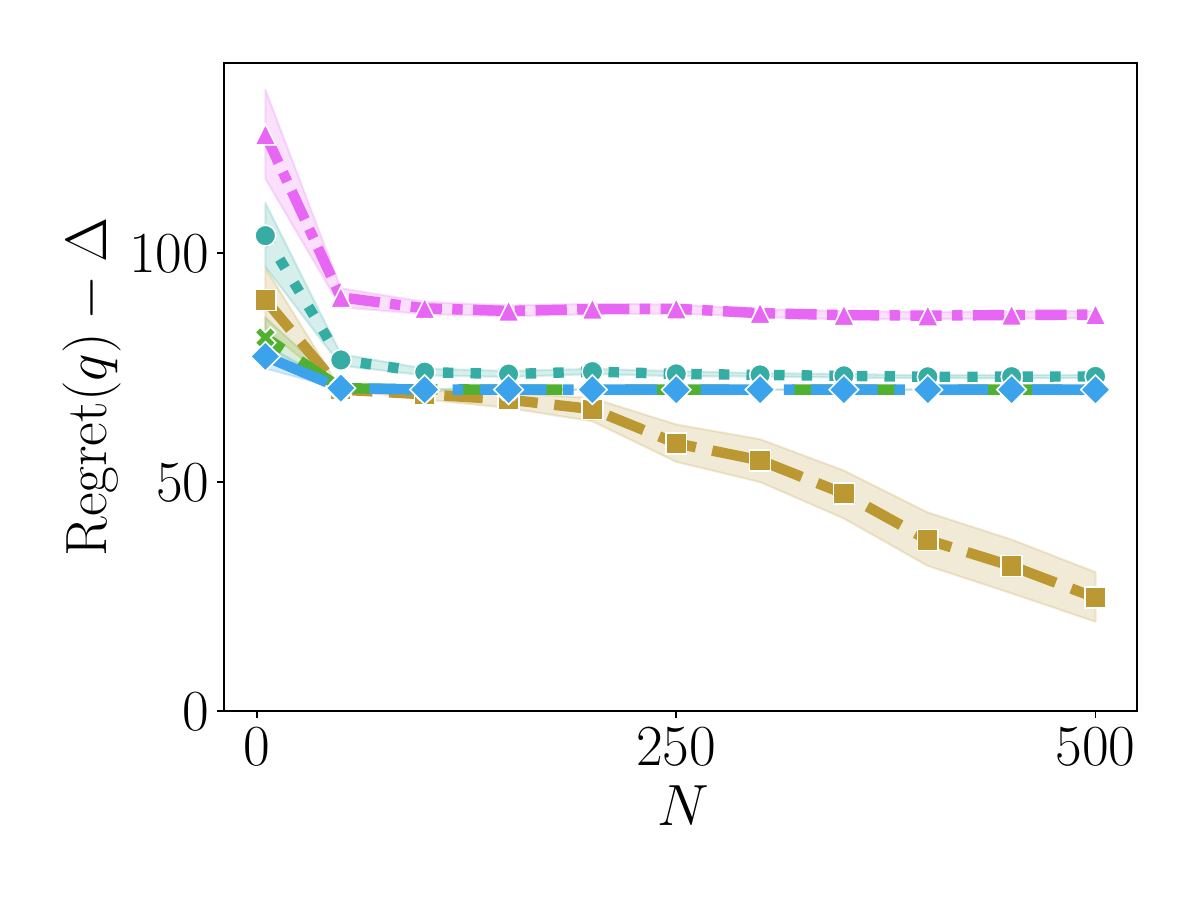}
        \caption{$\lambda = 82.64$ (Unidentifiable)}\label{fig:nb}
    \end{subfigure}
    \begin{subfigure}{.48\textwidth}
        \centering
        \includegraphics[width=\textwidth]{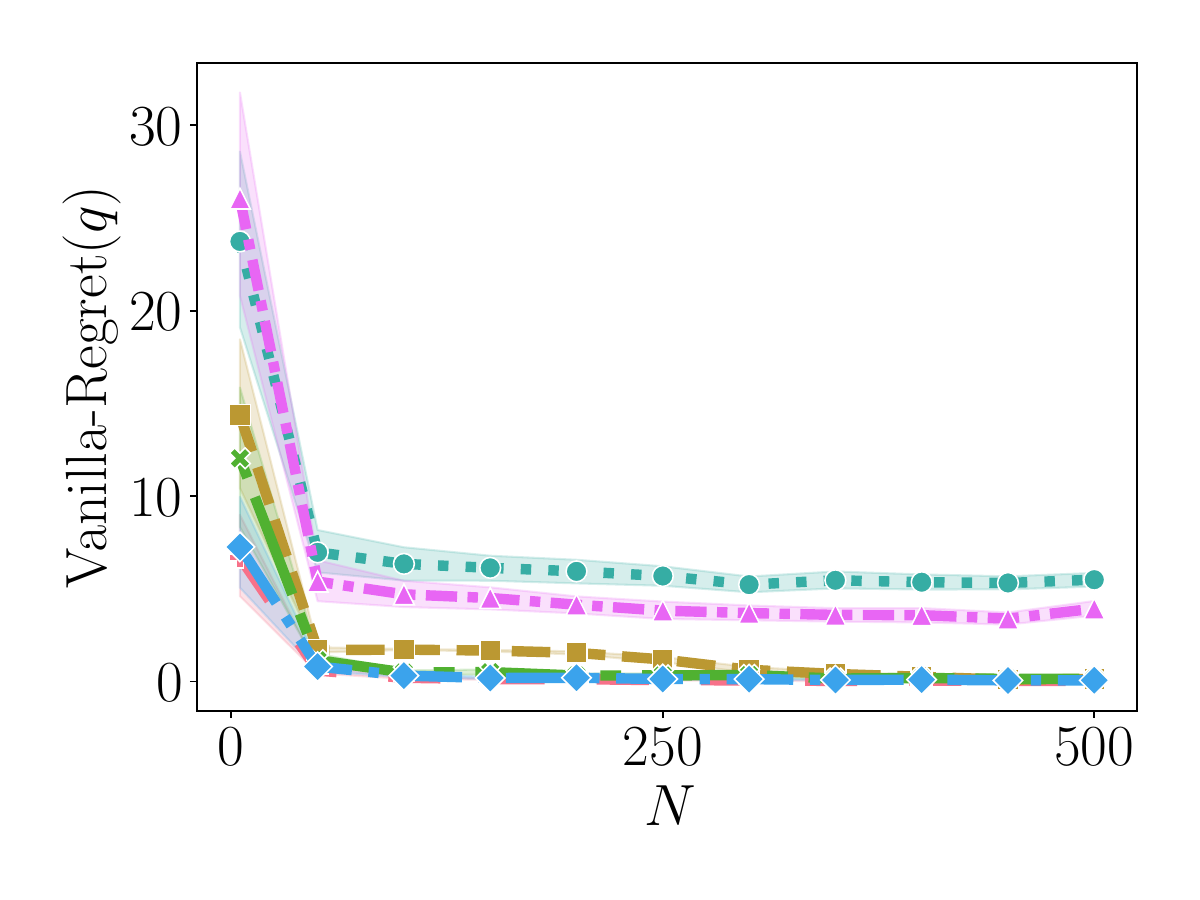}
        \caption{$\lambda = 95.36$ (Identifiable)}\label{fig:nc}
    \end{subfigure}
    \begin{subfigure}{.48\textwidth}
        \centering
        \includegraphics[width=\textwidth]{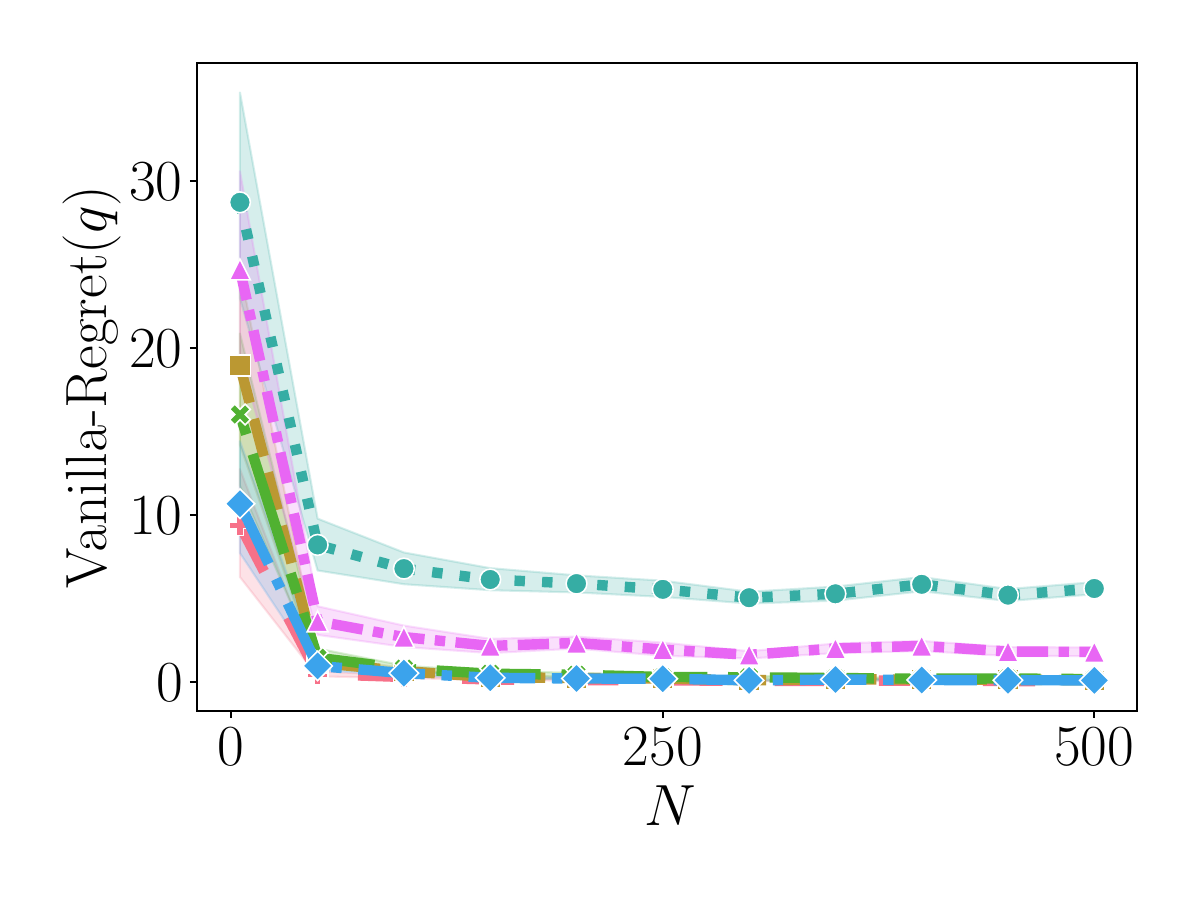}
        \caption{$\lambda = 108.07$ (Identifiable)}\label{fig:nd}
    \end{subfigure}
    \begin{subfigure}{\textwidth}
        \centering
        \includegraphics[width = .75\textwidth]{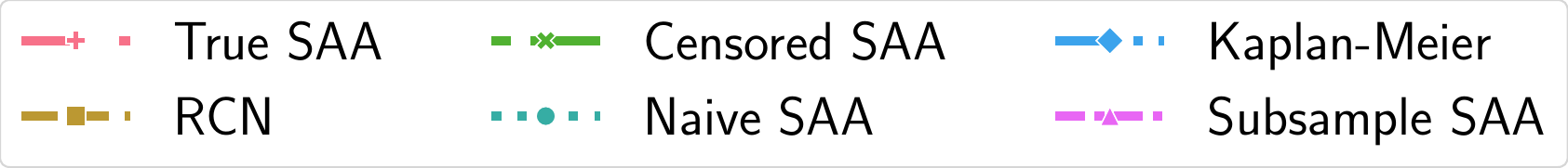}
    \end{subfigure}
    \caption{\centering {$\Regret(q) - \risk$ vs. $N$ (when unidentifiable) and $\UncensoredRegret(q)$ vs. $N$ (when identifiable) for the different benchmark policies, for $\lambda \in \{69.93,82.64,95.36,108.07\}$. Here, $D \sim \text{Uniform}\{0,100\}$ and $\rho = 0.9$, with $\qopt_G = 89$.}}
    \label{fig:line_plots}
\end{figure}

We first investigate the performance of all policies of interest as $N$ grows large. Here, the underlying demand distribution $D\sim\text{Uniform}\{0,100\}$ and $\rho = 0.9$, for various values of $\lambda$. For this value of $\rho$, $\qopt_G = 89$; therefore, the problem is unidentifiable for $\lambda \in \{69.93, 82.64\}$ and identifiable for $\lambda \in \{95.36, 108.07\}$. (The results for all other sets of parameters and distributions are entirely analogous. We omit them as such.)

\smallskip

When the problem is unidentifiable, we plot $\Regret(q^\pi)-\risk$ versus $N$ (see \Cref{fig:na,fig:nb}). \ALG is the only algorithm whose worst-case regret asymptotically converges to the minimax risk $\risk$. We observe that $\Regret(q^\pi)-\risk$ converges to a constant for all other policies. This is expected, as these latter policies aim to estimate $\qopt_G$, as opposed to attempting to hedge against nature in the unidentifiable regime. Hence, additional samples do not help their performance.  Due to censoring at $\lambda$, these algorithms often converge to outputting $\lambda$, which is why \CensoredSAA and \KM achieve similar performance.

Notice moreover that our algorithm's convergence to $\risk$ is much slower when $\lambda = 82.64$ compared to when $\lambda = 69.93$. This is because the problem is much closer to identifiability (which would occur at $\lambda \geq \qopt_G = 89$), meaning that \ALG requires significantly more samples to test whether or not the problem is identifiable. Specifically, when $N$ is small, we are in the difficult regime where \ALG does not have confidence about whether or not the problem is identifiable, and defaults to outputting $\lambda$, instead of an estimate of $\qcrit_G$. With more samples, however, \ALG correctly classifies the problem as {likely unidentifiable}, and outputs $\qalg = \qcrit_{\offeCDF}$. Despite this burn-in period, across all values of $\lambda$ in the {unidentifiable} regime \ALG ensures a $5\%$ relative optimality gap $\mathcal{R}^{ui}(\qalg)$ in roughly $N \sim 150$ samples.

\smallskip

{When the problem is identifiable, given that the optimal newsvendor cost is indeed achievable in the worst case, we plot \UncensoredRegret\xspace versus $N$ (see \Cref{fig:nc,fig:nd}). We observe that \ALG, \KM, \CensoredSAA, and \TrueSAA all quickly converge to the optimal newsvendor quantity $\qopt_G$, resulting in vanishing vanilla regret. \minedit{On the other hand, the performance of \SubsampleSAA and \IgnorantSAA plateaus as $N$ grows large.} This again is expected since these policies learn biased estimates of the underlying demand distribution.
Note that, for $\lambda = 95.36$, while our algorithm's regret is very low, it lies slightly above that of the three other policies for small values of $N$. This is the same phenomenon as the one observed when $\lambda = 82.64$: our policy requires a {burn-in period} when $\lambda$ is close to $\qopt_G$ to determine identifiability.}

\subsubsection{Impact of observable boundary.}
\label{sec:experiment_observable_boundary}

We next study the impact of the observable boundary $\lambda$ on policy performance.  Intuitively, since $\Regret(q^\pi)-\risk$ represents the loss due to errors in estimating $\qrisk$, \ALG should exhibit the best performance for extreme values of $\lambda$, when the problem is {\it easily unidentifiable} ($\lambda \ll \qopt_G$) or {\it easily identifiable} ($\lambda \gg \qopt_G$), since these regimes are where the algorithm will make the fewest identifiability errors. As $\lambda$ approaches $\qopt_G$ from either direction, however, one would expect the frequency with which \ALG either misclassifies the regime or defaults to $\lambda$ to increase, leading to higher values of $\Regret(q^\pi)-\risk$. 

To test this hypothesis, we fix $N = 500$ and $\rho = 0.9$, and consider the uniform, exponential, and Poisson distributions described in \cref{sec:experiments_benchmarks}. Our results are summarized in \cref{tab:rel-vs-lam-09}; the thick vertical line in each table marks the transition from the unidentifiable regime, where we report $\Regret(q^\pi) - \risk$ and $\mathcal{R}^{ui}(q^\pi)$ in parentheses, to the identifiable regime, where we report $\UncensoredRegret(q^\pi)$ and $\mathcal{R}^{id}(q^\pi)$ in parentheses. Results for other values of $\rho$ are provided in Appendix~\ref{apx:observe-boundary}.

\begin{table}[!ht]
\setlength\tabcolsep{3pt}
\small
\begin{subtable}[b]{\linewidth}
\centering
\begin{tabular}{lrrrr!{\vrule width 1.5pt}rrrr}
\toprule
$\lambda$ & 44.50 & 57.21 & 69.93 & 82.64 & 95.36 & 108.07 & 120.79 & 133.50 \\
\midrule
\TrueSAA &  &  &  &  & 0.0 (0.1\%) & 0.0 (0.1\%) & 0.0 (0.0\%) & 0.0 (0.0\%) \\
\midrule
\ALG & 4.0 (1.7\%) & 6.7 (3.3\%) & 7.7 (4.5\%) & 27 (27\%) & 0.1 (0.2\%) & 0.1 (0.2\%) & 0.1 (0.1\%) & 0.1 (0.2\%) \\
\CensoredSAA & 1033 (450\%) & 653 (320\%) & 341 (200\%) & 70 (70\%) & 0.1 (0.3\%) & 0.1 (0.2\%) & 0.1 (0.2\%) & 0.1 (0.3\%) \\
\KM & 1033 (450\%) & 653 (320\%) & 341 (200\%) & 70 (70\%) & 0.1 (0.1\%) & 0.1 (0.1\%) & 0.0 (0.1\%) & 0.1 (0.1\%) \\
\IgnorantSAA & 1033 (450\%) & 653 (320\%) & 341 (200\%) & 73 (73\%) & 5.0 (11\%) & 5.4 (12\%) & 4.5 (9.9\%) & 4.3 (9.6\%) \\
\SubsampleSAA & 1067 (465\%) & 684 (335\%) & 369 (217\%) & 87 (87\%) & 3.5 (7.7\%) & 1.9 (4.2\%) & 2.0 (4.5\%) & 1.8 (4.1\%) \\
\bottomrule
\end{tabular}
\caption{Uniform, $\qopt_G = 89$}
\label{tab:uniform}       
\end{subtable}
\medskip
\begin{subtable}[b]{\linewidth}
\centering
\begin{tabular}{lrrrr!{\vrule width 1.5pt}rrrr}
\toprule
$\lambda$ & 92.07 & 118.38 & 144.68 & 170.99 & 197.30 & 223.60 & 249.91 & 276.22 \\
\midrule
\TrueSAA &  &  &  &  & 0.4 (0.2\%) & 0.4 (0.2\%) & 0.3 (0.2\%) & 0.3 (0.2\%) \\
\midrule
\ALG & 6.7 (4.2\%) & 6.9 (6.0\%) & 21 (29\%) & 4.3 (19\%) & 1.0 (0.5\%) & 7.4 (4.0\%) & 3.9 (2.1\%) & 0.8 (0.5\%) \\
\CensoredSAA & 345 (216\%) & 148 (128\%) & 45 (64\%) & 4.6 (20\%) & 0.6 (0.3\%) & 0.9 (0.5\%) & 0.8 (0.4\%) & 0.6 (0.3\%) \\
\KM & 345 (216\%) & 148 (128\%) & 45 (64\%) & 4.5 (19\%) & 0.6 (0.3\%) & 0.6 (0.3\%) & 0.5 (0.3\%) & 0.5 (0.3\%) \\
\IgnorantSAA & 345 (216\%) & 148 (128\%) & 57 (82\%) & 29 (125\%) & 25 (13\%) & 20 (11\%) & 18 (9.8\%) & 15 (8.1\%) \\
\SubsampleSAA & 410 (257\%) & 214 (185\%) & 108 (153\%) & 57 (243\%) & 40 (22\%) & 31 (17\%) & 25 (14\%) & 20 (11\%) \\
\bottomrule
\end{tabular}
\caption{Exponential, $\qopt_G = 184.21$}
\label{tab:exponential}       
\end{subtable}
\medskip
\begin{subtable}[b]{\linewidth}
\centering
\begin{tabular}{lrrrr!{\vrule width 1.5pt}rrrr}
\toprule
$\lambda$ & 46 & 59.14 & 72.29 & 85.43 & 98.57 & 111.71 & 124.86 & 138 \\
\midrule
\TrueSAA &  &  &  &  & 0.0 (0.2\%) & 0.0 (0.1\%) & 0.0 (0.2\%) & 0.0 (0.2\%) \\
\midrule
\ALG & 0.0 (0.0\%) & 0.5 (0.2\%) & 2.4 (1.1\%) & 7.1 (4.7\%) & 0.1 (0.3\%) & 0.0 (0.3\%) & 0.1 (0.3\%) & 0.1 (0.4\%) \\
\CensoredSAA & 2260 (900\%) & 2131 (891\%) & 1542 (697\%) & 247 (165\%) & 0.1 (0.4\%) & 0.1 (0.3\%) & 0.1 (0.4\%) & 0.1 (0.4\%) \\
\KM & 2260 (900\%) & 2131 (891\%) & 1542 (697\%) & 247 (165\%) & 0.1 (0.3\%) & 0.0 (0.3\%) & 0.1 (0.3\%) & 0.1 (0.4\%) \\
\IgnorantSAA & 2260 (900\%) & 2131 (891\%) & 1542 (697\%) & 247 (165\%) & 2.0 (12\%) & 2.0 (12\%) & 1.9 (12\%) & 1.6 (9.8\%) \\
\SubsampleSAA & 2260 (900\%) & 2138 (894\%) & 1544 (699\%) & 250 (168\%) & 0.2 (1.2\%) & 0.1 (0.7\%) & 0.3 (1.7\%) & 0.3 (1.7\%) \\
\bottomrule
\end{tabular}
\caption{Poisson, $\qopt_G = 92$}
\label{tab:poisson}       
\end{subtable}
\centering
\caption{
Impact of $\lambda$ on policy performance. Values to the left of the thick vertical line correspond to the unidentifiable regime, where we report $\Regret(q^\pi) - \risk$ and $\mathcal{R}^{ui}(q^\pi)\%$; values to the right of the thick vertical line correspond to the identifiable regime, where we report $\UncensoredRegret(q^\pi)$ and $\mathcal{R}^{id}(q^\pi)\%$.}
\label{tab:rel-vs-lam-09}
\end{table}

When the problem is {\it easily unidentifiable} (i.e., $\lambda \ll \qopt_G$), \ALG is the only near-optimal algorithm, incurring regret of at most 6\% relative to the minimax risk $\risk$. (This occurs in the case of the exponential distribution, for $\lambda = 118.38$; see \Cref{tab:exponential}.) While our algorithm approximates the unidentifiable quantile $\qcrit_G$, the non-robust algorithms incur regret that is at least twice as high as $\risk$ in the best case, and 10 times as high in the worst case. Observe that the performance of these algorithms is near-identical across many instances. This is because, for small values of $\lambda$, any SAA-style solution outputs a quantity close to $\lambda$, since most demand realizations are censored.

When the problem is {\it easily identifiable} (i.e. $\lambda \gg \qopt_G$), \ALG remains near-optimal: in the worst case, its cost is within $4\%$ of the true newsvendor optimum (occurring at $\lambda = 223.60$ in \Cref{tab:exponential}). \CensoredSAA and \KM similarly perform exceptionally well, with a relative vanilla regret of at most 1\% in all cases. Again, this highlights that censoring is essentially immaterial in this regime.  Finally, \SubsampleSAA and \IgnorantSAA retain their poor performance across most instances. This performance slightly improves for larger values of $\lambda$, since fewer demand samples are censored as $\lambda$ increases.

We note that across all three tables, as $\lambda$ approaches $\qopt_G$ from below, \ALG incurs higher regret. This is in line with our intuition that determining identifiability becomes more challenging close to the identifiability boundary $\lambda = \qopt_G$, with \ALG making more misclassification errors. The exception to this trend is the exponential distribution (\Cref{tab:exponential}), where we observe the algorithm's regret steeply decreasing at $\lambda = 170.99$. For this instance, \ALG classifies the problem as ``knife-edge'' in 99\% of replications, thereby defaulting to outputting $\lambda$. However, for this value of $\lambda$, $\qcrit_G = 184.97$; this explains the algorithm's strong performance in the knife-edge regime.

Overall, our results illustrate the robustness of \ALG  to varying values of $\lambda$, with strong performance relative to the true newsvendor cost when the problem is identifiable. Moreover, our experiments show that the benefits of our algorithm are the largest in highly censored settings.\footnote{\minedit{We refer the reader to Appendix \ref{apx:knife-edge} for additional experiments that ``zoom into'' the knife-edge regime for each of these three distributions. In this ``zoomed in'' regime, we observe that \ALG, \KM, and \CensoredSAA achieve comparable performance according to both worst-case and vanilla regret metrics.}}

\subsubsection{Demand variability.}
\label{sec:experiments_variance_parameters}

\begin{figure}[!t]
    \centering
        \centering
        \includegraphics[scale=0.4]{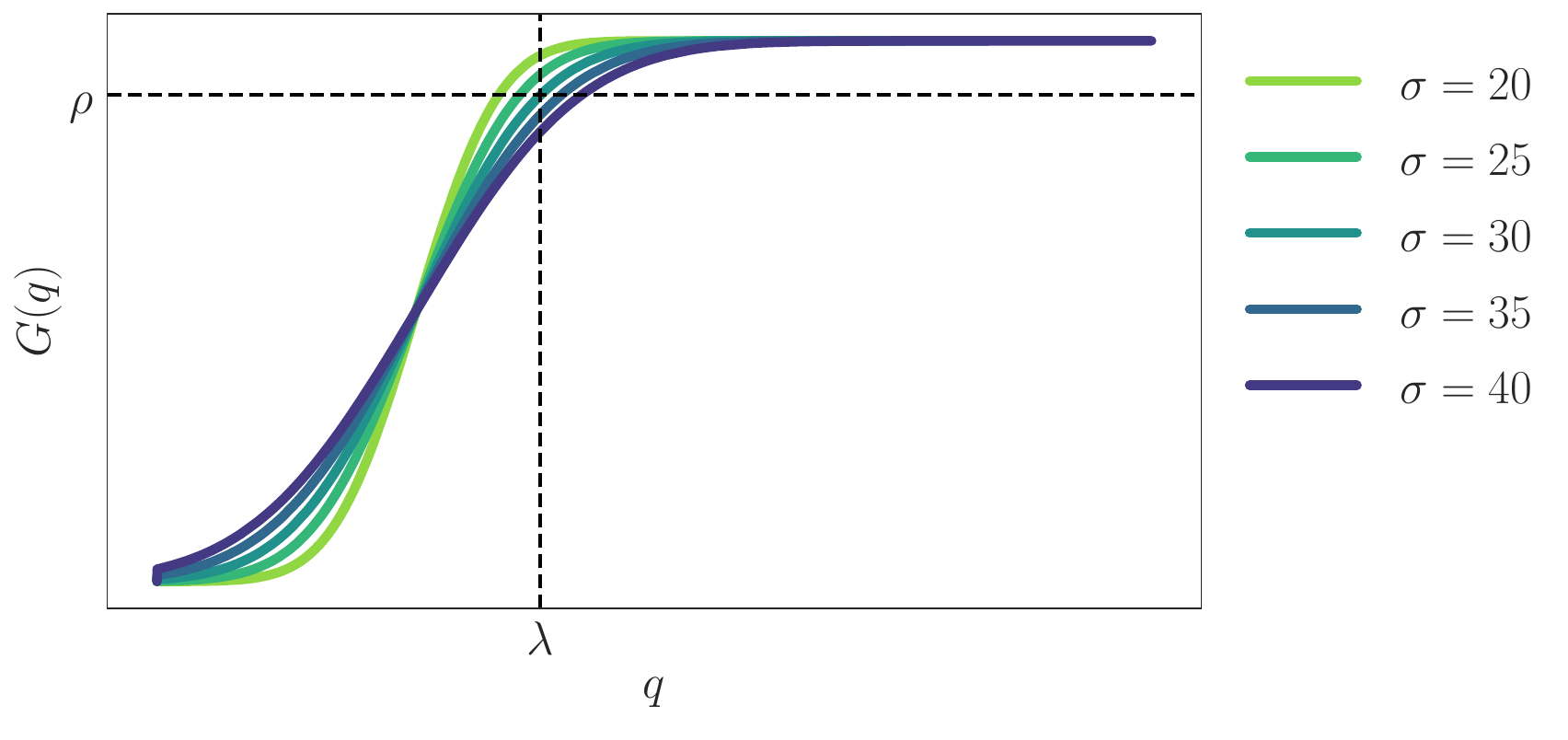} 
        \caption{Cumulative distribution function of $D = \max\{0,X\}$, with $X\sim \mathcal{N}(80,\sigma^2)$, $\sigma \in \{20,25,\ldots,40\}$.}
        \label{fig:normal}
\end{figure}
\begin{table}
\centering
            \small
        \centering
\begin{tabular}{lrrr!{\vrule width 1.5pt}rr}
\toprule
$\sigma$ & 20 & 25 & 30 & 35 & 40 \\
\midrule
\TrueSAA & 0.0 (0.1\%) & 0.1 (0.2\%) & 0.1 (0.1\%) &  & \\
\midrule
\ALG & 0.1 (0.3\%) & 1.2 (2.6\%) & 0.0 (0.0\%) & 20 (37\%) & 24 (29\%) \\
\CensoredSAA & 0.2 (0.5\%) & 0.2 (0.4\%) & 0.1 (0.2\%) & 20 (36\%) & 57 (68\%) \\
\KM & 0.1 (0.3\%) & 0.1 (0.3\%) & 0.1 (0.2\%) & 20 (36\%) & 57 (68\%) \\
\IgnorantSAA & 4.3 (12\%) & 5.4 (12\%) & 6.4 (12\%) & 26 (48\%) & 61 (73\%) \\
\SubsampleSAA & 1.2 (3.5\%) & 3.8 (8.7\%) & 8.2 (16\%) & 34 (63\%) & 76 (91\%) \\
\bottomrule
\end{tabular}
        \caption{
        Impact of $\sigma$ on policy performance. Here, $\rho = 0.9$ and $\lambda = 118.46$. Values of $\sigma$ to the right of the thick vertical line correspond to the unidentifiable regime, where we report $\Regret(q^\pi) - \risk$ and  $\mathcal{R}^{ui}(q^\pi)\%$; values of $\sigma$ to the left of the thick vertical line correspond to the identifiable regime, where we report $\UncensoredRegret(q^\pi)$ and $\mathcal{R}^{id}(q^\pi)\%$.}
        \label{tab:normal}
\end{table}

We conclude our synthetic experiments by investigating the impact of demand variability on algorithm performance. Here, we fix \minedit{$N = 500$}, $\rho = 0.9$, and let \mbox{$D = \max\{0,X\}$}, with {\mbox{$X \sim \mathcal{N}(80, \sigma^2)$}, varying $\sigma \in \{20,25,\ldots,40\}$.}
In order to cover both identifiable and unidentifiable regimes, we choose $\lambda = 118.46 = G^{-1}(0.9)$, where $G$ corresponds to the distribution with $\sigma = 30$. 
\Cref{fig:normal} shows the cdf of $D$ for these five values of $\sigma$. Comparing $G(q)$ to $\rho$ on this figure, we observe that $\qopt_G$ is increasing in $\sigma$. Moreover, we can see that the problem is identifiable for $\sigma \in \{20,25,30\}$ and unidentifiable for $\sigma \in \{35,40\}$. \Cref{tab:normal} presents our results across these identifiability regimes. (Entirely analogous results for other values of $\lambda$ are deferred to Appendix \ref{apx:demand-var}.)

We again observe that \ALG has comparable performance to \CensoredSAA and \KM in the identifiable regime, {despite} using half as many samples. Note moreover that, when $\sigma = 30$, \ALG achieves {\em perfect} performance. This occurs because the algorithm finds itself in the {knife-edge} case and defaults to outputting $\lambda$, and, as we can see in \Cref{fig:normal}, $\qopt_G = \lambda$ for $\sigma = 30$. At $\sigma = 35$, having just entered the unidentifiable regime, we observe a degradation in performance across all algorithms, with \ALG, \CensoredSAA and \KM achieving similar relative regret of around 36\%. This is due to the fact that all three algorithms output quantities close to $\lambda$, but slightly off from $\qrisk = \qcrit_G$. When $\sigma = 40$, $\qopt_G \gg \lambda$; in this easily unidentifiable regime, our policy's performance improves, in contrast to that of \CensoredSAA and \KM, which decreases. Overall, these experiments underscore that \ALG is robust to demand variability.

\subsection{Real-World Dataset}
\label{sec:exp_real_data}

Finally, we complement our synthetic experiments by demonstrating our algorithm's robust performance on a real-world grocery retail dataset \citep{wang2025freshretailnet}. Given the poor performance of \IgnorantSAA and \SubsampleSAA in our synthetic experiments, for this set of experiments we only consider \KM and \CensoredSAA as benchmark algorithms.

\subsubsection{Dataset description and pre-processing.} The original dataset consists of three months of detailed hourly sales data for 863 distinct perishable products across 898 grocery stores in China.\footnote{In the remainder of this section, we abuse terminology and refer to each (product, store) tuple as a product.} While the dataset does not specify the number of units in stock at the beginning of each day, it indicates whether an out-of-stock event occurred for each (product, hour) tuple. We use this out-of-stock indicator to reconstruct historical inventory levels as follows.

For each product $p$ in the dataset, we first filter out all days that began with an out-of-stock event. For each remaining day $t$, we reconstruct the initial inventory level of the product, denoted by $\inv{pt}$, by first considering the set of all days in which a stockout event occurred. We use $\stockoutdays_p$ to denote this set. Then, for all $t \in \stockoutdays_p$, $\inv{pt}$ corresponds to the total sales recorded up until the stockout event. Now, for all $t \in \stockoutdays_p^c$, we let $\inv{pt} = \max\left\{\min\limits_{\substack{t'\in\stockoutdays_p:\\\inv{pt'} > \sales{pt}}}\inv{pt'}, \sales{pt}+1\right\}$, where $\sales{pt}$ is used to denote the recorded sales on day $t$. In words, on non-stockout days we let the inventory level be the smallest known inventory level (i.e., across all stockout days); if the recorded sales exceeds all known inventory levels, we assume the inventory level exceeded the recorded sales by 1. We use this data to construct $K_p$ selling seasons for product $p$, where each selling season corresponds to days $t$ for which the initial inventory levels where identical. With this setup in hand, in the remainder of the section we adopt the same notation as before, letting $N_{pk}$ denote the number of days associated with selling season $k$ of product $p$, with $\soff_{pki}$ and $\qoff_{pki}$ respectively denoting the observed sales and historical ordering quantities for day $i \in [N_{pk}]$. Finally, let $\mathcal{P}$ be the set of all products in our pre-processed dataset. 

\Cref{fig:train-stats-oos} shows the distribution of out-of-stock frequency (i.e., the percentage of days in which demand was censored) in our dataset. On average, products stocked out on 31\% of days, with 25\% of products stocking out at least 40\% of the time. These results highlight that {\it demand censoring is a common occurrence in practice}, and underscores the importance of algorithms that are robust to high levels of censoring.

\begin{figure}[t]
\centering
\includegraphics[scale=0.35]{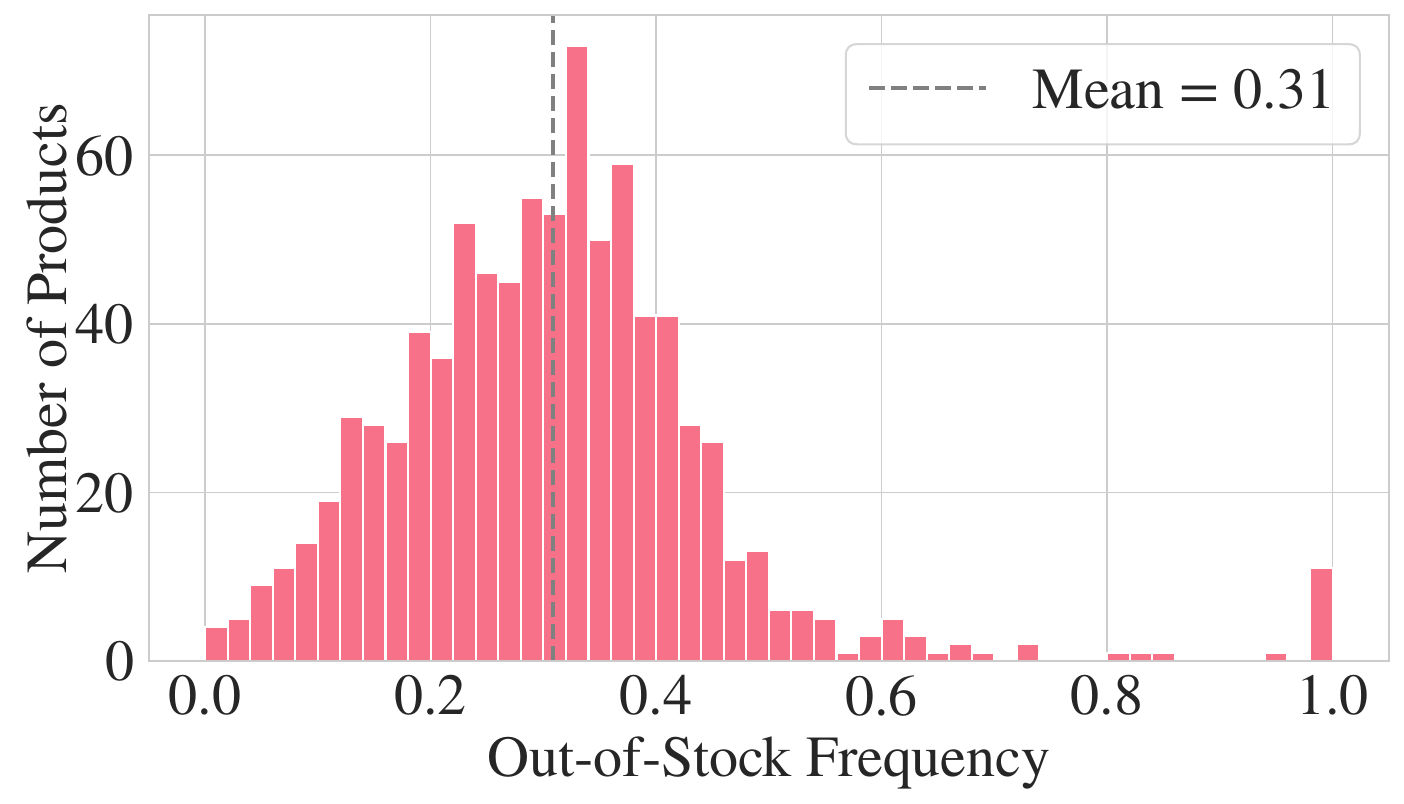}
\caption{Distribution of out-of-stock frequency.}\label{fig:train-stats-oos}
\end{figure}

\paragraph{Training and testing datasets.}  The dataset provided by \cite{wang2025freshretailnet} is randomly partitioned into training and evaluation sets, for each product $p \in \mathcal{P}$. These sets respectively comprise $2,605,058$ and $215,978$ samples (i.e., a 92-8 split), where each sample corresponds to a $(\soff_{pki}, \qoff_{pki})$ pair. Only the training set is used by the algorithms; the testing set is exclusively used for evaluation. We highlight that there is no notion of {\it true} demand distribution $G_p$ that the decision-maker could use to evaluate the benchmark algorithms, given that there is no data past the observable boundary of the testing set. Indeed, this is precisely the obstacle that motivates the introduction of the distributionally robust optimization framework for this problem. Given this issue, we use our testing set to construct a partial demand distribution $G_p$ {\it up until the observable boundary $\lambda_p$}. This partial distribution, constructed using the Kaplan-Meier estimator, is used to determine identifiability for each tested value of $\rho$; we also report the worst-case regret of each algorithm using this partial cdf.

\Cref{fig:sample-dist} shows the partial cdf's for six products in our dataset. These cdf's exhibit significant heterogeneity, both with respect to the observed support of the distributions (where, for instance, the blue distribution is supported on a wide range of values in $[0,15]$, whereas the pink curve is supported on few values in $[0,4]$) as well as $\Gminus(\lambda)$ itself (with these two distributions having $\Gminus(\lambda) = 0.05$ and $\Gminus(\lambda) = 1$, respectively). \Cref{fig:gminus-lam-hist}, which shows the distribution of $\Gminus(\lambda)$ across all products in the dataset, further underscores this heterogeneity. We observe that the average value of $\Gminus(\lambda)$ across products is 0.83, with only 66\% of products such that $\Gminus(\lambda) \geq 0.9$. Given this, we expect non-robust algorithms to perform particularly poorly in settings where $b$ is much larger than $h$, as is frequently the case in practice, and assumed in the literature \citep{huh2009nonparametric}.

\begin{figure}[t]
 \centering
 \begin{subfigure}[b]{0.45\textwidth}
\includegraphics[width=\textwidth]{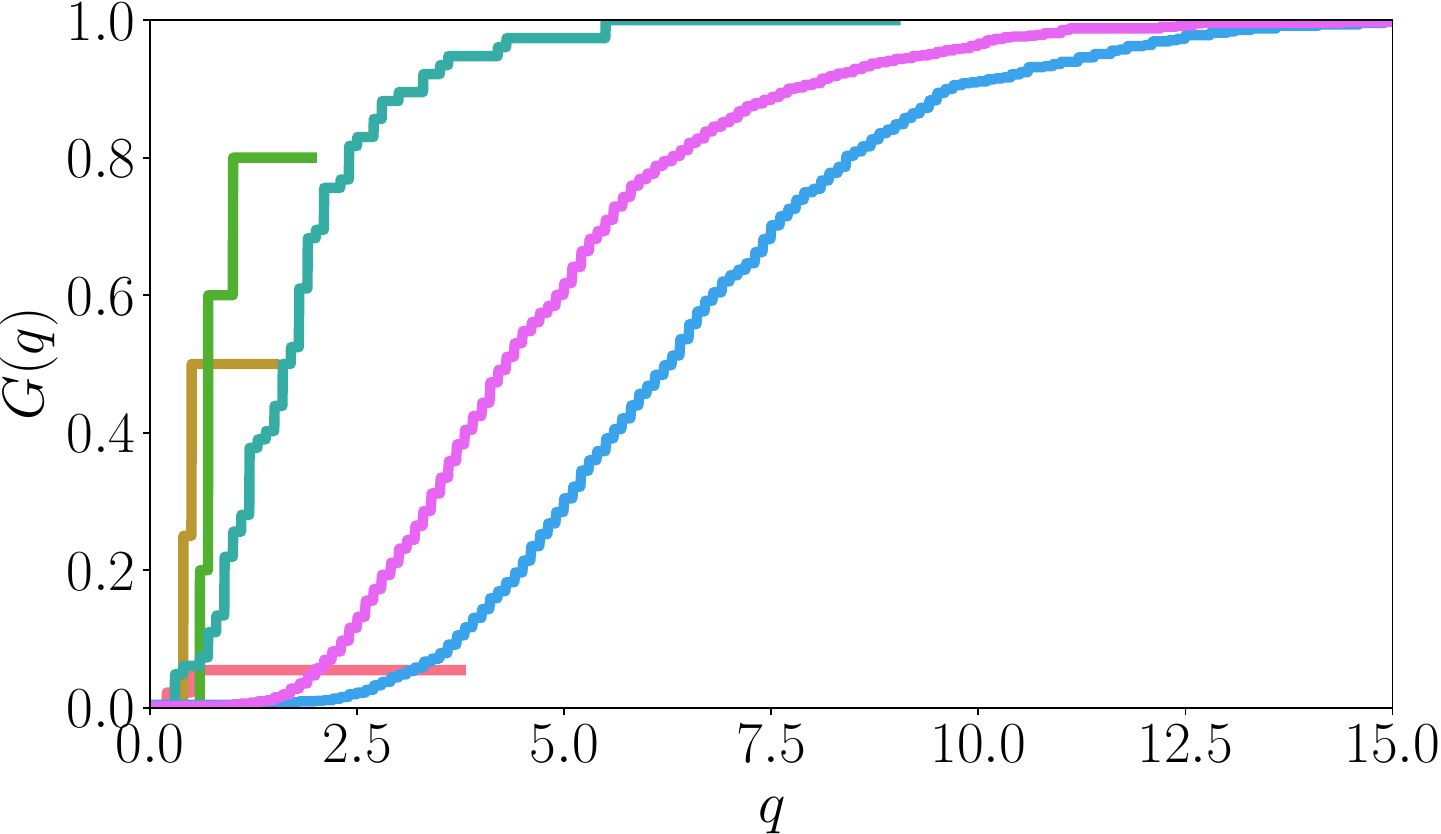}
\caption{\small Partial cdfs of six sample products. }\label{fig:sample-dist}
\label{fig:oos}
\end{subfigure}
\begin{subfigure}[b]{0.45\textwidth}
\centering
\includegraphics[width=\textwidth]{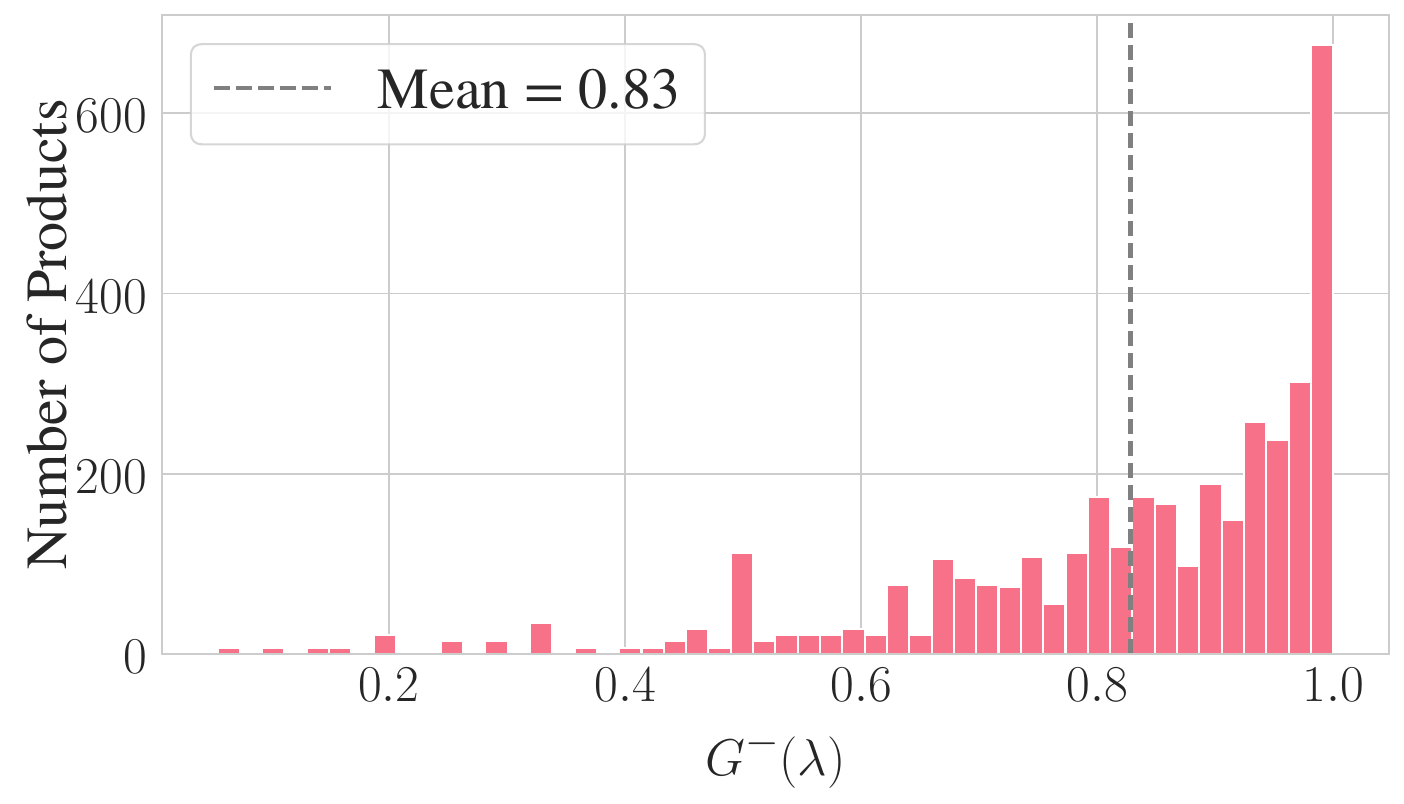}
\caption{\small Distribution of $\Gminus(\lambda)$}\label{fig:gminus-lam-hist}
\end{subfigure}
\caption{Descriptive statistics on observed level of demand censoring.}\label{fig:censoring-descriptive-stats}
\end{figure}

\subsubsection{Algorithms.} 
Despite the strong theoretical guarantees of \ALG, our algorithm discards potentially valuable data by exclusively using sales data where the historical ordering quantity is equal to $\lambda$. To illustrate the salience of this issue, we plot statistics related to the sample sizes associated with each historical selling season, for each product in our dataset. \Cref{fig:num-order-levels} first shows the distribution of $K_p$, the number of selling seasons for each product. On average, products have over 20 historical ordering quantities that the algorithm can leverage. Moreover, \Cref{fig:train-stats-nk} shows that the vast majority of products have fewer than 10 samples at the boundary, with the average sample size at the boundary less than 4, and only 1\% of products having $N_K \geq 20$. In contrast, \Cref{fig:train-stats-lower-vals}, which plots the average number of samples associated with each selling season, shows that the sample size increases significantly as we move away from the boundary. In particular, by the 10th-highest historical selling season, the number of samples exceeds 75; it reaches 200 by the 15th-highest historical selling season. These observations underscore the potential value of using the entirety of the offline dataset.

\begin{figure}[t]
\begin{subfigure}[b]{0.45\textwidth}
\centering
\includegraphics[width=\textwidth]{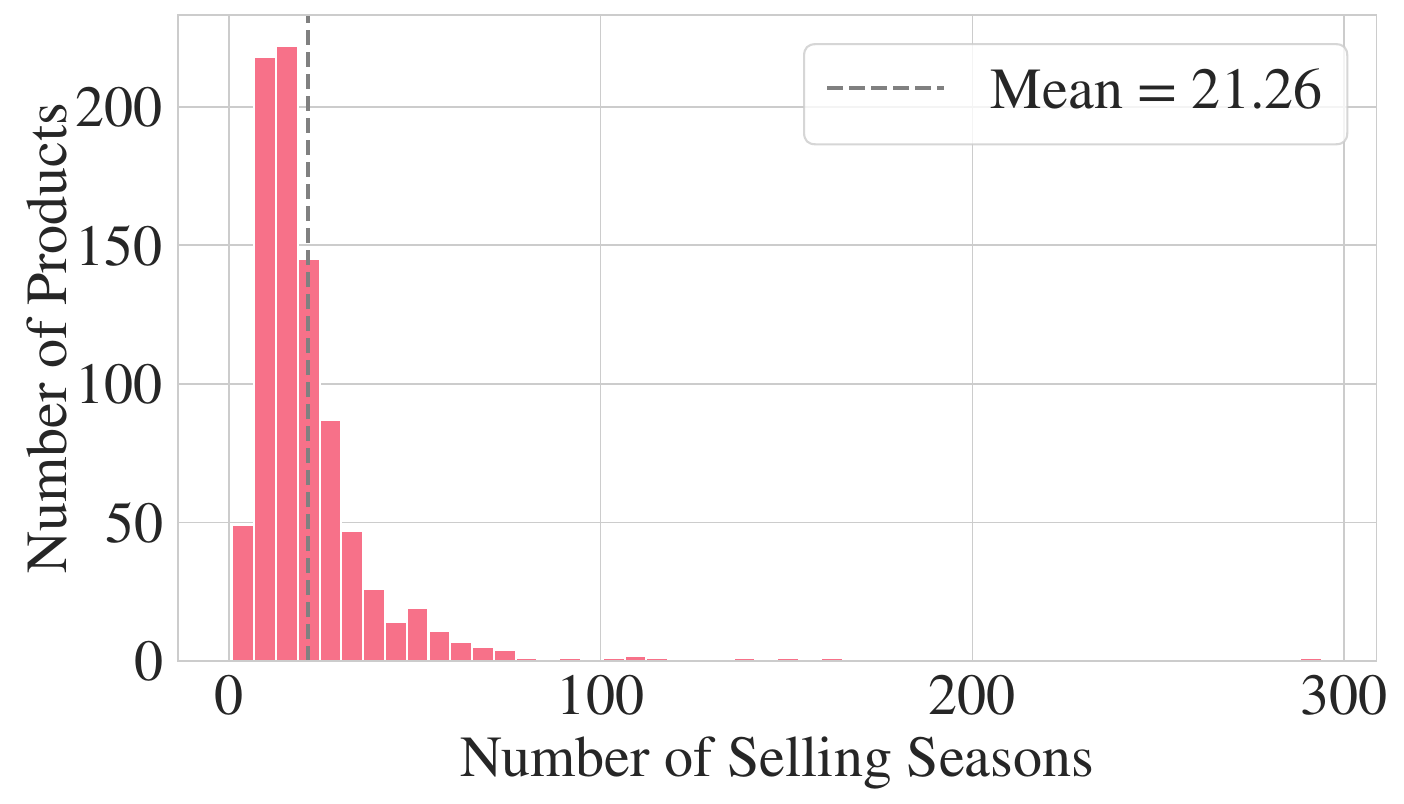}
\caption{\centering \small Distribution of number of selling seasons}\label{fig:num-order-levels}
\end{subfigure}
\begin{subfigure}[b]{0.45\textwidth}
\centering
\includegraphics[width=\textwidth]{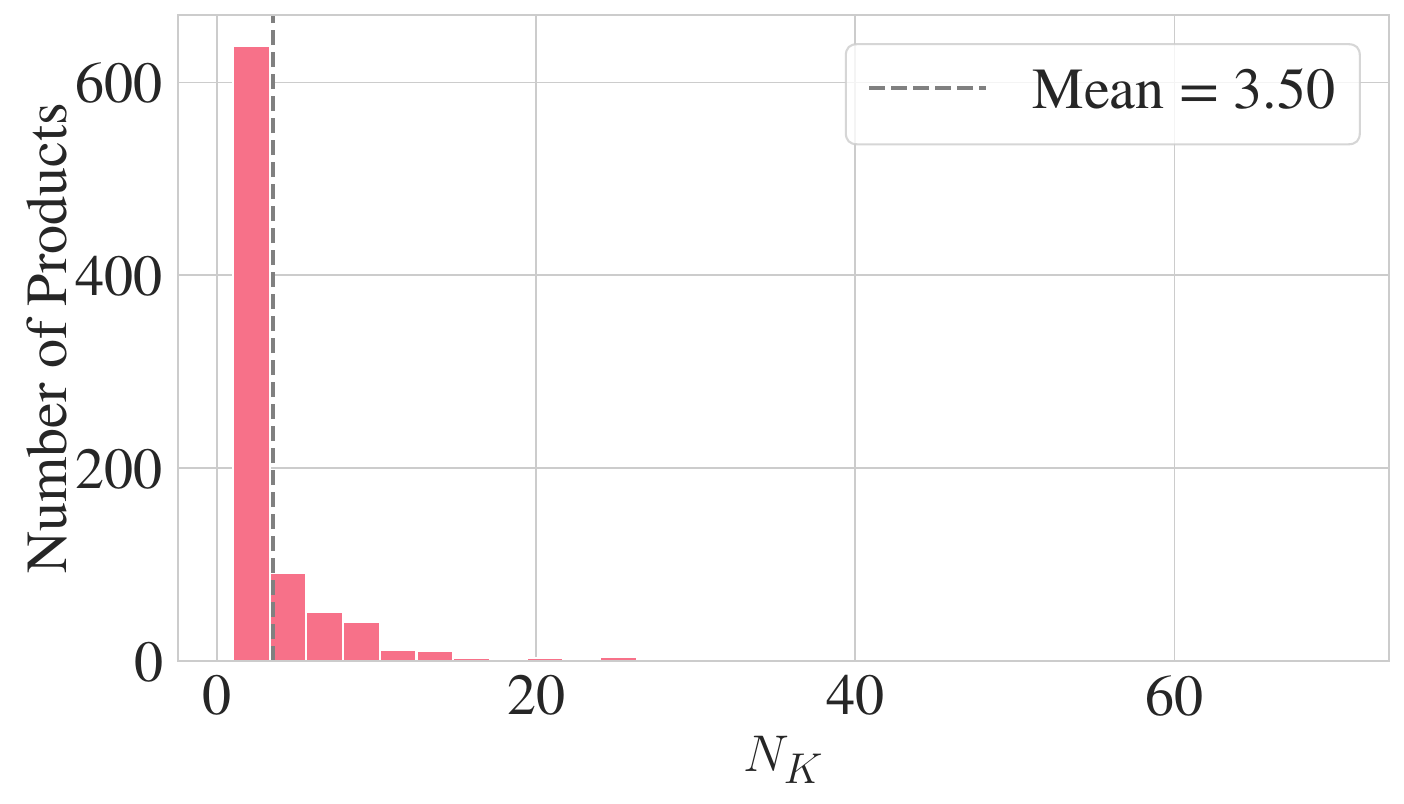}
\caption{\centering \small Distribution of $N_K$}\label{fig:train-stats-nk}
\end{subfigure}
\\
\centering \begin{subfigure}[b]{0.45\textwidth}
\centering
\includegraphics[width=\textwidth]{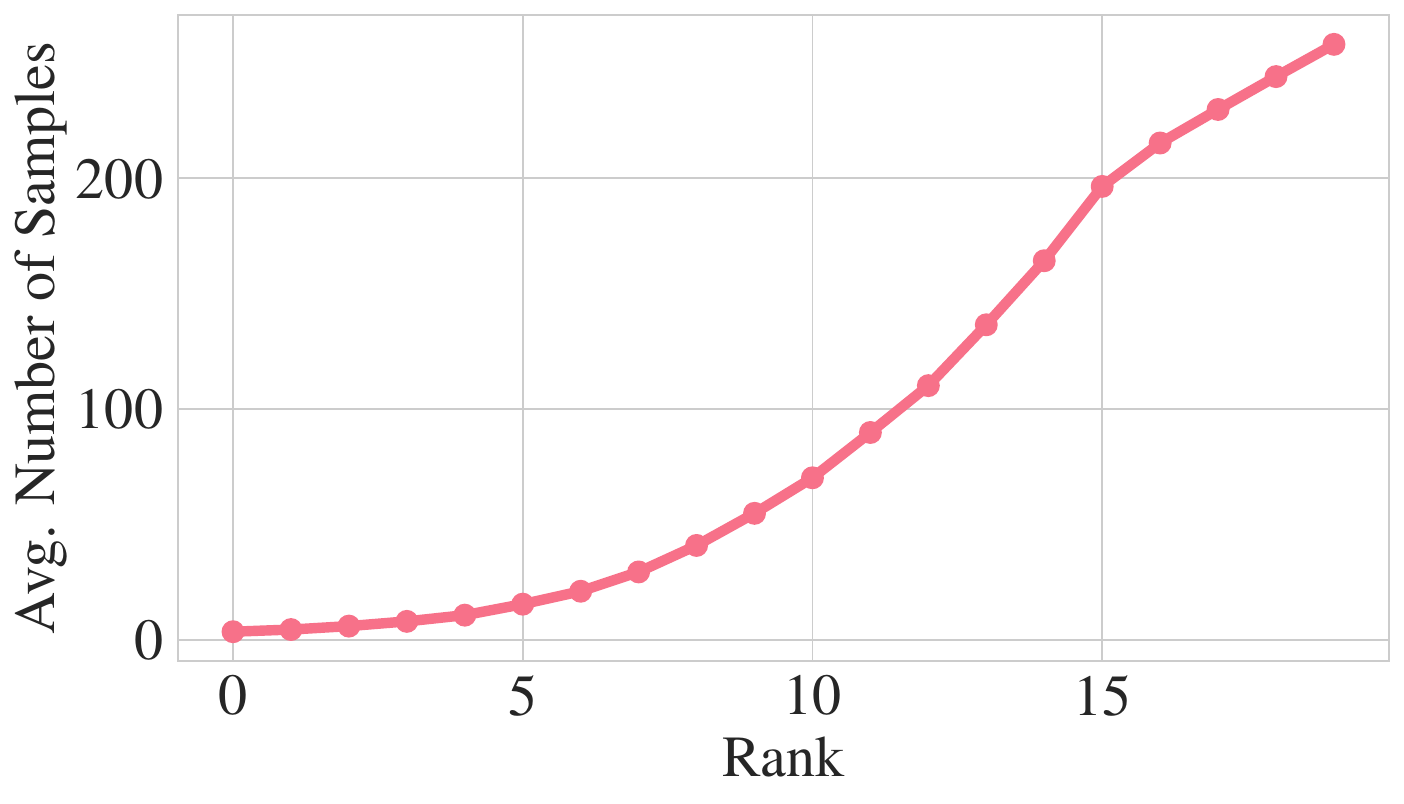}
\caption{\centering \small Average number of samples vs. season rank}\label{fig:train-stats-lower-vals}
\end{subfigure}
\caption{Descriptive statistics on training dataset sample size. In \Cref{fig:train-stats-lower-vals}, rank 0 corresponds to the selling season associated with $\lambda_p$; rank 1 corresponds to the selling season with the next highest ordering quantity, etc.}
\label{fig:train-stat}
\end{figure}

We address this sample inefficiency by proposing \ALGplus, an extension of \ALG that leverages {\it all} sales data. Before presenting the modified algorithm, we provide some intuition behind our approach.\footnote{For notational simplicity we omit all quantities' dependence on the product $p$ in the remainder of this subsection.} Suppose there exists historical ordering quantity $\qoff_k < \lambda$ such that $\Gminus(\qoff_k) \geq \rho$. Then, $\Gminus(\lambda) \geq \rho$, which implies that the underlying problem is identifiable, with minimax optimal ordering quantity $\qrisk = \qopt_G$, by \Cref{thm:minimax-risk-identifiable}. Our modified algorithm leverages this insight to improve the sample efficiency of \ALG. Specifically, the first stage of our algorithm can be viewed through the lens of multiple hypothesis testing. For each $k \in [K]$, we test whether $\Gminus(\qoff_k) \geq \rho$. If any $k$ passes this test, the problem is classified as {\it likely identifiable}, and we compute the censored SAA of $\qopt_G$ using {\it all} of the data that has satisfied this test. If there doesn't exist a value of $k$ for which we estimate that $\Gminus(\qoff_k) \geq \rho$, then we may be in the unidentifiable regime, and we proceed as in \ALG. We provide a formal description of this practical extension in  \Cref{alg:newsvendor_plus_plus}. Note that, while $\Gminus(\cdot)$ is monotonically increasing in $\qoff_k$, this is not the case for $\Gminushat_k(\qoff_k)$ defined in \Cref{eq:gminus-qoff}, for two reasons: (i) the inherent randomness in $\soff_k$, and (ii) heterogeneity in $N_k$. As a result, it may be the case that, in the identifiable regime, $\lambda$ fails the identifiability test, but $\qoff_k$ passes it, for some $\qoff_k < \lambda$. This demonstrates that \ALGplus is indeed able to leverage potentially valuable information from historical ordering quantities away from the boundary. We defer a theoretical analysis of \ALGplus, which relies on entirely analogous arguments to those used in the analysis of \ALG, to Appendix \ref{apx:rcn-plus}. 
\minedit{
Specifically, for $\width_k = \widetilde{\Theta}\left(\sqrt{\frac{1}{N_k}}\right)$,  we recover a regret guarantee of $\risk + \widetilde{O}\left(\sqrt{\frac{1}{N_K}} \right)$ in the unidentifiable regime where $\Gminus(\lambda) < \rho$, as in \ALG.  In the strictly identifiable regime where $\Gminus(\qoff_k) \geq \rho + 2\width_k$ for some $k \in [K]$, and   our regret bound weakly improves to $\widetilde{O}\left(\sqrt{\frac{1}{\sum_{k\in\mathcal{U}_1} N_k}}\right)$, where $\mathcal{U}_1 = \left\{k: \Gminus(\qoff_k) \geq \rho + 2\width_k \right\}$ corresponds to the set of strictly identifiable ordering quantities. Interestingly, in the knife-edge regime where $\Gminus(\lambda) \geq \rho$ and $\Gminus(\qoff_k) < \rho + 2\width_k$ for all $k \in [K]$, we obtain a weaker regret bound of $\widetilde{O}\left(\max\left\{\sqrt{\frac{1}{N_K}},\sqrt{\frac{1}{\min_{k\in\mathcal{U}_2}N_k}}\right\}\right)$, where $\mathcal{U}_2 = \left\{k: \Gminus(\qoff_k) \geq \rho\right\}$ corresponds to the set of identifiable ordering quantities.} \minedit{This weaker bound highlights an a priori unexpected tension between estimation error and regret in this regime. In particular, since $\lambda$ is close to $\qopt_G$, the regret incurred from estimating $\qopt_G$ may be higher than the gap between $\lambda$ and $\qopt_G$. Hence, defaulting to $\lambda$ may actually be { beneficial} in this regime. Despite this shortcoming of our theoretical bounds, our experiments will show that \ALGplus (weakly) outperforms \ALG in all regimes.
}

\begin{algorithm}[!t]
\minedit{
\DontPrintSemicolon 
\KwIn{Dataset $\soff$, confidence terms $\width_k$ for $k \in [K]$, $\delta > 0$}
\KwOut{Ordering quantity $\qalg$}
\tcp*[h]{Conduct multiple hypothesis test of identifiability}

For all $k \in [K]$, compute censored SAA of $\Gminus(\qoff_k)$:
\begin{equation}\label{eq:gminus-qoff}
{\Gminushatk{k}(\qoff_k) = \frac{1}{N_k} \sum_{i \in [N_k]} \Ind{\soff_{ki} < \qoff_k}.}
\end{equation}
\noindent {Construct ``likely identifiable'' set $\Uest$, defined as:
\begin{equation}
\Uest = \left\{k \in [K] \mid \Gminushatk{k}(\qoff_k) \geq \rho + \width_k\right\}
\end{equation}}
\uIf(\tcp*[h]{Likely identifiable}){$\Uest \neq \emptyset$}{
Compute censored SAA of $\qopt_G$ using all samples associated with $\Uest$:
{
\begin{align}
\qalg = \inf \left\{x \mid \frac{1}{\sum_{k\in \Uest} N_k}\sum_{k\in\Uest}\sum_{i\in[N_k]}\Ind{\soff_{ki} \leq x} \geq \rho \right\}.
\end{align}
}
}
{
\uElseIf(\tcp*[h]{Likely unidentifiable}){$\Gminushatk{K}(\lambda) < \rho-\width_K$}{
Compute empirical estimate of $\qcrit_G$:
\begin{align}
\qalg= \frac{bM+h\lambda-(b+h)\Gminushatk{K}(\lambda)M}{(b+h)(1-\Gminushatk{K}(\lambda))}.
\end{align}
}
}
\uElse(\tcp*[h]{Knife-edge case}){$\qalg = \lambda$.
}
\Return{$\qalg$}
}
	\caption{\textsf{Robust Censored Newsvendor$^+$} (\ALGplus)}
	\label{alg:newsvendor_plus_plus}
\end{algorithm}

\subsubsection{Results.} In all experiments we let $M_p = 2.5\lambda_p$ for all $p \in \mathcal{P}$. We moreover let $h = 1$ and vary $b \in \{3,9,49\}$. For these respective values of $b$, 83\%, 66\%, and 49\% of products are identifiable. \Cref{fig:delta-hist} shows the distribution of minimax risk $\Delta$ across all unidentifiable products for these three values of $b$. As expected, this distribution depends significantly on $b$, with the average value of $\Delta$ increasing from 0.87, to 1.7, to 3.05 as $b$ increases. We moreover observe a wide range of $\Delta$ for $b = 49$, with $\Delta = 56$ in the worst case. 

\begin{figure}[t]
\centering
\includegraphics[scale=0.4]{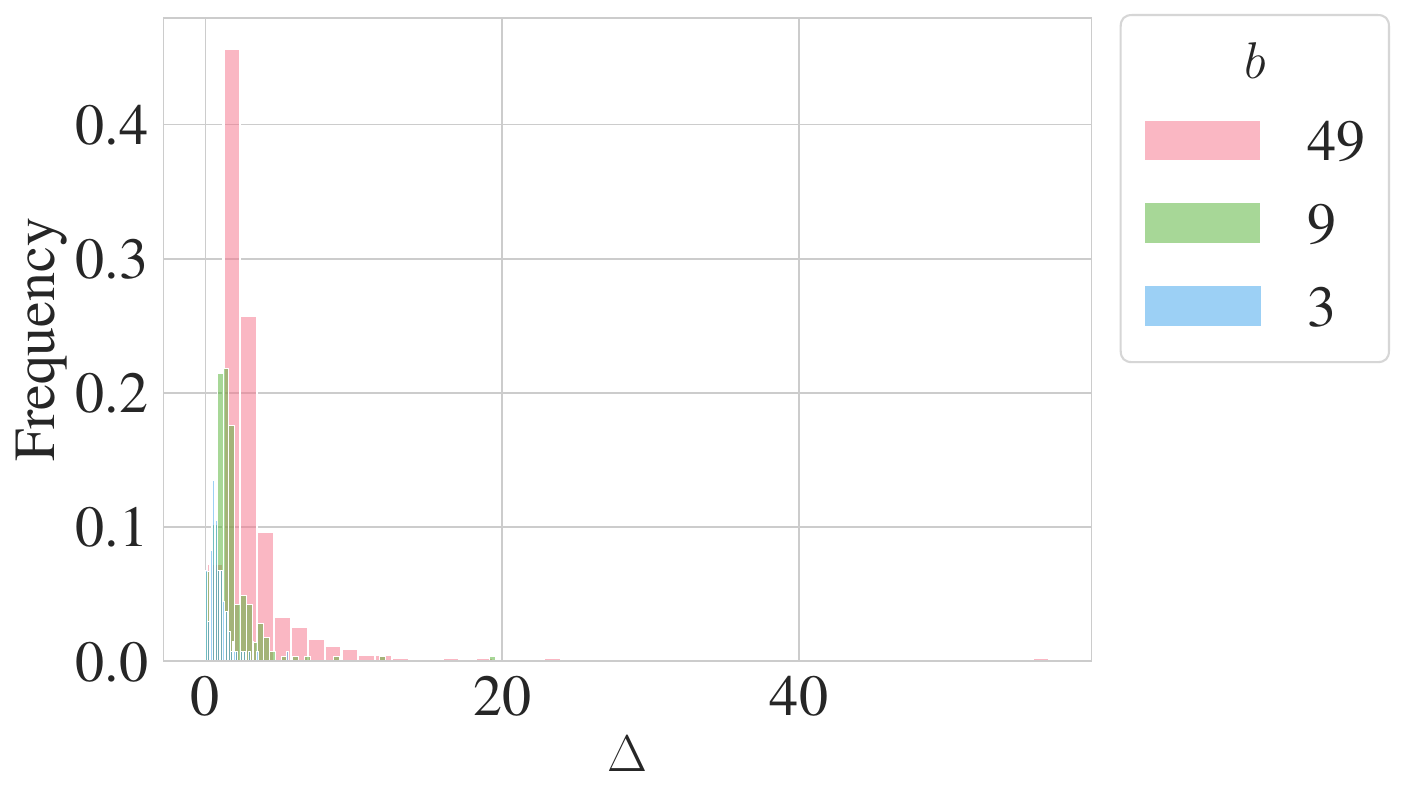}
\caption{Distribution of minimax risk $\risk$ for $b \in \{3,9,49\}$.}\label{fig:delta-hist}
\end{figure}

In \Cref{fig:real-results} we plot $\Regret(q)-\Delta$ for all four algorithms and each value of $b$, across both identifiable and unidentifiable instances. (Recall, since $\Delta = 0$ in the identifiable regime, this reduces to reporting the worst-case regret $\Regret(q)$ for these instances.)
As observed in our synthetic experiments, in the identifiable regime (\Cref{fig:real-id-reg}) \CensoredSAA and \KM are the best-performing algorithms,  with \mbox{$\Regret(q) \leq 1.3$} on average, across all values of $b$. 
\minedit{\ALG, however, achieves slightly higher regret, with \mbox{$\Regret(q) \in \{3.43,4.86,6.36\}$} on average.} The reason for this is two-fold. On the one hand, if \ALG classifies the instance as likely identifiable, it only uses samples at the boundary, of which there are very few; this results in a higher estimation error of $\qopt_G$. On the other hand, \ALG struggles to determine identifiability with confidence as $b$ grows larger and the gap between $\rho$ and $\Gminus(\lambda)$ decreases; this results in more misclassification errors, where it estimates either $\qcrit_G$ or simply outputs $\lambda$. 
\minedit{
Our results show that \ALGplus mitigates these issues to a certain extent, with $\Regret(q) \in \{2.94,4.81,6.36\}$ on average.  For small values of $b$, \ALGplus improves over \ALG since more order levels $k$ satisfy $\Gminushatk{k}(\qoff_k) \geq \rho + \width_k$, and are therefore used to compute the SAA estimate. These benefits decrease with larger values of $b$, due to a higher misclassification rate.  (We highlight that, in general, this misclassification rate is the reason that even \ALGplus cannot achieve the same performance as \CensoredSAA and \KM.)
}

\Cref{fig:real-uid-reg}, on the other hand, demonstrates that both of our algorithms generate substantial gains relative to non-robust algorithms {in the unidentifiable regime}, as the lost-sales penalty grows large. In the worst case ($b = 49$), \CensoredSAA and \KM incur $\text{Regret}(q)-\Delta \geq 26$ on average. 
\minedit{
In contrast, \ALG and \ALGplus --- whose behavior coincides in this regimes --- incur $\text{Regret}(q)-\Delta$ of 1.89.
}
Finally, we observe that when $b$ is small, even the non-robust algorithms incur low regret. The reason for this is that, even though $\lambda = 1.68 < \qcrit_G = 2.55$ on average, lost sales are not as costly. Hence, this regime is ``easy'' for any reasonable (non-robust) algorithm.

\begin{figure}[t]
 \centering
 \begin{subfigure}[b]{0.45\textwidth}
\includegraphics[width=\textwidth]{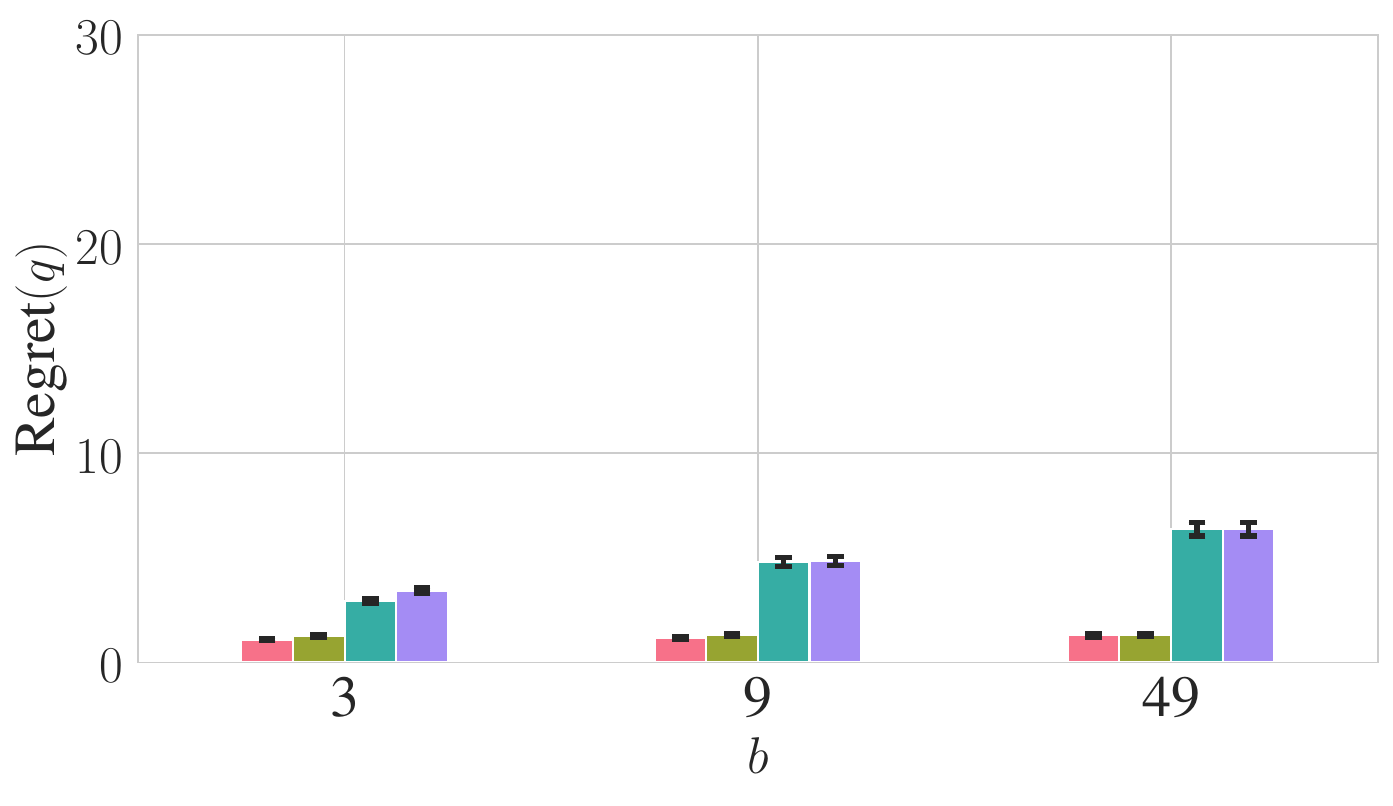}
\caption{\small Identifiable products}\label{fig:real-id-reg}
\end{subfigure}
\begin{subfigure}[b]{0.45\textwidth}
\centering
\includegraphics[width=\textwidth]{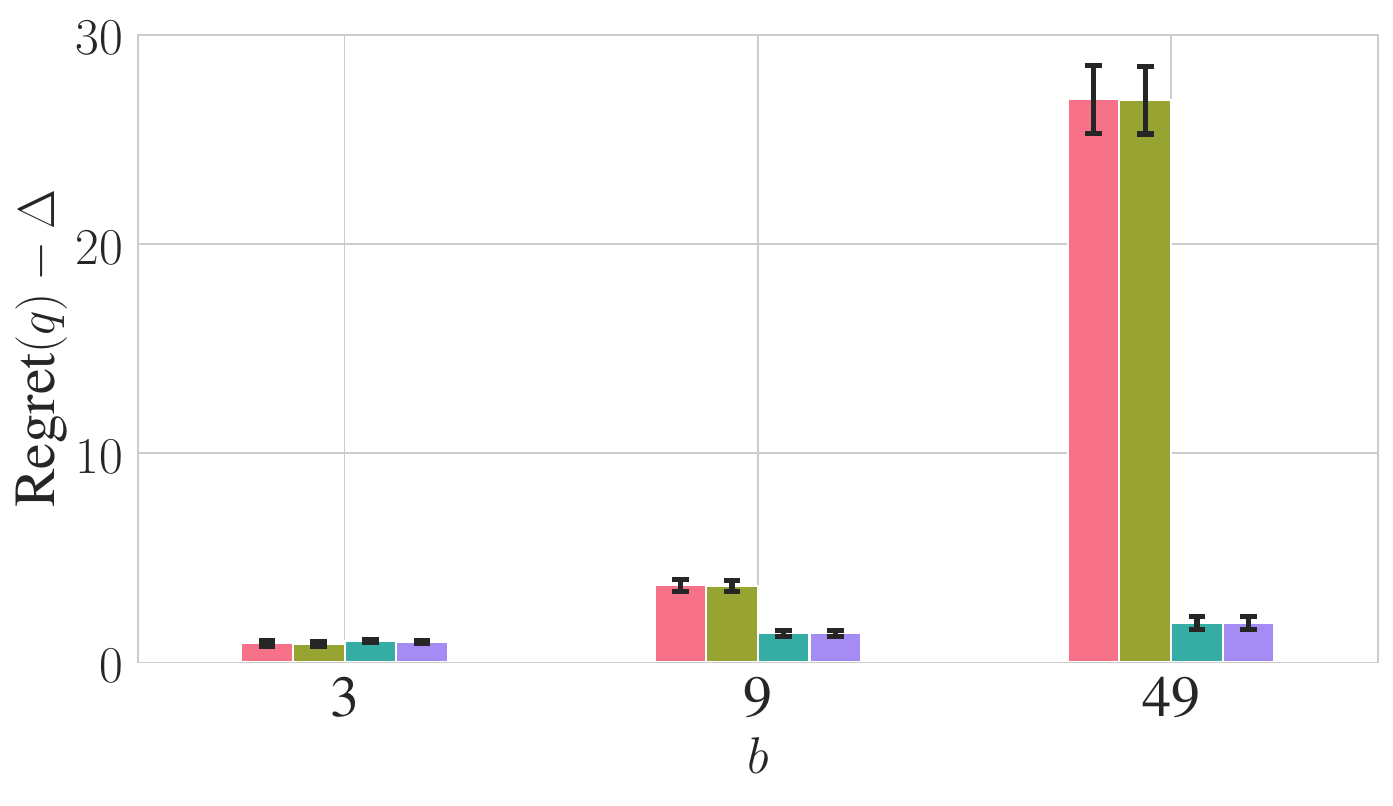}
\caption{\small Unidentifiable products}\label{fig:real-uid-reg}
\end{subfigure}
\begin{subfigure}[b]{\textwidth}
\centering
\includegraphics[width=0.7\textwidth]{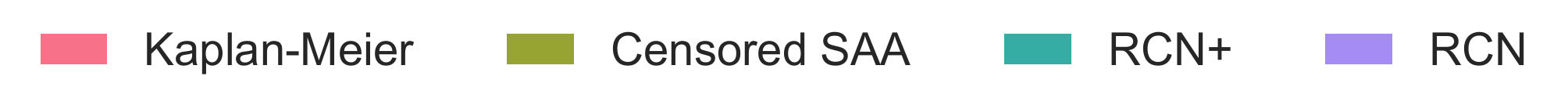}
\end{subfigure}
\caption{$\text{Regret}(q)-\Delta$ as a function of $b$, for both identifiable and unidentifiable products.  \minedit{The black lines correspond to a 95\% confidence interval.}}\label{fig:real-results}
\end{figure}

\minedit{Finally, we explore the robustness of our algorithms to historical out-of-stock frequency, a practical and interpretable proxy for regime difficulty. Specifically, in \Cref{tab:robustness-to-oos} we report the Pearson correlation, computed across products, between $\text{Regret}(q)-\Delta$ and the historical out-of-stock frequency.}  
\minedit{
Our findings corroborate \Cref{fig:real-results}. Namely, when $b = 3$, all algorithms perform well regardless of whether the product is identifiable; hence, there is no significant correlation between out-of-stock frequency and performance. However, for $b \in \{9, 49\}$, the regret of \KM and \CensoredSAA is positively correlated with out-of-stock frequency, whereas that of \ALG and \ALGplus is negatively correlated.}  This is due to the \minedit{decrease} in $\text{Regret}(q)-\risk$ observed between \Cref{fig:real-id-reg} and \Cref{fig:real-uid-reg}.

\begin{table}[t]
    \centering
    \minedit{
    \begin{tabular}{lcccc}
        \toprule
        \textbf{$b$} & \KM & \CensoredSAA & \ALG & \ALGplus \\
        \midrule
        3  & 0.038       & -0.028       & -0.050       & 0.052 \\
        9  & 0.196 ($^{**}$)  & 0.171 ($^{**}$)   & -0.082 ($^{**}$)  & -0.071 ($^{**}$) \\
        49 & 0.257 ($^{**}$)  & 0.257 ($^{**}$)   & -0.167($^{**}$)       & -0.167 ($^{**}$) \\
        \bottomrule
    \end{tabular}
    }
        \caption{\minedit{Sample correlation} between $\text{Regret}(q)-\risk$ and historical out-of-stock frequency. Here, $(^{**})$ represents statistically significant results ($p$-value $<$ 0.05).}
    \label{tab:robustness-to-oos}
\end{table}

\section{Conclusion}\label{sec:conclusion}

In this work we considered the offline data-driven censored newsvendor problem. Specifically, we leveraged the distributionally robust optimization framework to isolate the impact of demand censoring on newsvendor decision-making. Through this framework we established crisp characterizations of the information loss due to censoring, and used these to design a robust algorithm with near-optimal performance guarantees across all possible demand censoring levels. Our algorithm is practical and intuitive, and demonstrates strong numerical performance on both synthetic and real-world datasets.

Several future directions emerge from our work. Most immediately, we conjecture that all of our analytical and algorithmic techniques go through under common assumptions on the demand distribution, as noted in \Cref{remark:well-sep}.  It would moreover be interesting to see if our framework and insights port over to the  {\it feature-based} setting, in which the decision-maker also has access to uncensored contextual information that may help mitigate the effect of demand censoring.  {Lastly, it would be interesting to extend our modeling framework to the setting with inventory carry-over, where demand may be endogenous to previous inventory levels.\footnote{This is frequently the case, for instance, for retailers who offer grocery delivery, and who can nudge customers to choose products with higher inventory levels in order to avoid stockouts~\citep{knight2025disclosing}.}}

\medskip

\ACKNOWLEDGMENT{The authors gratefully acknowledge the department editor, associate editor, and three anonymous referees whose comments and suggestions were tremendously helpful in improving the paper. The authors moreover thank Vishal Gupta, Qi Luo, Will Ma, Omar Mouchtaki, Bahar Taskesen, and Linwei Xin for their helpful comments on a preliminary version of this manuscript.}

\newpage

\bibliographystyle{informs2014} 
\bibliography{references} 


%
%
%

\newpage

\AtBeginEnvironment{APPENDICES}{%
  \renewcommand{\theHsection}{appendix.\Alph{section}}%
  \renewcommand{\theHsubsection}{\theHsection.\arabic{subsection}}%
  \renewcommand{\theHsubsubsection}{\theHsubsection.\arabic{subsubsection}}%
  \renewcommand{\theHfigure}{\theHsection.\arabic{figure}}%
  \renewcommand{\theHtable}{\theHsection.\arabic{table}}%
  \renewcommand{\theHequation}{\theHsection.\arabic{equation}}%
  \renewcommand{\theHtheorem}{\theHsection.\arabic{theorem}}%

  \crefalias{section}{appendix}%
  \crefalias{subsection}{appendix}%
  \crefalias{subsubsection}{appendix}%
}

\crefalias{section}{appendix}
\begin{APPENDICES}
\OneAndAHalfSpacedXI 
\section{Useful Facts}\label{apx:useful-facts}

\subsection{Closed-form expression for difference in newsvendor costs}

We repeatedly rely on the following closed-form expression for the difference between the newsvendor cost under two arbitrary ordering quantities $q_1$ and $q_2$.  
\begin{proposition}
\label{lem:cost_difference_robustness}
For any distribution $F\in\F$, $q_1, q_2 \in [0,M]$:
\begin{equation}\label{eq:cost_difference_robustness}
C_F(q_1) - C_F(q_2) = b(q_2 - q_1) + (b+h)\E_F\left[(q_1-D)\Ind{D \leq q_1} - (q_2-D)\Ind{D \leq q_2}\right].
\end{equation}
\end{proposition}

\begin{proof}{\it Proof.}
By definition, for any $q$:
\begin{align}
C_F(q) & = b\E_F[(D-q)^+] + h\E_F[(q-D)^+] \nonumber \\
& = b\E_F[(D-q)\Ind{D > q}] + h\E_F[(q-D)\Ind{D \leq q}] \nonumber \\
& = b\E_F[(D-q)(1 - \Ind{D \leq q})] + h\E_F[(q-D)\Ind{D \leq q}] \nonumber \\
& = b\E_F[D-q] - b\E_F[(D-q)\Ind{D \leq q}] + h\E_F[(q-D)\Ind{D \leq q}] \nonumber \\
& = b\E_F[D-q] + (b+h)\E_F[(q-D)\Ind{D \leq q}]. \label{eq:cost_alternative}
\end{align}
Applying  \eqref{eq:cost_alternative} to $q_1$ and $q_2$, we obtain:
\begin{align*}
    C_F(q_1) - C_F(q_2) & = b\E_F[D-q_1] - b\E_F[D - q_2] + (b+h)\E_F[(q_1-D)\Ind{D \leq q_1}] \\
    &\qquad- (b+h)\E_F[(q_2-D)\Ind{D \leq q_2}] \\
    & = b(q_2 - q_1) + (b+h)\E_F[(q_1-D)\Ind{D \leq q_1} - (q_2-D)\Ind{D \leq q_2}].
\end{align*}
\hfill\Halmos
\end{proof}

\subsection{On the relationship between $\lambda$ and $M$}\label{apx:lam-m}

We implicitly rely on the following fact in all of our results.

\begin{proposition}\label{prop:lam-m}
Suppose $\Gminus(\lambda) < \rho$. Then, $\lambda \leq \qopt_G \leq M$.
\end{proposition}

\begin{proof}{\it Proof.}
We argue that $\Gminus(\lambda) < \rho$ implies $\lambda \leq \qopt_G$ via the converse. Namely, suppose $\lambda > \qopt_G$. Recall, by definition of $\qopt_G$, $G(\qopt_G) \geq \rho$, which implies there exists $q < \lambda$ such that $G(q) \geq \rho$, i.e., $\Gminus(\lambda) \geq \rho$. Note that $\qopt_G \leq M$ by assumption, and we obtain the claim.
\hfill\Halmos
\end{proof}

\newpage

\section{Omitted Proofs}

\subsection{\Cref{sec:minimax-risk} Omitted Proofs}\label{apx:minimax-proofs}

\subsubsection{Proof of \Cref{prop:iff}}\label{apx:iff}

\begin{proof}{\it Proof.}
Suppose first that $\risk > 0$. For any policy $\pi$,
$\sup_{F \in \ambig{\lambda}} \left\{ C_F(q^\pi) - C_F(\qopt_F)\right\} \geq \risk > 0$, by definition of $\risk$. Therefore, for all $\epsilon \in (0,\risk)$, $\sup_{F \in \ambig{\lambda}} \Big\{ C_F(q^\pi) - C_F(\qopt_F)\Big\} > \epsilon$, which implies that the problem is unidentifiable.

The proof of the converse is a corollary of our main algorithmic result, in which we design a policy $\pi$ such that $\sup_{F \in \ambig{\lambda}} \Big\{ C_F(q^\pi) - C_F(\qopt_F)\Big\} = \widetilde{O}(1/\sqrt{N})$ {with probability at least $1-O(1/\sqrt{N})$} when $\risk = 0$ (see \Cref{thm:minmax_regret}). Taking $N \to \infty$ then implies that the problem is identifiable. \hfill\Halmos 
\end{proof}

\medskip

\subsubsection{Proof of \Cref{lem:identifiable_same_opt_quantile}}\label{apx:identifiable_same_opt_quantile}

\begin{proof}{\it Proof.}
By definition,
$\qopt_G =  \inf\{q: G(q) \geq \rho\}.$ If $\Gminus(\lambda) \geq \rho$, there must exist $q < \lambda$ such that $G(q) \geq \rho$, which implies $\qopt_G < \lambda$.

Consider now $\qopt_F$. Again, by definition,
\begin{align*}
\qopt_F &= \inf\{q: F(q) \geq \rho\} \\
&= \inf\left\{q: G(q) \mathds{1}\{q < \lambda\} + (\Pr_G(D < \lambda) + \Pr_F(\lambda \leq D < q))\mathds{1}\{q \geq \lambda\} \geq \rho\right\},
\end{align*}
where the second equality uses the fact that $F\in\ambig{\lambda}$, hence $F(x) = G(x) \ \forall \ x < \lambda$. Again, using the fact that $\Gminus(\lambda) \geq \rho$, there must be $q < \lambda$ such that $G(q) \geq \rho$. Hence, the above infimum is necessarily achieved at $\qopt_G$.
\hfill\Halmos
\end{proof}

\medskip

\subsubsection{Proof of \Cref{lem:sup_expression_identifiability}}\label{apx:sup_expression_identifiability}
\begin{proof}{\it Proof.}
By \cref{lem:cost_difference_robustness}, we have:
\begin{align}\label{eq:worst-case-rewrite}
&\sup_{F \in \ambig{\lambda}} C_F(q) - C_F(\qopt_F) \notag \\ &= \sup_{\tilde{q}\in[\lambda,M]}\sup_{\substack{F\in \ambig{\lambda}:\\\qopt_F=\tilde{q}}}b(\tilde{q} - q) + (b+h)\E_F\left[(q-D)\Ind{D \leq q} - (\tilde{q}-D)\Ind{D \leq \tilde{q}}\right] \notag \\
&=\sup_{\tilde{q}\in[\lambda,M]}b(\tilde{q} - q) + (b+h)\left(\sup_{\substack{F\in \ambig{\lambda}:\\
\qopt_F=\tilde{q}}} \E_F\left[(q-D)\Ind{D \leq q} - (\tilde{q}-D)\Ind{D \leq \tilde{q}}\right]\right).
\end{align}
Here, the supremum is over all distributions $F \in \ambig{\lambda}$ such that {$\tilde{q} = \qopt_F \geq \lambda$}, since \mbox{$\qopt_F < \lambda \implies F^{-}(\lambda) \geq \rho.$} However, since $F \in \ambig{\lambda}$, $F^{-}(\lambda) = \Gminus(\lambda) < \rho$, a contradiction.

We analyze \eqref{eq:worst-case-rewrite} depending on the following {three} cases: (i) $q > \tilde{q}$, (ii) $q < \tilde{q}$, and (iii) $q = \tilde{q}$.
\begin{enumerate}[(i)]
\item $q > \tilde{q}$: In this case, the worst-case distribution $F$ that has $\tilde{q}$ as its newsvendor solution maximizes:
\begin{align}\label{eq:worst-case-1}
\E_F\left[(q-D)\Ind{D \leq q} - (\tilde{q}-D)\Ind{D \leq \tilde{q}}\right] = \E_F\left[(q-\tilde{q})\Ind{D \leq \tilde{q}} + ({q}-D)\Ind{\tilde{q} <D \leq q}\right]. 
\end{align}
Observe that $q-D < q-\tilde{q}$ for all $D \in (\tilde{q},q]$. Hence, for any distribution $F$ such that \mbox{$\Pr_F(D \in (\tilde{q},q]) > 0$}, there exists a strictly worse distribution that sets $\Pr_F(D \in (\tilde{q},q]) = 0$, moving all of the remaining mass to $\tilde{q}$. (Note that such a distribution will still have $\qopt_F = \tilde{q}$.) For any such $F$, then,
$\eqref{eq:worst-case-1} = (q-\tilde{q})\Pr_F(D \leq \tilde{q}).$
Using this in \eqref{eq:worst-case-rewrite}:
\begin{align}\label{eq:worst-case-1a}
&\sup_{\tilde{q}\in[\lambda,q)}b(\tilde{q}-q) + (b+h)\left(\sup_{\substack{F\in \ambig{\lambda}:\\\qopt_F=\tilde{q}}}\E_F\left[(q-D)\Ind{D \leq q} - (\tilde{q}-D)\Ind{D \leq \tilde{q}}\right]\right)\notag \\
&=
\sup_{\tilde{q}\in[\lambda,q)}b(\tilde{q}-q)+(b+h)(q-\tilde{q})\left(\sup_{\substack{F\in \ambig{\lambda}:\\\qopt_F=\tilde{q}}}\Pr_F(D\leq \tilde{q})\right)\notag \\
&=\sup_{\tilde{q}\in[\lambda,q)}b(\tilde{q}-q)+(b+h)(q-\tilde{q}),
\end{align}
achieved by letting $F$ be such that $\Pr_F(D\leq \tilde{q}) = 1$, since $q > \tilde{q}$ by assumption. Further simplifying, we obtain:
\begin{align}\label{eq:worst-case-1a-a}
\eqref{eq:worst-case-1a} &= \sup_{\tilde{q}\in[\lambda,q)} h(q-\tilde{q}) =h(q-\lambda).
\end{align}
\item $q < \tilde{q}$: Suppose first that $q < \lambda$. In this case, for a fixed $\tilde{q}$:
\begin{align}
\label{eq:worst_case_other_setting}
&\sup_{\substack{F\in\ambig{\lambda}:\\\qopt_F=\tilde{q}}} \E_F\left[(q-D)\Ind{D \leq q} - (\tilde{q}-D)\Ind{D \leq \tilde{q}}\right] \notag \\ &=\sup_{\substack{F\in\ambig{\lambda}:\\\qopt_F=\tilde{q}}} \E_G\left[(q-D)\Ind{D \leq q}\right] - \E_G\left[(\tilde{q}-D)\Ind{D < \lambda}\right] - \E_F\left[(\tilde{q}-D)\Ind{\lambda \leq D < \tilde{q}}\right].
\end{align}
Since $\tilde{q}-D > 0$ for all $D \in [\lambda,\tilde{q})$, a similar argument as above establishes that any worst-case distribution $F \in \ambig{\lambda}$ necessarily sets $\Pr_F(D \in [\lambda,\tilde{q})) = 0$. Using this fact above, we obtain:
\begin{align*}
&\sup_{\substack{F\in\ambig{\lambda}:\\\qopt_F=\tilde{q}}} \E_F\left[(q-D)\Ind{D \leq q} - (\tilde{q}-D)\Ind{D \leq \tilde{q}}\right] \\&= \E_G[(q-D)\Ind{D \leq q}] - \E_G\left[(\tilde{q}-D)\Ind{D < \lambda}\right].
\end{align*}
Further using this in \eqref{eq:worst-case-rewrite}, we then have:
\begin{align}\label{eq:worst-case-q<q-a}
&\sup_{\tilde{q}\in[\lambda,M]}b(\tilde{q} - q) + (b+h)\left(\sup_{\substack{F\in \ambig{\lambda}:\\\qopt_F=\tilde{q}}} \E_F\left[(q-D)\Ind{D \leq q} - (\tilde{q}-D)\Ind{D \leq \tilde{q}}\right]\right)\notag \\
&=\sup_{\tilde{q}\in[\lambda,M]}b(\tilde{q} - q) + (b+h)\left(\E_G[(q-D)\Ind{D \leq q}] - \E_G\left[(\tilde{q}-D)\Ind{D < \lambda}\right]\right) \notag \\
&= -bq+(b+h)\left(\E_G[(q-D)\Ind{D \leq q}]+\E_G[D\Ind{D < \lambda}]\right) \notag \\
&\quad+\sup_{\tilde{q}\in[\lambda,M]}\bigg\{\tilde{q}(b- (b+h)\Gminus(\lambda))\bigg\},
\end{align}
where the supremum is attained at $\tilde{q} = M$, since $\Gminus(\lambda) < \rho \implies b-(b+h)\Gminus(\lambda) > 0$. Therefore:
\begin{align}\label{eq:worst-case-q<q-b}
\eqref{eq:worst-case-q<q-a}&=b(M-q) + (b+h)\left(\E_G[(q-D)\Ind{D \leq q}] - \E_G\left[(M-D)\Ind{D < \lambda}\right]\right).
\end{align}

Now, if $q \geq \lambda$, for a fixed $\tilde{q}$ we have:
\begin{align}
\label{eq:last_case_for_worst_case_distribution}
&\E_F\left[(q-D)\Ind{D \leq q} - (\tilde{q}-D)\Ind{D \leq \tilde{q}}\right]\notag \\ &= \E_F\left[(q-\tilde{q})\Ind{D \leq q} - (\tilde{q}-D)\Ind{{q} <D \leq \tilde{q}}\right] \notag \\
&= \E_F\left[-(\tilde{q}-{q})\Ind{D \leq q} - (\tilde{q}-D)\Ind{{q} <D \leq \tilde{q}}\right] \notag \\
&=\E_F\left[-(\tilde{q}-{q})\Ind{D < \lambda} - (\tilde{q}-{q})\Ind{\lambda \leq D \leq q} - (\tilde{q}-D)\Ind{{q} <D \leq \tilde{q}}\right] \notag \\
&= -(\tilde{q}-q)\Gminus(\lambda)-(\tilde{q}-q)\Pr_F(\lambda\leq D \leq q)-\E_F\bigg[(\tilde{q}-D)\Ind{{q} <D \leq \tilde{q}}\bigg].
\end{align}
Since $0 < \tilde{q}-D < \tilde{q}-q$ for all $D \in (q,\tilde{q})$, the supremum of the above is achieved for $F$ such that  \mbox{$\Pr_F(\lambda \leq D < \tilde{q}) = 0$}. Then,
\begin{align}\label{eq:worst-case-2}
&\sup_{\tilde{q}\in[q,M]}b(\tilde{q} - q) + (b+h)\left(\sup_{\substack{F\in \ambig{\lambda}:\\\qopt_F=\tilde{q}}} \E_F\left[(q-D)\Ind{D \leq q} - (\tilde{q}-D)\Ind{D \leq \tilde{q}}\right]\right)\notag \\
&=\sup_{\tilde{q}\in[q,M]}b(\tilde{q} - q) - (b+h)(\tilde{q}-q)\Gminus(\lambda) \notag \\&=\sup_{\tilde{q}\in[q,M]} (\tilde{q}-q)\left(b-(b+h)\Gminus(\lambda)\right) \notag \\
&= (M-q)(b-(b+h)\Gminus(\lambda)),
\end{align}
again, since $\Gminus(\lambda) < \rho$ by assumption.
\item $q = \tilde{q}$: In this case, $C_F(q)-C_F(\tilde{q}) = 0$ for all $F \in \ambig{\lambda}$.
\end{enumerate}

Putting these three cases together and applying them to \eqref{eq:worst-case-rewrite}, we have:
\begin{align*}
&\sup_{F\in\ambig{\lambda}}C_F(q)-C_F(\qopt_F)\\&=\begin{cases}b(M-q) + (b+h)\left(\E_G[(q-D)\Ind{D \leq q}] - \E_G\left[(M-D)\Ind{D < \lambda}\right]\right) \quad &\text{if } q <\lambda\\
\max\left\{h(q-\lambda),\left(b-(b+h)\Gminus(\lambda)\right)(M-q)\right\} \quad &\text{if } q \geq \lambda.
\end{cases}
\end{align*}
Noting that $h(q-\lambda) > (b-(b+h)\Gminus(\lambda))(M-q) \iff q > \qcrit_G$, with equality at $\qcrit_G$, we obtain the result.\hfill\Halmos
\end{proof}

\medskip

\subsubsection{Proof of \Cref{thm:minimax-risk-identifiable}}\label{apx:minimax-risk-identifiable}

\begin{proof}{\it Proof.} We proceed separately for each regime.

\paragraph{Case I: $\Gminus(\lambda) \geq \rho$.}
By \Cref{lem:identifiable_same_opt_quantile}, $\qopt_F = \qopt_G < \lambda$ for all $F \in \ambig{\lambda}$. Therefore, for all $F\in\ambig{\lambda}$, $q \in [0,M]$:
\begin{align*}
C_F(q) - C_F(\qopt_F) &= C_F(q)-C_F(\qopt_G).
\end{align*}
Letting $q = \qopt_G$, this implies that $C_F(\qopt_G)-C_F(\qopt_F) = 0$ for all $F \in \ambig{\lambda}$. Using the fact that $\sup_{F\in\ambig{\lambda}}C_F(q)-C_F(\qopt_F) \geq 0$ for all $q$ by definition of $\qopt_F$, we then have \mbox{$\risk := \inf_{q} \sup_{F\in\ambig{\lambda}}C_F(q)-C_F(\qopt_F) = 0$}, precisely achieved at $\qrisk = \qopt_G$. 

\medskip

\paragraph{Case II: $\Gminus(\lambda) < \rho$.} 
Recall, $\qcrit_G := \frac{bM+h\lambda-(b+h)\Gminus(\lambda)M}{(b+h)(1-\Gminus(\lambda))}$ by definition. We analyze \mbox{$\inf_{q \in [0,M]}\sup_{F \in \ambig{\lambda}} C_F(q) - C_F(\qopt_F)$} by partitioning the proof into three cases: (i) $q < \lambda$, (ii) $q \in [\lambda,\qcrit_G]$, and (iii) $q > \qcrit_G$.

By \Cref{lem:sup_expression_identifiability}, for $q < \lambda$:
\[\sup_{F \in \ambig{\lambda}} C_F(q) - C_F(\qopt_F)
    = 
        b (M-q) + (b+h)\bigg[\E_G\Big[(q-D)\Ind{D \leq q}] - (M-D)\mathds{1}\{D < \lambda\}\Big]\bigg].\]
Hence, we have:
\begin{align}
&\inf_{q < \lambda} \sup_{F \in \ambig{\lambda}} C_F(q) - C_F(\qopt_F) \notag\\
&= \inf_{q < \lambda} bM - (b+h)\E_G\left[(M-D)\Ind{D < \lambda}\right] -bq+(b+h)\E_G\left[(q-D)\Ind{D \leq q}\right],\label{eq:inf} 
\end{align}
where \eqref{eq:inf} uses the fact that $F\in\ambig{\lambda}$ for the first expectation.  
Now, the newsvendor cost can equivalently be written as 
\begin{align}\label{eq:alt-newsvendor}
C_G(q) = b\left(\E_G[D]-q\right) + (b+h)\E_G[(q-D)\Ind{D \leq q}].
\end{align}
Plugging this formulation into \eqref{eq:inf}, we obtain:
\begin{align}
\inf_{q < \lambda} \sup_{F \in \ambig{\lambda}} C_F(q) - C_F(\qopt_F) &=bM - (b+h)\E_G\left[(M-D)\Ind{D < \lambda}\right]-b\E_G[D] +\inf_{q<\lambda} C_G(q).
\end{align}

Recall, $\Gminus(\lambda) < \rho$ implies $ \qopt_G \geq \lambda$. Hence, by convexity of $C_G(q)$, $C_G(\lambda) \leq C_G(q)$ for all $q < \lambda$. Then:
\begin{align*}
&\inf_{q < \lambda} \sup_{F \in \ambig{\lambda}} C_F(q) - C_F(\qopt_F) \\ &= bM -(b+h)\E_G\left[(M-D)\mathds{1}\{D < \lambda\}\right] - b\E_G[D] + C_G(\lambda) \\
&= bM -(b+h)\E_G\left[(M-D)\mathds{1}\{D < \lambda\}\right] - b\lambda + (b+h)\E_G[(\lambda-D)\Ind{D\leq\lambda}],
\end{align*}
where the second equality follows from \eqref{eq:alt-newsvendor}. Simplifying, we finally obtain:
\begin{align}\label{eq:inf-before-lam}
\inf_{q < \lambda} \sup_{F \in \ambig{\lambda}} C_F(q) - C_F(\qopt_F) &=  b (M-\lambda) - (b+h)\E_G\Big[(M-\lambda)\Ind{D < \lambda}\Big]\notag\\
&= (M-\lambda)\left(b-(b+h)\Gminus(\lambda)\right).
\end{align}

Suppose now $q \in \left[\lambda,\qcrit_G \right]$.  Again, by \Cref{lem:sup_expression_identifiability}:
\begin{align}\label{eq:before-crit}
     \min_{q \in [\lambda, \qcrit_G]} \sup_{F \in \ambig{\lambda}} C_F(q) - C_F(\qopt_F) &= \min_{q \in [\lambda, \qcrit_G]} (M-q)\left(b - (b+h)\Gminus(\lambda)\right) \notag \\
     &=(M-\qcrit_G)\left(b-(b+h)\Gminus(\lambda)\right),
\end{align}
where the final equality follows from the fact that the right-hand side of \eqref{eq:before-crit} is decreasing in $q$, since $\Gminus(\lambda) < \rho$. It is easy to derive that $\qcrit_G \geq \lambda$ for $\Gminus(\lambda) < \rho$. Comparing \eqref{eq:inf-before-lam} and \eqref{eq:before-crit}, this implies:
\begin{align}
     \min_{q \in [\lambda, \qcrit_G]} \sup_{F \in \ambig{\lambda}} C_F(q) - C_F(\qopt_F) \leq \inf_{q < \lambda} \sup_{F \in \ambig{\lambda}} C_F(q) - C_F(\qopt_F)\notag.
\end{align}
Hence, setting $q = \qcrit_G$ (weakly) dominates setting $q < \lambda$.

We conclude by considering $q \in (\qcrit_G, M]$.  By \Cref{lem:sup_expression_identifiability}:
\begin{align*}
     \inf_{q \in (\qcrit_G, M]} \sup_{F \in \ambig{\lambda}} C_F(q) - C_F(\qopt_F) &= \inf_{q \in (\qcrit_G, M]} h(q - \lambda) = h(\qcrit_G-\lambda).
\end{align*}
By definition of $\qcrit_G$,
\[h(\qcrit_G-\lambda) = (M-\qcrit_G)(b-(b+h)\Gminus(\lambda)) = \frac{h(b-(b+h)\Gminus(\lambda)(M-\lambda)}{(b+h)(1-\Gminus(\lambda))}.\]
Putting this together with \eqref{eq:before-crit}, we conclude that $\qrisk = \qcrit_G$, with
\[\Delta = \frac{h(b-(b+h)\Gminus(\lambda)(M-\lambda)}{(b+h)(1-\Gminus(\lambda))}.\]
\hfill\Halmos
\end{proof}

\medskip

\subsubsection{Proof of \Cref{prop:worst-case-for-identifiable}}\label{apx:worst-case-id}

\begin{proof}{\it Proof.}
Recall that, in this regime, by \Cref{lem:identifiable_same_opt_quantile}, $\qopt_F = \qopt_G < \lambda$ for all $F\in\ambig{\lambda}$. Applying this fact to \Cref{lem:cost_difference_robustness}, for any $F \in \ambig{\lambda}$ we have:
\begin{align}\label{eq:id-step-1}
C_F(q) - C_F(\qopt_F) &= b(\qopt_G - q) + (b+h)\E_F\left[(q-D)\Ind{D \leq q} - (\qopt_G-D)\Ind{D \leq \qopt_G}\right].
\end{align}

If $q < \lambda$, since $F \in \ambig{\lambda}$ and $\qopt_F < \lambda$, the above evaluates to:
\begin{align*}
C_F(q) - C_F(\qopt_F) &=b(\qopt_G - q) + (b+h)\E_G\left[(q-D)\Ind{D \leq q} - (\qopt_G-D)\Ind{D \leq \qopt_G}\right]\\
&= C_G(q)-C_G(\qopt_G),
\end{align*}
which proves the first case.

\smallskip

Suppose now that $q \geq \lambda$. Note that only the second term of \eqref{eq:id-step-1} has dependence on $F$. Again, since $\qopt_G < \lambda$, for all $F \in \ambig{\lambda}$ we have:
\begin{align}\label{eq:id-step-2}
&\E_F\left[(q-D)\Ind{D \leq q} - (\qopt_G-D)\Ind{D \leq \qopt_G}\right] \notag \\
&= \E_G\left[(q-D)\Ind{D < \lambda}\right] + \E_F\left[(q-D)\Ind{\lambda \leq D < q}\right] - \E_G\left[ (\qopt_G-D)\Ind{D \leq \qopt_G}\right] \notag \\
&\leq \E_G\left[(q-D)\Ind{D < \lambda}\right] + (q-\lambda)\Pr_F(D = \lambda) - \E_G\left[ (\qopt_G-D)\Ind{D \leq \qopt_G}\right],
\end{align}
where the inequality is attained by the distribution $F$ that places all remaining mass on $\lambda$. Re-arranging \eqref{eq:id-step-2}, we obtain:
\begin{align*}
&\sup_{F\in\ambig{\lambda}} \E_F\left[(q-D)\Ind{D \leq q} - (\qopt_G-D)\Ind{D \leq \qopt_G}\right] \notag \\
&= \E_G\left[(q-\lambda)\Ind{D < \lambda}\right] + \E_G[(\lambda-D)\Ind{D < \lambda}] \\&\qquad+ (q-\lambda)(1-\Gminus(\lambda)) - \E_G\left[ (\qopt_G-D)\Ind{D \leq \qopt_G}\right]\\
&= (q-\lambda) + \E_G[(\lambda-D)\Ind{D < \lambda}] - \E_G\left[ (\qopt_G-D)\Ind{D \leq \qopt_G}\right].
\end{align*}
Plugging this back into \eqref{eq:id-step-1}, the second case is shown.
\hfill\Halmos
\end{proof}

\subsection{{Worst-Case Nature of Two-Point Mass Distributions}}\label{apx:bernoulli-is-worst}

\subsubsection{Proof of \Cref{prop:ber-is-worst-case}}

\begin{proof}{\it Proof.}
This fact is a corollary of the derivation of $\Regret(q)$ in the proof of \Cref{lem:sup_expression_identifiability}. To show this, recall our three cases: (i) $q > \tilde{q} \geq \lambda$, (ii) $q < \lambda \leq \tilde{q}$, and (iii) $\lambda \leq q < \tilde{q}$, where $\tilde{q} := \qopt_F$.
\begin{enumerate}[(i)]
\item $q > \tilde{q} \geq \lambda$: In this case, we showed that any worst-case distribution $F$ is such that $\tilde{q} = \lambda$, and \mbox{$\Pr_F(D \leq \lambda) = 1$}. Therefore, the worst-case distribution places the entirety of the remaining mass at $\lambda$.
\item $q < \lambda \leq \tilde{q}$: In this case, we argued that any worst-case distribution $F$ is such that $\Pr_F(\lambda \leq D < \tilde{q}) = 0$ and $\tilde{q} = M$. This is satisfied by a distribution $F$ that places the entirety of the remaining mass at $M$.
\item $\lambda \leq q < \tilde{q}$: Here we similarly argued that any worst-case distribution $F$ is such that $\Pr_F(\lambda \leq D < \tilde{q}) = 0$ and $\tilde{q} = M$. Again, this is achieved by a distribution $F$ that places the entirety of the remaining mass on $M$.
\end{enumerate}
\hfill\Halmos
\end{proof}
\subsection{Proof of \Cref{thm:minmax_regret}}\label{apx:ub_proof}


\begin{proof}{\it Proof.}
 \Cref{prop:regret-decomp} below upper bounds our algorithm's regret by (i) the minimax risk $\risk$, and (ii) the cost difference between $\qalg$ and $\qrisk$, for any distribution $F \in \ambig{\lambda}$. We defer its proof to Appendix \ref{apx:regret-decomp}.

\begin{proposition}\label{prop:regret-decomp}
The regret of \Cref{alg:newsvendor} is upper bounded by:
\begin{align}\label{eq:regret-decomp}
\Regret(\qalg) \leq \Delta + \sup_{F \in \ambig{\lambda}} C_F(\qalg) - C_F(\qrisk).
\end{align}
\end{proposition}

Hence, the remainder of the proof focuses on bounding the second term. Before doing so, we show a series of useful facts. First off, \Cref{lem:newsvendor_lipschitz} establishes Lipschitzness of the newsvendor cost function. Its proof can be found in Appendix \ref{apx:newsvendor_lipschitz}.
\begin{lemma}
\label{lem:newsvendor_lipschitz}
For any distribution $F$, and ordering quantities $q, q'$: 
\[
|C_F(q) - C_F(q')| \leq \max\{b,h\}|q - q'|.
\]
\end{lemma}

Moreover, by \Cref{key-fact}, $\Ind{\lambdasamples_i < \lambda} = \Ind{\doff_i < \lambda}$ for all $i \in [N]$. We use this to prove \Cref{lem:g_minus_concentration}, which establishes that $\Gminushat(\lambda)$ and $\Gminus(\lambda)$ are close with constant probability, below. We defer its proof to Appendix \ref{apx:g_minus_concentration}.

\begin{lemma}\label{lem:g_minus_concentration}
Let $\Event = \left\{ \abs{\Gminushat(\lambda) - \Gminus(\lambda)} {<} \width \right\}$. Then, $\Pr_G(\Event) \geq 1-\delta$.
\end{lemma}

We moreover use \Cref{key-fact} to show that $\qopt_{\eCDF}$ is precisely equal to the actual SAA of $\qopt_G$.  To formalize this, we introduce some additional notation. Let $\offeCDFtrue(x) = \frac{1}{N}\sum_{i \in [N]} \Ind{\doff_i \leq x}$ be the empirical cdf of the (true) uncensored demand, and \mbox{$\qopt_{\offeCDFtrue} = \inf\{x:\offeCDFtrue(x) \geq \rho\}$} be the corresponding newsvendor quantile. We repeatedly rely on the following lemma in the remainder of the proof. Its proof can be found in Appendix \ref{apx:saa_no_bias}.

\begin{lemma}\label{prop:saa_no_bias}
Suppose $\Gminushat(\lambda) \geq \rho$. Then, $\qopt_{\offeCDF} = \qopt_{\offeCDFtrue} < \lambda$.
\end{lemma}

Having established these useful facts, we may begin with the proof of our upper bound. We prove each of the four regret bounds separately, beginning with the two ``extreme'' cases of $\Gminus(\lambda)\geq \rho +2\width$ and $\Gminus(\lambda) < \rho-2\width$. Moreover, in the remainder of the proof we condition on the good event $\Event$ defined in \Cref{lem:g_minus_concentration}.

\smallskip

\paragraph{Case I: $\Gminus(\lambda) \geq \rho + 2 \width$.}
Under event $\Event$, $\Gminushat(\lambda) {>} \rho + \width$, which implies $\qalg=\qopt_{\offeCDF}$ by construction.  Since $\Gminushat(\lambda) \geq \rho$, by \cref{prop:saa_no_bias},
    $\qopt_{\offeCDF} = \qopt_{\offeCDFtrue} < \lambda$. By \Cref{prop:worst-case-for-identifiable}, then,
    \[\Regret(\qalg) \leq C_G(\qopt_{\offeCDFtrue})-C_G(\qopt_G).\]

Hence, bounding the regret of $\qalg$ in this region reduces to the question of bounding the regret of the SAA of $\qopt_G$.

The following lemma, {whose proof is} adapted from \citet{chen2024survey} (see Appendix \ref{apx:not_well_separated_concentration_high_probability}), allows us to bound this SAA error.

\begin{lemma}
\label{lem:not_well_separated_concentration_high_probability}
Suppose $\Gminus(\lambda)\geq \rho$, and $\Gminushat(\lambda) \geq \rho$. With probability at least $1 - \delta$,
\[
C_G(\qopt_{\offeCDFtrue}) - C_G(\qopt_G) \leq \lambda(b+h)\sqrt{\frac{\log(2/\delta)}{2N}}.
\]
\end{lemma}
Taking a union bound over the event $\Event$ and this latter good event, we obtain the first case in the theorem statement.

\smallskip

\paragraph{Case II: $\Gminus(\lambda) < \rho - 2 \width$.} 
By \Cref{thm:minimax-risk-identifiable}, $\qrisk = \qcrit_G$ in this case. Moreover, under event $\Event$, \mbox{$\Gminushat(\lambda) < \rho - \width$}, and our algorithm outputs $\qalg = \qcrithat$. Then, by \Cref{prop:regret-decomp}:
\begin{align}\label{eq:case2a}
    \Regret(\qalg) 
    &\leq \risk + \sup_{F \in \ambig{\lambda}} C_F(\qcrithat) - C_F(\qcrit_G) \leq \risk + \max\{b,h\}|\qcrithat - \qcrit_G|,
\end{align}
where the second inequality follows from the fact that $C_F$ is $\max\{b,h\}$-Lipschitz (\Cref{lem:newsvendor_lipschitz}).
It suffices then to bound the error in our algorithm's estimate of $\qcrit_G$. \Cref{lem:qcrit-err}, whose proof can be found in Appendix \ref{apx:qcrit-err}, leverages the fact that $\abs{\Gminus(\lambda)-\Gminushat(\lambda)} < \width$ under event $\Event$ to bound this error.
\begin{lemma}\label{lem:qcrit-err}
Under event $\Event$, $\abs{\qcrit_G-\qcrithat} \leq \frac{M-\lambda}{1-\rho}\width$.
\end{lemma}

Plugging this back into \eqref{eq:case2a}, we obtain:
\begin{align}\label{eq:final-case-2a}
\Regret(\qalg) &\leq \risk + \frac{\max\{b,h\}(M-\lambda)}{1-\rho}\width,
\end{align}
completing the proof of this case in the theorem statement.

\smallskip

We now turn our attention to the remaining case, where $\Gminus(\lambda)\in[\rho-2\width,\rho+2\width)$, further partitioning the analysis based on whether $\Gminus(\lambda) \geq \rho$.

\smallskip

\paragraph{Case III: $\Gminus(\lambda) \in [\rho, \rho + 2 \width)$.} By \Cref{thm:minimax-risk-identifiable}, $\qrisk = \qopt_G$, and $\risk= 0$.  Moreover, given $\Event$, \mbox{$\Gminushat(\lambda) \in {(}\rho - \width, \rho + 3 \width)$}. 

Suppose first that $\Gminushat(\lambda) \in [\rho+\width,\rho+3\width)$. 
Then, $\qalg = \qopt_{\offeCDF} < \lambda$. Identical arguments as those used for Case I yield:
\begin{align}\label{eq:case2-saa-bound}
\Regret(\qalg)& \leq C_G(\qopt_{\offeCDFtrue}) - C_G(\qopt_G) \leq {{\lambda}(b+h)\width.}
\end{align}

Suppose now that $\Gminushat(\lambda) \in {(}\rho - \width, \rho + \width)$. 
In this case, $\qalg = \lambda$. Using the fact that $\qrisk = \qopt_G < \lambda$ and $\risk = 0$, by \Cref{prop:regret-decomp}, we have:
\begin{align}
 \Regret(\qalg) &\leq \risk + \sup_{F \in \ambig{\lambda}} C_F(\qalg) - C_F(\qrisk)= \sup_{F \in \ambig{\lambda}} C_F(\lambda) - C_F(\qopt_G). \notag
 \end{align}
\Cref{prop:worst-case-for-identifiable} and \Cref{lem:identifiable_same_opt_quantile} together imply that:
\begin{align*}
&\sup_{F \in \ambig{\lambda}} C_F(\lambda) - C_F(\qopt_G)\\ &= b(\qopt_G-\lambda) + (b+h)\E_G\bigg[(\lambda-D)\Ind{D < \lambda} - (\qopt_G-D)\Ind{D \leq \qopt_G}\bigg]\\
&=b(\qopt_G-\lambda) + (b+h)\E_G\bigg[(\lambda-\qopt_G)\Ind{D < \lambda} + (\qopt_G-D)\Ind{D < \lambda} - (\qopt_G-D)\Ind{D \leq \qopt_G}\bigg]\\
&=b(\qopt_G-\lambda) + (b+h)\E_G\bigg[(\lambda-\qopt_G)\Ind{D < \lambda} + (\qopt_G-D)\Ind{\qopt_G <D < \lambda}\bigg]\\
&\leq (\lambda-\qopt_G)\bigg((b+h)\Gminus(\lambda)-b\bigg)\\
&= (b+h)(\lambda-\qopt_G)(\Gminus(\lambda)-\rho).
\end{align*}

Using the fact that $\qopt_G \in [0,\lambda)$ for $\Gminus(\lambda) \geq \rho$, and $\Gminus(\lambda) \leq \rho + 2\width$ by assumption, we obtain the final bound of: 
\[\Regret(\qalg) \leq 2\lambda(b+h)\width.\]

We proceed to our final case.
\smallskip

\paragraph{Case IV: $\Gminus(\lambda) \in [\rho - 2 \width, \rho)$.} By \Cref{thm:minimax-risk-identifiable}, $\qrisk = \qcrit_G$. Moreover, under $\Event$, \mbox{$\Gminushat(\lambda) \in {(}\rho - 3 \width, \rho + \width)$}.  

Suppose first that $\Gminushat(\lambda)\in{(}\rho-3\width,\rho-\width)$. 
In this case, the algorithm outputs $\qalg=\qcrithat$. By the same arguments as those used in Case II, we obtain:
\begin{align*}
    \Regret(\qalg) \leq \risk +  \frac{\max\{b,h\}(M-\lambda)}{(1-\rho)}\width.
\end{align*}

Suppose now that $\Gminushat(\lambda) \in [\rho-\width,\rho+\width)$. In this case, $\qalg = \lambda$. Again, leveraging \Cref{prop:regret-decomp}, we have:
\begin{align}\label{eq:bound-lambda-qcrit}
\Regret(\qalg) &\leq \risk + \sup_{F\in\ambig{\lambda}}C_F(\lambda)-C_F(\qcrit_G) \leq \risk + \max\{b,h\}|\lambda-\qcrit_G|,
\end{align}
where the final inequality follows from Lipschitzness of the newsvendor cost function (\cref{lem:newsvendor_lipschitz}). Plugging in the definition of $\qcrit_G$, algebra gives us that:
\begin{align*}
|\lambda-\qcrit_G| 
&=\frac{(M-\lambda)(b-(b+h)\Gminus(\lambda))}{(b+h)(1-\Gminus(\lambda))}\\
&\leq \frac{(M-\lambda)(b-(b+h)(\rho-2\width))}{(b+h)(1-\rho)}\\
&= \frac{2(M-\lambda)(b+h)\width}{h}\\
&=\frac{2(M-\lambda)}{1-\rho}\width,
\end{align*}
where the first inequality uses the fact that $\Gminus(\lambda)\in[\rho-2\width,\rho)$, and the subsequent equalities plug in the definition of $\rho = b/(b+h)$. Plugging this bound back into \eqref{eq:bound-lambda-qcrit}, we obtain:
\begin{align*}
\Regret(\qalg)\leq \risk + \max\{b,h\}\frac{2(M-\lambda)}{1-\rho}\width,
\end{align*}
concluding the proof of the theorem.
\hfill\Halmos
\end{proof}

\medskip

\subsubsection{Proof of \Cref{prop:regret-decomp}}\label{apx:regret-decomp}

\begin{proof}{\it Proof.}
Adding and subtracting $C_F(\qrisk)$ from the definition of regret, we have:
\begin{align*}
    \Regret(\qalg) 
    & = \sup_{F \in \ambig{\lambda}} C_F(\qalg) - C_F(\qrisk) + C_F(\qrisk) - C_F(\qopt_F) \notag \\
    & \leq \sup_{F \in \ambig{\lambda}} C_F(\qalg) - C_F(\qrisk) + \sup_{F \in \ambig{\lambda}} C_F(\qrisk) - C_F(\qopt_F) \notag \\
    &=\sup_{F \in \ambig{\lambda}} C_F(\qalg) - C_F(\qrisk) + \risk,
\end{align*}
where the final equality uses the definition of $\risk = \sup_{F\in\ambig{\lambda}} C_F(\qrisk)-C_F(\qopt_F)$.\hfill\Halmos
\end{proof}

\medskip

\subsubsection{Proof of \Cref{lem:newsvendor_lipschitz}}\label{apx:newsvendor_lipschitz}

\begin{proof}{\it Proof.}
By definition:
\begin{align*}
\abs{C_F(q) - C_F(q')} & = \bigg{|}\ \E_F\left[b\left((D - q)^+ - (D - q')^+\right) + h\left((q-D)^+ - (q' - D)^+\right)\right]\ \bigg{|}.
\end{align*}
Without loss of generality, suppose $q < q'$. Consider first the term $b((D - q)^+ - (D - q')^+)$.  We have:
\begin{align*}
(D - q)^+ - (D - q')^+ &= \begin{cases}
q'-q \quad &\text{if } D \geq q' \\
0 \quad &\text{if } D \leq q \\
D-q \quad &\text{if } D \in (q,q'). 
\end{cases}
\end{align*}
Similarly:
\begin{align*}
(q-D)^+ - (q'-D)^+ &= \begin{cases}
0 \quad &\text{if } D \geq q' \\
q-q' \quad &\text{if } D \leq q \\
D-q' \quad &\text{if } D \in (q,q'). 
\end{cases}
\end{align*}
Putting these two together, we have:
\begin{align*}
&b\left((D - q)^+ - (D - q')^+\right) + h\left((q-D)^+ - (q' - D)^+\right) = \begin{cases}
b(q'-q) \quad &\text{if } D \geq q' \\
h(q-q') \quad &\text{if } D \leq q \\
b(D-q)+h(D-q') \quad &\text{if } D \in (q,q')
\end{cases}\\
\implies &\bigg{|}\ \E_F\left[b\left((D - q)^+ - (D - q')^+\right) + h\left((q-D)^+ - (q' - D)^+\right)\right]\ \bigg{|} \leq \max\{b,h\}|q-q'|.
\end{align*}
\hfill\Halmos
\end{proof}

\medskip

\subsubsection{Proof of \Cref{lem:g_minus_concentration}}\label{apx:g_minus_concentration}
\begin{proof}{\it Proof.}
Since $\Ind{\lambdasamples_i < \lambda} = \Ind{\doff_i < \lambda}$ for all $i$, we have:
\begin{align*}
\Gminushat(\lambda) = \frac1N\sum_{i\in[N]}\Ind{\doff_i < \lambda} 
\implies \E_G[\Gminushat(\lambda)] = \Gminus(\lambda).
\end{align*}
Hence, by Hoeffding's inequality \citep{blm}:
\begin{align*}
\Pr_G(|\Gminushat(\lambda) - \Gminus(\lambda)| \geq \width) \leq 2 \exp(-2N\width^2).
\end{align*}
Letting $\width = \sqrt{\frac{\log(2 / \delta)}{2 N}}$, we obtain the result.\hfill\Halmos
\end{proof}

\medskip

\subsubsection{Proof of \Cref{prop:saa_no_bias}}\label{apx:saa_no_bias}
\begin{proof}{\it Proof.}
We first show that $\Gminushat(\lambda) \geq \rho$ implies $\qopt_{\offeCDFtrue} < \lambda$.
Again, since $\Ind{\lambdasamples_i < \lambda} = \Ind{\doff_i < \lambda}$,
\begin{align*}
\frac1N\sum_{i=1}^N\Ind{\doff_i < \lambda} = \frac1N\sum_{i=1}^N\Ind{\lambdasamples_i < \lambda} = \Gminushat(\lambda) \geq \rho,
\end{align*}
by assumption.
Then, it must be that $\offeCDFtrue(x) \geq \rho$ for some $x < \lambda$, which implies that $\qopt_{\offeCDFtrue} < \lambda$.

We use this to show that $\qopt_{\offeCDF} = \qopt_{\offeCDFtrue}$. By definition, $\qopt_{\offeCDFtrue}$ satisfies:
\begin{align*}
\begin{cases}
\frac1N\sum_{i\in[N]}\Ind{\doff_i\leq x} < \rho \quad \forall \ x < \qopt_{\offeCDFtrue} \\
\frac1N\sum_{i\in[N]}\Ind{\doff_i\leq x} \geq \rho \quad \forall \ x \geq \qopt_{\offeCDFtrue}. 
\end{cases}
\end{align*}
Fix $x \geq \qopt_{\offeCDFtrue}$. Since $\lambdasamples_i \leq \doff_i$ for all $i$,
\begin{align*}
\frac1N\sum_{i\in[N]}\Ind{\lambdasamples_i\leq x} \geq \frac1N\sum_{i\in[N]}\Ind{\doff_i\leq x} \geq \frac1N\sum_{i\in[N]}\Ind{\doff_i\leq \qopt_{\offeCDFtrue}} \geq {\rho},
\end{align*}
by definition of $\qopt_{\offeCDFtrue}$. 

\smallskip

Now, fix $x < \qopt_{\offeCDFtrue}$. Since $x < \qopt_{\offeCDFtrue} < \lambda$, if $\doff_i \geq \lambda$, $\lambdasamples_i = \lambda$, which then implies that $\Ind{\lambdasamples_i \leq x} = 0$. We moreover have in this case that $\Ind{\doff_i \leq x} = 0$. Putting these two together, we have that $\Ind{\lambdasamples_i \leq x} = \Ind{\doff_i \leq x}$ if $\doff_i \geq \lambda$.

If $\doff_i < \lambda$, on the other hand, $\lambdasamples_i = \doff_i$ by definition. Therefore, $\Ind{\lambdasamples_i \leq x} = \Ind{\doff_i \leq x}$ whenever $\doff_i < \lambda$.

Since $\Ind{\lambdasamples_i \leq x} = \Ind{\doff_i \leq x}$ whenever $x < \qopt_{\offeCDFtrue}$, we obtain: 
\[\frac1N\sum_{i\in[N]}\Ind{\lambdasamples_i\leq x} = \frac1N\sum_{i\in[N]}\Ind{\doff_i\leq x} < \rho \quad \forall \ x < \qopt_{\offeCDFtrue}.\]

Putting these two facts together, we have:
\begin{align*}
\begin{cases}
\frac1N\sum_{i\in[N]}\Ind{\lambdasamples_i\leq x} < \rho \quad \forall \ x < \qopt_{\offeCDFtrue} \\
\frac1N\sum_{i\in[N]}\Ind{\lambdasamples_i\leq x} \geq \rho \quad \forall \ x \geq \qopt_{\offeCDFtrue},
\end{cases}
\end{align*}
which implies that $\qopt_{\offeCDF}=\qopt_{\offeCDFtrue}$.\hfill\Halmos
\end{proof}

\medskip

\subsubsection{Proof of \Cref{lem:not_well_separated_concentration_high_probability}}\label{apx:not_well_separated_concentration_high_probability}
\begin{proof}{\it Proof.}
The proof is modified from \citet{chen2024survey}. 

By the Dvoretzky–Kiefer–Wolfowitz–Massart inequality (DKW) inequality \citep{massart1990tight}:
\[
\Pr\left[ \sup_{a \geq 0} |\offeCDFtrue(a) - G(a)| \leq \sqrt{\frac{\log(2 / \delta)}{2N}} \right] \geq 1 - 2 \exp\left(-2N\left(\sqrt{\frac{\log(2/\delta)}{2N}} \right)^2 \right) = 1 - \delta.
\]
Therefore with probability at least $1 - \delta$:
\begin{equation}
\label{eq:concentration_dkw_proof_two}
\sup_{a \geq 0} |\offeCDFtrue(a) - G(a)| \leq \sqrt{\frac{\log(2 / \delta)}{2N}}.
\end{equation}
We condition our analysis on the above event, partitioning the proof based on (i) $\qopt_{\offeCDFtrue} \leq \qopt_G$, and (ii) $\qopt_{\offeCDFtrue} > \qopt_G$.

\paragraph{Case I: $\qopt_{\offeCDFtrue} \leq \qopt_G$.} By \Cref{lem:cost_difference_robustness}:
\begin{align}\label{eq:first-case-a}
&C_G(\qopt_{\offeCDFtrue}) - C_G(\qopt_G) \notag \\ &= b(\qopt_G-\qopt_{\offeCDFtrue}) + (b+h)\E_G\left[(\qopt_{\offeCDFtrue}-D)\Ind{D\leq\qopt_{\offeCDFtrue}}-(\qopt_G-D)\Ind{D\leq\qopt_G}\right] \notag \\ 
&= b(\qopt_G-\qopt_{\offeCDFtrue}) + (b+h)\E_G\left[(\qopt_{\offeCDFtrue}-\qopt_G)\Ind{D\leq\qopt_{\offeCDFtrue}}-(\qopt_G-D)\Ind{\qopt_{\offeCDFtrue} < D\leq\qopt_G}\right] \notag \\
&\leq (b+h)(\qopt_G-\qopt_{\offeCDFtrue})(\rho-G(\qopt_{\offeCDFtrue})),
\end{align}
where the upper bound follows from the fact that $(\qopt_G-D)\Ind{\qopt_{\offeCDFtrue}< D\leq \qopt_G} \geq 0$, and $b = (b+h)\rho$ by definition. Moreover:
\begin{equation}
\label{eq:first_case}
\rho - G(\qopt_{\offeCDFtrue}) =  \rho - \offeCDFtrue(\qopt_{\offeCDFtrue}) + \offeCDFtrue(\qopt_{\offeCDFtrue})  - G(\qopt_{\offeCDFtrue}) \leq \sup_{a \geq 0} |\offeCDFtrue(a) - G(a)|,
\end{equation}
where the inequality follows from $\offeCDFtrue(\qopt_{\offeCDFtrue}) \geq \rho$, by definition of $\qopt_{\offeCDFtrue}$. 

Using this fact in \eqref{eq:first-case-a}, we have:
\begin{align*}
    C_G(\qopt_{\offeCDFtrue}) - C_G(\qopt_G) & \leq (b+h)(\qopt_G - \qopt_{\offeCDFtrue}) (\rho - G(\qopt_{\offeCDFtrue})) \\
    & \leq (b+h)\cdot \abs{\qopt_G - \qopt_{\offeCDFtrue}}\cdot \sup_{a \geq 0} |\offeCDFtrue(a) - G(a)|.
\end{align*}

\paragraph{Case II: $\qopt_{\offeCDFtrue} > \qopt_G$.} Again, by \Cref{lem:cost_difference_robustness}:
\begin{align}\label{eq:use-for-rcn-plus}
&C_G(\qopt_{\offeCDFtrue}) - C_G(\qopt_G) \notag \\ &= b(\qopt_G-\qopt_{\offeCDFtrue}) + (b+h)\E_G\left[(\qopt_{\offeCDFtrue}-D)\Ind{D\leq\qopt_{\offeCDFtrue}}-(\qopt_G-D)\Ind{D\leq\qopt_G}\right] \notag \\ 
&= -\rho(b+h)(\qopt_{\offeCDFtrue}-\qopt_G)+(b+h)\E_G\bigg[(\qopt_{\offeCDFtrue}-\qopt_G)\Ind{D\leq \qopt_G}+(\qopt_{\offeCDFtrue}-D)\Ind{\qopt_G< D < \qopt_{\offeCDFtrue}}\bigg]\notag \\
&\leq (b+h)(\qopt_{\offeCDFtrue}-\qopt_G)\left(-\rho+\Gminus\left(\qopt_{\offeCDFtrue}\right)\right),
\end{align}
where the inequality follows from $\qopt_{\offeCDFtrue}-D < \qopt_{\offeCDFtrue}-\qopt_G$ for all $D > \qopt_G$.

For all $q < \qopt_{\offeCDFtrue}$:
\[ G(q) - \rho =  G(q) - \offeCDFtrue(q) + \offeCDFtrue(q)  - \rho \leq \sup_{a \geq 0} |\offeCDFtrue(a) - G(a)|,
\]
where the inequality follows from the fact that $\offeCDFtrue(q) < \rho$ for all $q < \qopt_{\offeCDFtrue}$, by definition of $\qopt_{\offeCDFtrue}$. We similarly use this to obtain the final upper bound:
\[C_G(\qopt_{\offeCDFtrue}) - C_G(\qopt_G) \leq (b+h)\cdot\abs{\qopt_{\offeCDFtrue}-\qopt_G}\cdot\sup_{a \geq 0} |\offeCDFtrue(a) - G(a)|.\]
Applying \eqref{eq:concentration_dkw_proof_two} to both cases, and using the fact that both $\qopt_G \leq \lambda$ and $\qopt_{\offeCDFtrue} \leq \lambda$ under the assumption that $\Gminus(\lambda) \geq \rho$ and $\Gminushat(\lambda) \geq \rho$, we obtain the claim.\hfill\Halmos 
\end{proof}

\medskip

\subsubsection{Proof of \Cref{lem:qcrit-err}}\label{apx:qcrit-err}

\begin{proof}{\it Proof.}
Consider the function $\func(x) = \frac{bM + h \lambda - (b+h) M x}{(b+h)(1-x)}.$ Observe that $\qcrit_G = \func(\Gminus(\lambda))$, and \mbox{$\qcrithat = \func(\Gminushat(\lambda))$} by definition. For any $x\in[0,\rho-\width]$:
\begin{align*}
\frac{d\func}{dx} = \frac{h(\lambda - M)}{(b+h)(1 - x)^2} \implies \bigg{|}\frac{d\func}{dx}\bigg{|}\leq \frac{h(M-\lambda)}{(b+h)(1-\rho+\width)^2}.
\end{align*}
Hence $\func$ is $\frac{h(M-\lambda)}{(b+h)(1-\rho+\width)^2}$-Lipschitz, which implies
\begin{align*}
|\qcrit_G-\qcrithat| \leq \frac{h(M-\lambda)}{(b+h)(1-\rho+\width)^2}|\Gminus(\lambda)-\Gminushat(\lambda)| \leq \frac{h(M-\lambda)}{(b+h)(1-\rho+\width)^2}\width = \frac{(1-\rho)(M-\lambda)}{(1-\rho+\width)^2}\width
\end{align*}
under event $\Event$. Using the fact that since $\rho < 1$ and $\width > 0$, $(1 - \rho)^2 \leq (1 - \rho + \width)^2$, which completes the proof of the claim. \hfill\Halmos
\end{proof}
\subsection{Proof of \Cref{cor:exp-regret}}\label{apx:exp-regret}

\begin{proof}{\it Proof.}
For ease of notation, we let $R^{\text{id}} = 2\lambda(b+h)\width$ be the upper bound on the regret of \ALG derived in \Cref{thm:minmax_regret}; we also let {$R_{\max}^{\id} = \max_{q \in [0,M]}\sup_{F \in \ambig{\lambda}} \Regret(q)$} when $\Gminus(\lambda) \geq \rho$. Similarly, let $R^{\text{ui}} = \Delta+ 2\max\left\{\frac{b}{h},1\right\}(M-\lambda)(b+h)\width$ and $R^{\text{ui}}_{\max} = \max_{q\in[0,M]}\sup_{F \in \ambig{\lambda}}\Regret(q)$ when $\Gminus(\lambda) < \rho$. Applying the bounds from \Cref{thm:minmax_regret}, we have:
\begin{align}\label{eq:exp-regret}
\E_G[\Regret(\qalg)] \leq \begin{cases}
R^{\text{id}} + R_{\max}^{\text{id}} \cdot 2\delta \quad \text{if } \Gminus(\lambda) \geq \rho \\
R^{\text{ui}} + R_{\max}^{\text{ui}} \cdot 2\delta \quad \text{if } \Gminus(\lambda) < \rho.
\end{cases}
\end{align}
We first bound $R^{\text{id}}_{\max}$. By \Cref{prop:worst-case-for-identifiable}, if $q < \lambda$, $\Regret(q) = C_G(q) - C_G(\qopt_G)$.
By \Cref{lem:newsvendor_lipschitz}:
\begin{align*}
C_G(q)-C_G(\qopt_G) \leq \max\{b,h\}|q-\qopt_G| \leq (b+h)\max\{\rho,1-\rho\}\lambda.
\end{align*}

If $q \geq \lambda$, by \Cref{prop:worst-case-for-identifiable},
\begin{align*}
\Regret(q) &= b(\qopt_G-q) + (b+h)\bigg[(q-\lambda)+\E_G\Big[(\lambda-D)\Ind{D < \lambda} - (\qopt_G - D)\Ind{D \leq \qopt_G}\Big]\bigg] \\
&= b\qopt_G +hq - (b+h)\lambda + (b+h)\E_G\Big[(\lambda-D)\Ind{D < \lambda} - (\qopt_G - D)\Ind{D \leq \qopt_G}\Big]\\
&\leq b\qopt_G+hq+(b+h)\left(-\lambda + \lambda\Gminus(\lambda)-\qopt_G\rho - \E_G\big[D\Ind{D \in (\qopt_G, \lambda)}\big]\right).
\end{align*}
Using the fact that $(b+h)\rho\qopt_G = b\qopt_G$ and $\E_G\big[D\Ind{D \in (\qopt_G, \lambda)}\big] \geq 0$, we obtain:
\begin{align*}
\Regret(q) \leq hM - (b+h)\lambda(1-\Gminus(\lambda)) \leq hM.
\end{align*}

Putting these two bounds together, we obtain:
\begin{align*}
R^{\text{id}}_{\max} \leq (b+h) \max\left\{\max\{\rho,(1-\rho)\}\lambda,(1-\rho)M\right\} = (b+h) \max\left\{\rho\lambda,(1-\rho)M\right\}.
\end{align*}

We now bound $R^{\text{ui}}_{\max}$. By \Cref{lem:sup_expression_identifiability}, for $q < \lambda$,
\begin{align*}
\Regret(q) &= b(M-q) + (b+h)\bigg[\E_G\Big[(q-D)\Ind{D \leq q} - (M-D)\mathds{1}\{D < \lambda\}\Big]\bigg] \leq bM,
\end{align*}
where the inequality follows from the fact that $ -bq + (b+h)\E_G[(q-D)\Ind{D \leq q}]$ is decreasing in $q$ for $q < \lambda < \qopt_G$.

For $q \in \left[\lambda,\qcrit_G\right]$,
\begin{align*}
\Regret(q) = (b-(b+h)\Gminus(\lambda))(M-q) \leq (b-(b+h)\Gminus(\lambda))(M-\lambda).
\end{align*}
Finally, for $q > \qcrit_G$,
\begin{align*}
\Regret(q) = h(q-\lambda) \leq h(M-\lambda).
\end{align*}
Putting these three bounds together, we obtain:
\begin{align*}
R^{\text{ui}}_{\max} &\leq (b+h) \max\left\{\rho M, (\rho-\Gminus(\lambda))(M-\lambda),(1-\rho)(M-\lambda)\right\}= (b+h)\max\left\{\rho M, (1-\rho)(M-\lambda)\right\}.
\end{align*}

We conclude by plugging these bounds back into \eqref{eq:exp-regret}. Since $\width = \sqrt{\log(2/\delta)/2N} = c'\sqrt{\log N / N}$ for some $c' > 0$ and $\delta = c/\sqrt{N}$, if $\Gminus(\lambda) \geq \rho$, we have:
\begin{align*}
\E_G[\Regret(\qalg)] &\leq 2c'(b+h)\left(\lambda + \max\left\{\rho\lambda,(1-\rho)M\right\}\right) \sqrt{\log N /N}.
\end{align*}

If $\Gminus(\lambda) < \rho$:
\begin{align*}
\E_G[\Regret(\qalg)] &\leq \Delta + 2c'(b+h)\Bigg(\max\left\{\frac{\rho}{1-\rho},1\right\}(M-\lambda) + \max\{\rho M, (1-\rho)(M-\lambda)\}\Bigg)\sqrt{\log N /N}.
\end{align*}
Renaming $c' = 2c'$, we obtain the result.
\hfill\Halmos 
\end{proof}

\subsection{Proof of \Cref{thm:lb}}
\label{app:lb_proof}

\begin{proof}{\it Proof.}
We consider an instance for which $\lambda \in (0,1)$, and present the set of ``hard'' distributions associated with each regime:
\begin{enumerate}[$(i)$]
\item strictly unidentifiable regime: 
\[G_0^\ui = \Ber(1-\rho+\delta_0^{\ui}), \quad G_1^\ui = \Ber(1-\rho + \delta_0^{\ui} + \delta_1^{\ui}),\] where $\delta_0^{\ui} = \min\left\{\frac\rho2,\frac{1-\rho}{2}\right\}$ and $\delta_1^{\ui} = \min\left\{\frac\rho4,\frac{3\rho-1}{4},\frac12\sqrt{\frac{1-\rho}{N}}\right\}$. For these distributions, \mbox{$G_0^{\ui,-}(\lambda) = \rho-\delta_0^{\ui}$}, and $G_1^{\ui,-}(\lambda) = \rho-(\delta_0^{\ui}+\delta_1^{\ui})$.
\item knife-edge regime: \[G_0^{\ke} = \Ber(1-\rho+\delta^{\ke}), \quad G_1^{\ke} = \Ber(1-\rho-\delta^{\ke}),\] where \mbox{$\delta^{\ke} = \min\left\{\frac{\rho}{2},\frac{1-\rho}{2},\frac14\sqrt{\frac{1-\rho}{N}}\right\}$}. Here, $G_0^{\ke,-}(\lambda)=\rho-\delta^{\ke}$, \mbox{$G_1^{\ke,-}(\lambda)=\rho+\delta^{\ke}$}.
\item strictly identifiable regime: $G_0^{\id}$, $G_1^{\id}$ are respectively defined by cdf's: 
\begin{align*}
G_0^{\id}(x) = \begin{cases}
0 \quad \forall \ x < 0 \\
\rho-\delta^{\id} \quad \forall \ x \in [0,H) \\
1 \quad \forall \ x \geq H,
\end{cases}
\quad \quad G_1^{\id}(x) = \begin{cases}
0 \quad \forall \ x < 0\\
\rho+\delta^{\id} \quad \forall \ x \in [0,H) \\
1 \quad \forall \ x \geq H,
\end{cases}
\end{align*}
where $\delta^{\id} = \min\left\{\frac{\rho}{2},\frac{1-\rho}{2},\frac14\sqrt{\frac{1-\rho}{N}}\right\}$, and $H = \lambda/2$. In this case, \mbox{$G_0^{\id,-}(\lambda) = G_1^{\id,-}(\lambda) = 1$}.
\end{enumerate}

Note that all instances are effectively uncensored. For $k \in \{\ui,\ke\}$, this holds because \mbox{$\lambdasamples = \lambda \implies \doff_i = 1$}. For $k = \id$, this holds because $\doff_i < \lambda$ for all $i$, which then implies that $\doff_i = \lambdasamples_i$ for all $i$. 

\smallskip 

For all $k \in \{\id,\ui,\ke\}$:
\begin{align}\label{eq:bound-two-dist}
\sup_{G \in \mathcal{G}^k}\E_{G}\bigg[\sup_{F\in\ambig{\lambda}}C_F(q^{\pi})-C_F(\qopt_F)-\risk_G\bigg] \geq \sup_{G \in \{G_0^k, G_1^k\}}\E_{G}\bigg[\sup_{F\in\ambig{\lambda}}C_F(q^{\pi})-C_F(\qopt_F)-\risk_G\bigg].
\end{align}

For clarity, in the remainder of the proof we make clear the dependence of $\risk$ and $\qrisk$ on the underlying demand distribution $G$.

Recall, by \Cref{thm:minimax-risk-identifiable}, for any distribution $G$, the {\minimaxquant} and minimax risk are respectively given by:
\begin{align}\label{eq:recall-risk}
\qrisk_G = \begin{cases}
        \qopt_G & \text{if } \Gminus(\lambda) \geq \rho \\
        \qcrit_G & \text{if } \Gminus(\lambda) < \rho
    \end{cases}
    \hspace{3cm}  \risk_G = \begin{cases}
    0 & \text{if } \Gminus(\lambda) \geq \rho \\
    \frac{h\left(b-(b+h)\Gminus(\lambda)\right)(M-\lambda)}{(b+h)(1-\Gminus(\lambda)} &\text{if } \Gminus(\lambda) < \rho,
    \end{cases}
\end{align}
with $\frac{h\left(b-(b+h)\Gminus(\lambda)\right)(M-\lambda)}{(b+h)(1-\Gminus(\lambda)} = \left(b-(b+h)\Gminus(\lambda)\right)(M-\qcrit_G) = h(\qcrit_G-\lambda)$ for $G$ such that $\Gminus(\lambda) < \rho$ (see proof of \Cref{thm:minimax-risk-identifiable}).

\Cref{lem:risks-for-all-regimes} provides the {\minimaxquant} for the three sets of distributions described above. We defer its proof to Appendix \ref{apx:risks-for-all-regimes}.
\begin{lemma}\label{lem:risks-for-all-regimes}
The following holds, in each regime:
\begin{enumerate}[$(i)$]
\item strictly unidentifiable regime:
\[\qrisk_G = \qcrit_G \quad \forall \ G \in \{G_0^{\ui}, G_1^{\ui}\}\]
\item knife-edge regime:
\[\qrisk_{G_0^{\ke}} = \qcrit_{G_0^{\ke}}, \quad \qrisk_{G_1^{\ke}} = \qopt_{G_1^{\ke}} = 0\]
\item strictly identifiable regime:
\[\qrisk_{G_0^{\id}} = \qopt_{G_0^{\id}} = H, \quad \qrisk_{G_1^{\id}} = \qopt_{G_1^{\id}} = 0.\]
\end{enumerate}
\end{lemma}

\Cref{lem:regret-to-est} next formalizes the idea that minimizing the worst-case regret in excess of $\risk_G$ reduces to the problem of estimating $\qrisk_G$. We defer its proof to Appendix \ref{apx:lem:regret-to-est}.

\begin{lemma}\label{lem:regret-to-est}
For $k \in \{\id,\ui,\ke\}$, $G \in \{G_0^k, G_1^k\}$:
\[\sup_{F\in\ambig{\lambda}}C_F(q^{\pi})-C_F(\qopt_F)-\risk_G  \geq \begin{cases}
(b+h)\delta_0^k \ \abs{q^{\pi}-\qrisk_G} \quad \text{if } k = \ui \\
(b+h)\delta^k \ \abs{q^{\pi}-\qrisk_G} \quad \text{if } k \in \{\ke,\id\}.
\end{cases}\]
\end{lemma}

Applying \Cref{lem:regret-to-est} to \eqref{eq:bound-two-dist}, it suffices to lower bound the worst-case absolute difference between $q^{\pi}$ and $\qrisk_G$. Namely:
\begin{align}
&\sup_{G \in \mathcal{G}^k}\E_{G}\bigg[\Regret(q^\pi)-\risk_G\bigg] \geq (b+h)\delta_0^k\sup_{G \in \{G_0^k, G_1^k\}}\E_{G}\left[\abs{q^{\pi}-\qrisk_G}\right] \quad\text{if } k = \ui \label{eq:regret-to-est-1} \\
&\sup_{G \in \mathcal{G}^k}\E_{G}\bigg[\Regret(q^\pi)-\risk_G\bigg] \geq (b+h)\delta^k\sup_{G \in \{G_0^k, G_1^k\}}\E_{G}\left[\abs{q^{\pi}-\qrisk_G}\right] \quad \text{if } k \in \{\ke,\id\}. \label{eq:regret-to-est-2}
\end{align}

We next lower bound $\sup_{G \in \{G_0^k,G_1^k\}}\mathbb{E}_{G}\big[\abs{q^{\pi}-\qrisk_G}\big]$ via a reduction to hypothesis testing. To formalize this, we introduce some additional notation. Let $\mathcal{D}^N$ be the set of all possible demand samples. For \mbox{$k \in \{\id,\ui,\ke\}$}, let $\Psi^k: \mathcal{D}^N\mapsto\left\{G_0^k,G_1^k\right\}$ denote a mapping from the observed demand samples to a prediction of the underlying demand distribution. We moreover let $G_0^{k,N}$ and $G_1^{k,N}$ be the joint distributions of the $N$ independent demand samples generated under $G_0^k$ and $G_1^k$, respectively.

{By Proposition 9.2.1. in \citet{duchi2016lecture}}:
\begin{align}\label{eq:est-to-test}
&\sup_{G\in \{G_0^k, G_1^k\}} \E_{G} \left[|q^{\pi}-\qrisk_G| \right]\notag \\ &\qquad \geq  \frac12 {\bigg{|}\qrisk_{G_0^k}-\qrisk_{G_1^k}\bigg{|}}{\inf_{\Psi^k}\bigg\{\Pr_{G_1^{k}}\left(\Psi^k(\doff) = G_0^k\right)+  \Pr_{G_0^{k}}\left(\Psi^k(\doff) = G_1^k \right)\bigg\}} \notag \\
&\qquad \geq \frac14\underbrace{\bigg{|}\qrisk_{G_0^k}-\qrisk_{G_1^k}\bigg{|}}_{(I)}\underbrace{\exp\left(-d_{KL}(G_1^{k,N}\mid\mid G_0^{k,N})\right)}_{(II)},
\end{align}
where $d_{KL}(G_1^{k,N}\mid\mid G_0^{k,N})$ denotes the Kullback-Leibler (KL) divergence, and the second inequality follows from the Bretagnolle-Huber inequality \citep{lattimore2020bandit}.

Equation \eqref{eq:est-to-test} demonstrates the main drivers of regret in our setting. In particular, any algorithm must incur higher regret if (i) the {minimax optimal ordering quantities} under $G_0^k$ and $G_1^k$ are far apart (term $(I)$), or (ii) $G_0^k$ and $G_1^k$ are distributionally ``close enough'' that any hypothesis test that aims to distinguish between the two incurs high error rate (term $(II)$). Lemmas \ref{lem:qrisk-diff} and \ref{lem:hypothesis} respectively lower bound these two component terms for all regimes. We defer their proofs to Appendix \ref{apx:qrisk-diff} and \ref{apx:hypothesis}, respectively.

\begin{lemma}\label{lem:qrisk-diff}
For $k \in \{\ui,\ke,\id\}$:
\begin{align*}
\big{|}\qrisk_{G_0^k}-\qrisk_{G_1^k}\big{|} \geq \begin{cases}
\delta_1^{\ui}\cdot \frac{h(M-\lambda)}{b+h} \quad &\text{if } k = \ui \\
\lambda \quad &\text{if } k = \ke \\
H \quad &\text{if } k = \id.
\end{cases}
\end{align*}
\end{lemma}

\medskip

\begin{lemma}\label{lem:hypothesis}
For $k \in \{\ui,\ke,\id\}$:
\begin{align*}
\exp\left(-d_{KL}(G_1^{k,N}\mid\mid G_0^{k,N})\right) \geq \begin{cases}
\exp\left(-2N\cdot\frac{(\delta_1^{k})^2}{1-\rho}\right) \quad&\text{if } k = \ui\\
\exp\left(-8N\cdot\frac{(\delta^{k})^2}{1-\rho}\right) \quad&\text{if } k \in \{\ke,\id\}.
\end{cases}
\end{align*}
\end{lemma}

Applying these two lemmas to \eqref{eq:est-to-test}, and plugging back into \eqref{eq:regret-to-est-1} and \eqref{eq:regret-to-est-2}, we obtain:
\begin{align*}
\sup_{G \in \mathcal{G}^k}\E_{G}\bigg[\Regret(q^\pi)-\risk_G\bigg] \geq \begin{cases}
\frac14{h(M-\lambda)}\delta_0^{\ui}\delta_1^{\ui}\exp\left(-2N\cdot\frac{(\delta_1^{\ui})^2}{1-\rho}\right) \quad& \text{if } k = \ui \\
\frac14\lambda(b+h)\delta^{\ke}\exp\left(-8N\cdot\frac{(\delta^{\ke})^2}{1-\rho}\right) \quad&\text{if } k = \ke\\
\frac14H(b+h)\delta^{\id}\exp\left(-8N\cdot\frac{(\delta^{\id})^2}{1-\rho}\right) \quad&\text{if } k = \id.
\end{cases}
\end{align*}
For $\delta_0^{\ui} = \min\left\{\frac\rho2,\frac{1-\rho}{2}\right\}$ and $\delta_1^{\ui} = \min\left\{\frac\rho4,\frac{3\rho-1}{4},\frac{1}{\sqrt{2\cdot 2N/(1-\rho)}}\right\}$, we obtain, for $k = \ui$: 
\begin{align*}
&\sup_{G \in \mathcal{G}^k}\E_{G}\bigg[\Regret(q^\pi)-\risk_G\bigg]\\ &\geq \frac{h(M-\lambda)}{8}\min\{\rho,1-\rho\}\cdot\min\left\{\frac\rho4,\frac{3\rho-1}{4},\frac{\sqrt{1-\rho}}{2\sqrt{N}}\right\}\cdot e^{-1/2}\\
\\&\geq \frac{h(M-\lambda)\sqrt{1-\rho}\min\{\rho,1-\rho\}\cdot\min\left\{\rho,{3\rho-1}\right\}\cdot e^{-1/2}}{64\sqrt{N}},
\end{align*}
where the second inequality follows from the fact that $\min\{x,y\} \geq xy$ for $x, y \in [0,1]$. 

\medskip 
For $\delta^{\ke} = \min\left\{\frac{\rho}{2},\frac{1-\rho}{2},\frac{1}{\sqrt{2\cdot 8N/(1-\rho)}}\right\}$, we obtain, for $k = \ke$:
\begin{align*}
&\sup_{G \in \mathcal{G}^k}\E_{G}\bigg[\sup_{F\in\ambig{\lambda}}C_F(q^{\pi})-C_F(\qopt_F)-\risk_G\bigg] \\
&\geq \frac{\lambda(b+h)}{4}\min\left\{\frac{\rho}{2},\frac{1-\rho}{2},\frac{1}{4}\sqrt{\frac{1-\rho}{N}}\right\}\cdot e^{-1/2}\\
&\geq \frac{\lambda(b+h)\sqrt{1-\rho}\min\{\rho,1-\rho\}e^{-1/2}}{32\sqrt{N}}.
\end{align*}
Finally, using the same instantiation of $\delta^{\id}$ as $\delta^{\ke}$, and setting $H = \lambda/2$, we obtain:
\begin{align*}
&\sup_{G \in \mathcal{G}^k}\E_{G}\bigg[\sup_{F\in\ambig{\lambda}}C_F(q^{\pi})-C_F(\qopt_F)-\risk_G\bigg] \\
&\geq \frac{H(b+h)\sqrt{1-\rho}\min\{\rho,1-\rho\}e^{-1/2}}{32\sqrt{N}}\\
&= \frac{\lambda(b+h)\sqrt{1-\rho}\min\{\rho,1-\rho\}e^{-1/2}}{64\sqrt{N}}.
\end{align*}
\hfill\Halmos
\end{proof}

\medskip

\subsubsection{Proof of \Cref{lem:risks-for-all-regimes}}\label{apx:risks-for-all-regimes}
\begin{proof}{\it Proof.} 
We proceed regime-by-regime.

\smallskip

\paragraph{Case I: Strictly unidentifiable regime.} 
This follows from the fact that $\Gminus(\lambda) < \rho$ for all \mbox{$G \in \{G_0^{\ui},G_1^{\ui}\}$}.

\smallskip

\paragraph{Case II: Knife-edge regime.} 
This follows from the fact that $G_0^{\ke,-}(\lambda) < \rho$, and $G_1^{\ke,-}(\lambda) \geq \rho$. Moreover, since $G_1^{\ke}$ is Bernoulli, it must then be that $\qopt_{G_1^{\ke}} = 0$.

\smallskip

\paragraph{Case III: Strictly identifiable regime.} This follows from the fact that $\Gminus(\lambda) \geq \rho$ for all \mbox{$G \in \{G_0^{\id}, G_1^{\id}\}$}. For $G = G_0^{\id}$, $\Pr_G(D = 0) < \rho$, which implies that $\qopt_G = H$. For $G = G_1^{\id}$, \mbox{$\Pr_G(D=0) \geq \rho$}, which implies $\qopt_G = 0$. 
\hfill\Halmos
\end{proof}

\medskip

\subsubsection{Proof of \Cref{lem:regret-to-est}}\label{apx:lem:regret-to-est}
\begin{proof}{\it Proof.}
We prove the lemma separately for each regime. In all cases, we rely repeatedly on the closed-form characterizations of the {minimax risk} established in \Cref{lem:sup_expression_identifiability} and \Cref{prop:worst-case-for-identifiable}, bounding the loss based on the location of $q^{\pi}$ relative to $\lambda$.

\smallskip

\paragraph{Case I: Strictly unidentifiable regime.} By \Cref{lem:risks-for-all-regimes}, for all $G \in \{G_0^{\ui}, G_1^{\ui}\}$, $\qrisk_G = \qcrit_G$.

{Suppose first that $q^{\pi} < \lambda$. By \Cref{lem:sup_expression_identifiability}, we have:
\begin{align*}
\sup_{F \in \ambig{\lambda}} C_F(q^{\pi}) - C_F(\qopt_F) &= b(M-q^{\pi})+(b+h)\bigg(q^{\pi}\Pr_G(D = 0) - M\Pr_G(D=0)\bigg)\\
&= (M-q^{\pi})\left(b-(b+h)\Gminus(\lambda)\right),
\end{align*}
since $\lambda \in (0,1)$ and $G$ is Bernoulli.
Subtracting $\risk_G = \left(b-(b+h)\Gminus(\lambda)\right)(M-\qcrit_G)$ on both sides, we have: 
\begin{align*}
\sup_{F \in \ambig{\lambda}} C_F(q^{\pi}) - C_F(\qopt_F)-\risk_G &= (M-q^{\pi})\left(b-(b+h)\Gminus(\lambda)\right)-(b - (b+h)\Gminus(\lambda))(M - \qcrit_G) \\
&= (\qcrit_G-q^{\pi})\left(b-(b+h)\Gminus(\lambda)\right)\\
&\geq (\qcrit_G-q^{\pi})(b+h)\delta_0^{\ui},
\end{align*}
where the final inequality follows from the fact that $\Gminus(\lambda) \leq \rho-\delta_0^{\ui}$ for $G \in \{G_0^{\ui}, G_1^{\ui}\}$, and $\qcrit_G \geq \lambda > q^{\pi}$.
}

Consider now the case where $q^{\pi} \in [\lambda, \qcrit_G]$.  By \cref{lem:sup_expression_identifiability}, for {$\Gminus(\lambda) < \rho$}, we have:
\begin{align*}
    \sup_{F \in \ambig{\lambda}} C_F(q^{\pi}) - C_F(\qopt_F) - \risk_G & = (b - (b+h)\Gminus(\lambda))(M - q^{\pi}) - (b - (b+h)\Gminus(\lambda))(M - \qcrit_G) \\
    &= (b - (b+h)\Gminus(\lambda))(\qcrit_G - q^{\pi}) \\
    &\geq (b+h)\delta_0^{\ui}(\qcrit_G - q^{\pi}),
\end{align*}
where again we used the fact that $\Gminus(\lambda) \leq \rho - \delta_0^{\ui}$ for $G \in \{G_0^{\ui}, G_1^{\ui}\}$.

Finally, if $q^{\pi} \in (\qcrit_G, M]$, using the alternative representation of $\risk_G = h(\qcrit_G-\lambda)$, by \cref{lem:sup_expression_identifiability}:
\begin{align*}
    \sup_{F \in \ambig{\lambda}} C_F(q^{\pi}) - C_F(\qopt_F) - \risk_G & = h(q^{\pi} - \lambda) - h(\qcrit_G - \lambda)  = h(q^{\pi} - \qcrit_G).
\end{align*}

Combining the three cases and using the fact that $\delta_0^{\ui} < 1-\rho \implies (b+h)\delta_0^{\ui} < h$, we obtain the lower bound for the strictly unidentifiable regime.

\smallskip

\paragraph{Case II: Knife-edge regime.} If $G = G_0^{\ke}$, since $G_0^{\ke,-}(\lambda) < \rho$, by the same arguments as above:
\begin{align}\label{eq:gminus-ke-1}
\sup_{F \in \ambig{\lambda}} C_F(q^{\pi}) - C_F(\qopt_F) - \risk_G \geq (b+h)\delta^{\ke}|q^{\pi}-\qcrit_G|.
\end{align}

If $G = G_1^{\ke}$, by \eqref{eq:recall-risk}, $\risk_G = 0$, since $G_1^{\ke,-}(\lambda) \geq \rho$ by construction. Moreover, by \Cref{lem:risks-for-all-regimes}, \mbox{$\qrisk_G = \qopt_G = 0$}. 

Suppose first that $q^{\pi} < \lambda$. Applying 
{
\Cref{prop:worst-case-for-identifiable} to $G = G_1^{\ke}$, we have:
\begin{align*}
\sup_{F \in \ambig{\lambda}} C_F(q^{\pi})-C_F(\qopt_F) &= -bq^{\pi} + (b+h) q^\pi \Pr_G(D = 0)\\
&=-bq^{\pi} + (b+h) q^\pi (\rho+\delta^{\ke}) \\
&= q^{\pi}(b+h)\delta^{\ke},
\end{align*}
where the first equality simplified the expression in \Cref{prop:worst-case-for-identifiable} using the fact that $\lambda \in (0,1)$, and $\qopt_G = 0$.
}

Suppose now that $q^{\pi} \geq \lambda$. Again, by 
{\Cref{prop:worst-case-for-identifiable}:
\begin{align*}
\sup_{F \in \ambig{\lambda}} C_F(q^{\pi})-C_F(\qopt_F) &= -bq^{\pi} + (b+h)\left(q^\pi-\lambda + \lambda\Pr_G(D = 0)\right)\\
&=hq^\pi - \lambda(b+h)\Pr_G(D = 1)\\
&\geq q^{\pi}\left(h-(b+h)(1-\rho-\delta^{\ke})\right)\\
&= q^\pi(b+h)\delta^{\ke}.
\end{align*}
}

Hence, in both cases:
\begin{align}\label{eq:gminus-ke-2}
\sup_{F \in \ambig{\lambda}} C_F(q^{\pi})-C_F(\qopt_F) \geq (q^{\pi}-\qopt_G)(b+h)\delta^{\ke}.
\end{align}

Since $\qrisk_G = \qcrit_G$ for $G = G_0^{\ke}$, and $\qrisk_G = \qopt_G$ for $G = G_1^{\ke}$, putting \eqref{eq:gminus-ke-1} and \eqref{eq:gminus-ke-2} together we obtain:
\begin{align}
\sup_{F\in\ambig{\lambda}} C_F(q^{\pi})-C_F(\qopt_F)-\risk_G \geq (b+h)\delta^{\ke}|q^{\pi}-\qrisk| \quad \forall \ G \in \{G_0,G_1\}. \notag
\end{align}

\smallskip

\paragraph{Case III: Strictly identifiable regime.}
By \Cref{lem:risks-for-all-regimes}, $\qrisk_{G_0^{\id}} = \qopt_{G_0^{\id}} = H$, and $\qrisk_{G_1^{\id}} = \qopt_{G_1^{\id}} = 0$. Moreover, for $G \in \{G_0^{\id}, G_1^{\id}\}$, $\risk_G = 0$, by \eqref{eq:recall-risk}.

\smallskip

Suppose first that $G = G_0^{\id}$. By 
{
\Cref{prop:worst-case-for-identifiable}, if $q^{\pi} < \lambda$:
\begin{align}\label{eq:id-loss1}
&\sup_{F \in \ambig{\lambda}} C_F(q^{\pi}) - C_F(\qopt_F) \notag \\&= b(H-q^\pi) + (b+h)\bigg(q^\pi\Pr_G(D = 0) +(q^\pi-H)\Ind{H \leq q^\pi}\Pr_G(D=H)-H\Pr_G(D = 0)\bigg) \notag\\
&=(H-q^{\pi})\left(b-(b+h)\Pr_G(D=0)-(b+h)\Pr_G(D=H)\Ind{H\leq q^{\pi}}\right) \notag \\
&=(b+h)(H-q^{\pi})\left(\delta^{\id}-(1-\rho+\delta^{\id})\Ind{H\leq q^{\pi}}\right)\notag \\
&=(b+h)(1-\rho)(q^{\pi}-H)\Ind{H \leq q^{\pi}} + (b+h)(H-q^{\pi})\delta^{\id}\Ind{H > q^{\pi}} \notag\\
&\geq (b+h)|H-q^{\pi}|\delta^{\id} = (b+h)|\qrisk_G-q^{\pi}|\delta^{\id},
\end{align}
where the inequality follows from the fact that $\delta^{\id} < 1-\rho$.
}

If $q^{\pi} \geq \lambda$, on the other hand, by 
{
\Cref{prop:worst-case-for-identifiable}:
{
\begin{align}\label{eq:id-loss2}
&\sup_{F \in \ambig{\lambda}} C_F(q^{\pi}) - C_F(\qopt_F)\notag \\ &= 
b(H-q^\pi)+(b+h)\bigg(q^\pi-\lambda + \lambda\Pr_G(D = 0) + (\lambda-H)\Pr_G(D = H) - H\Pr_G(D = 0)\bigg)\notag \\
&=b(H-q^{\pi})+(b+h)(q^{\pi}-H)\notag \\
&= h(q^{\pi}-H)\notag\\
&\geq (b+h)|H-q^{\pi}|\delta^{\id} = (b+h)\abs{\qrisk_G-q^{\pi}}\delta^{\id},
\end{align}
where the inequality follows from the fact that $q^{\pi} \geq \lambda > H$, and $\delta^{\id} < 1-\rho$.
}
}

Putting \eqref{eq:id-loss1} and \eqref{eq:id-loss2} together, we obtain the lower bound for $G = G_0^{\id}$.

\smallskip

{
Suppose now that $G = G_1^{\id}$.
By \Cref{prop:worst-case-for-identifiable}, if $q^{\pi} < \lambda$:{
\begin{align}\label{eq:id-loss3}
\sup_{F \in \ambig{\lambda}} C_F(q^{\pi}) - C_F(\qopt_F) &=-bq^\pi + (b+h)\bigg(q^\pi \Pr_G(D = 0) + (q^\pi-H)\Pr_G(D = H)\Ind{H \leq q^\pi}\bigg)\notag \\&= -bq^{\pi} + (b+h)\left(q^{\pi}(\rho+\delta^{\id})+(q^{\pi}-H)\Ind{H\leq q^{\pi}}(1-\rho-\delta^{\id})\right)\notag \\
&= (b+h)\left(q^{\pi}\delta^{\id}+(q^{\pi}-H)(1-\rho-\delta^{\id})\right)\Ind{H \leq q^{\pi}} \notag \\
&\quad +(b+h)\delta^{\id}q^{\pi}\Ind{H > q^{\pi}} \notag\\
&\geq (b+h)\delta^{\id}q^{\pi} = (b+h)\delta^{\id}\abs{q^{\pi}-\qrisk_G},
\end{align}
where the inequality follows from the fact that $\delta^{\id} < 1-\rho$.
}
}

If $q^{\pi} \geq \lambda$, on the other hand,
{by \Cref{prop:worst-case-for-identifiable}:{
\begin{align}\label{eq:id-loss4}
\sup_{F \in \ambig{\lambda}} C_F(q^{\pi}) - C_F(\qopt_F) &= -bq^{\pi} + (b+h)\bigg[(q^{\pi}-\lambda)+\lambda\Pr_G(D=0)+(\lambda-H)\Pr_G(D=H)\bigg] \notag \\
&=hq^{\pi} - H(b+h)(1-\rho-\delta^{\id}) \notag \\
&\geq \left(h-(b+h)(1-\rho-\delta^{\id})\right)q^{\pi} \notag \\
&=(b+h)\delta^{\id}q^{\pi} = (b+h)\delta^{\id}\abs{q^{\pi}-\qrisk_G},
\end{align}
where the inequality follows from the fact that $q^{\pi} \geq \lambda > H$.
}
}

Putting \eqref{eq:id-loss3} and \eqref{eq:id-loss4} together, we obtain the lower bound for $G = G_1^{\id}$.
\hfill\Halmos
\end{proof}

\medskip

\subsubsection{Proof of \Cref{lem:qrisk-diff}}\label{apx:qrisk-diff}

\begin{proof}{\it Proof.}
We proceed case-by-case.

\medskip

\paragraph{Case I: Strictly unidentifiable regime.} By \Cref{lem:risks-for-all-regimes}, in this setting, for $G \in \{G_0^{\ui}, G_1^{\ui}\}$, 
\begin{align*}
\qrisk_G = \qcrit_G &= \frac{h\lambda + (b-(b+h)\Gminus(\lambda))M}{(b+h)(1-\Gminus(\lambda))}.
\end{align*}
As in the proof of \Cref{thm:minmax_regret}, define $\func_G(\cdot)$ as:\[
\func_G(x) = \frac{h \lambda + (b - (b+h)x) M}{(b+h)(1 - x)}.
\]
Recall, $G_0^{\ui,-}(\lambda) = \rho-\delta_0^{\ui}$, and $G_1^{\ui,-}(\lambda) = \rho-\delta_0^{\ui}-\delta_1^{\ui}$. Hence, by the mean value theorem, for some $c \in (\rho-\delta_0^{\ui}-\delta_1^{\ui},\rho-\delta_0^{\ui})$:
\begin{align}\label{eq:mvt}
&\func_G(\rho-\delta_0^{\ui})-\func_G(\rho-(\delta_0^{\ui}+\delta_1^{\ui})) = \delta_1^{\ui} \cdot \frac{d\qcrit_G}{dx}\bigg{|}_{x=c} = \delta_1^{\ui} \cdot \frac{h(\lambda-M)}{(b+h)(1-c)^2} \notag \\
\implies &\big{|}\func_G(\rho-\delta_0^{\ui})-\func_G(\rho-(\delta_0^{\ui}+\delta_1^{\ui}))\big{|} \geq \delta_1^{\ui} \cdot \frac{h(M-\lambda)}{b+h},
\end{align}
thus providing the bound for the unidentifiable regime.  

\smallskip

\paragraph{Case II: Knife-edge regime.} By \Cref{lem:risks-for-all-regimes}, in this setting $\qrisk_{G_0^{\ke}} = \qcrit_{G_0^{\ke}}$, and $\qrisk_{G_1^{\ke}} = \qopt_{G_1^{\ke}} = 0$. Hence:
\begin{align*}
\abs{\qrisk_{G_0^{\ke}}-\qrisk_{G_1^{\ke}}} =\frac{h\lambda+(b-(b+h)(\rho-\delta^{\ke}))M}{(b+h)(1-\rho+\delta^{\ke})} \geq \frac{h\lambda+(b-(b+h)\rho)M}{(b+h)(1-\rho)} = \lambda,
\end{align*}
where the inequality follows from the fact that $\func_G(x)$ is decreasing in $x$, and $\rho-\delta^{\ke} \leq \rho$ for all $\delta \geq 0$.

\smallskip

\paragraph{Case III: Strictly identifiable regime.} This follows from the fact that, by \Cref{lem:risks-for-all-regimes}, $\qrisk_{G_0^{\id}} = H$ and $\qrisk_{G_1^{\id}} = 0$. 
\hfill\Halmos
\end{proof}

\medskip

\subsubsection{Proof of \Cref{lem:hypothesis}}\label{apx:hypothesis}

\begin{proof}{\it Proof.}
Since we are in the uncensored setting and demand samples are drawn i.i.d., we have that \mbox{$d_{KL}(G_1^{k,N}\mid\mid G_0^{k,N}) = N d_{KL}(G_1^k\mid\mid G_0^k)$}. Moreover, across all regimes, $G_0^k$ and $G_1^k$ are distributions with support over two point masses. Hence, we can use reverse Pinsker's inequality to upper bound the KL divergence from $G_1^{k,N}$ to $G_0^{k,N}$ {(\cite{sason2015reverse}, Eqn. 10)}. Namely, for any two distributions $G_0^k$, $G_1^k$ with support on the same two point masses $a$ and $b$, with $a > b$, we have:
\begin{align*}
d_{KL}(G_1^k \mid\mid G_0^k) = 2\cdot\frac{(p-q)^2}{q},
\end{align*}
where $p$ and $q$ respectively denote the mass $G_1^k$ and $G_0^k$ place on $a$. We apply this to each regime, to obtain:
\paragraph{Case I: Strictly unidentifiable regime.}
\begin{align*}
\exp\left(-d_{KL}(G_1^{\ui,N}\mid\mid G_0^{\ui,N})\right) \geq \exp\left(-2N\cdot\frac{(\delta_1^{\ui})^2}{1-\rho+\delta_0^{\ui}}\right) \geq \exp\left(-2N\cdot\frac{(\delta_1^{\ui})^2}{1-\rho}\right).
\end{align*}

\paragraph{Case II: Knife-edge regime.}
\begin{align*}
\exp\left(-d_{KL}(G_1^{\ke,N}\mid\mid G_0^{\ke,N})\right) \geq \exp\left(-2N\cdot\frac{(2\delta^{\ke})^2}{1-\rho+\delta^{\ke}}\right) \geq \exp\left(-8N\cdot\frac{(\delta^{\ke})^2}{1-\rho}\right).
\end{align*}

\paragraph{Case III: Strictly identifiable regime.}
\begin{align*}
\exp\left(-d_{KL}(G_1^{\id,N}\mid\mid G_0^{\id,N})\right) \geq \exp\left(-2N\cdot\frac{(2\delta^{\id})^2}{1-\rho+\delta^{\id}}\right) \geq \exp\left(-8N\cdot\frac{(\delta^{\id})^2}{1-\rho}\right).
\end{align*}
\hfill\Halmos
\end{proof}

\section{Analysis of \ALGplus}\label{apx:rcn-plus}

\minedit{
\begin{theorem}
Fix $\delta \in (0,1)$, and let {$\width_k = \sqrt{\frac{\log(2(K-1)/\delta)}{2N_k}}$ for $k \neq K$ and $\width_K = \sqrt{\frac{\log(2/\delta)}{2N_k}}$}. With probability at least $1 - {3}\delta$, {\Cref{alg:newsvendor_plus_plus}} outputs an ordering quantity $\qalg$ such that:
\begin{enumerate}[(i)]
\item if $\Gminus(\qoff_k) \geq \rho + 2 \width_k$ for some $k \in [K]$:
\[ \Regret(\qalg) \leq \lambda(b+h)\sqrt{\frac{\log(2/\delta)}{2\sum_{k\in\mathcal{U}_1}N_k}}, \]
where $\mathcal{U}_1 = \{k: \Gminus(\qoff_k) \geq \rho + 2\width_k\}$.
\item if $\Gminus(\lambda) \geq \rho$ and $\Gminus(\qoff_k) < \rho + 2\width_k$ for all $k \in [K]$:
\[\Regret(\qalg) \leq \lambda(b+h)\max\left\{2\width_K, \sqrt{\frac{\log(2/\delta)}{2\min_{k\in\mathcal{U}_2}N_k}}\right\},\]
where $\mathcal{U}_2 = \{k: \Gminus(\qoff_k) \geq \rho\}$.
\item if {$\Gminus(\lambda) \in [\rho - 2 \width_K, \rho)$}:
\[\Regret(\qalg) \leq \risk + {2 \max\left\{\frac{b}{h}, 1\right\}(M - \lambda)}(b+h) \width_K\]
\item if {$\Gminus(\lambda) < \rho - 2 \width_K$}:
\[\Regret(\qalg) \leq \risk + {\max\left\{\frac{b}{h},1\right\}(M-\lambda)}(b+h)\width_K.\]
\end{enumerate}
\end{theorem}

\begin{proof}{\it Proof.}
Abusing notation, we define $\qcrithat = \frac{bM+h\lambda-(b+h)\Gminushatk{K}(\lambda)M}{(b+h)(1-\Gminushatk{K}(\lambda))}$. We also let $\Nest = \sum_{k\in\Uest}N_k$, \mbox{$\widehat{G}(x) = \frac{1}{\Nest}\sum_{k\in\Uest}\sum_{i\in[N_k]}\Ind{\soff_{ki} \leq x}$}, and $\offeCDFtrue(x) = \frac{1}{\Nest}\sum_{k\in\Uest}\sum_{i\in[N_k]}\Ind{\doff_{ki} \leq x}$.

Recall, by \Cref{prop:regret-decomp}, 
\begin{align}\label{eq:regret-decomp-2}
\Regret(\qalg) \leq \Delta + \sup_{F \in \ambig{\lambda}} C_F(\qalg) - C_F(\qrisk).
\end{align}

Therefore, as in the proof of \Cref{thm:minmax_regret}, we focus on bounding $\sup_{F \in \ambig{\lambda}} C_F(\qalg) - C_F(\qrisk)$. We will rely heavily on the following facts in the remainder of our proof.

\begin{lemma}\label{lem:g_minus_concentration_plus}
Let $\Event = \left\{ \abs{\Gminushatk{k}(\qoff_k) - \Gminus(\qoff_k)} {<} \width_k \text{ for all } k \in [K] \right\}$. Then, $\Pr_G(\Event) \geq 1-{2}\delta$.
\end{lemma}

\Cref{lem:g_minus_concentration_plus} is an analog of \Cref{lem:g_minus_concentration}, where we additionally union bound over all selling seasons $k \in [K]$. It similarly follows from the fact that $\Ind{\soff_{ki} < \qoff_k} = \Ind{\doff_{ki} < \qoff_k}$ for all $k \in [K]$, $i \in [N_k]$. We omit the proof of this fact as such.

With these facts in hand, we prove each of the regret bounds separately, beginning with the two ``extreme'' cases of $\Gminus(\qoff_k) \geq \rho +2\width_k$ for some $k \in [K]$ and {$\Gminus(\lambda) < \rho-2\width_K$}. Moreover, in the remainder of the proof we condition on the good event $\Event$ defined in \Cref{lem:g_minus_concentration_plus}.

\smallskip 

\paragraph{Case I: $\Gminus(\qoff_k) \geq \rho +2\width_k$ for some $k \in [K]$.} In this case, the following facts hold:
\begin{enumerate}
\item $\Gminus(\lambda)\geq \rho$. Therefore, $\qopt_G < \lambda$.
\item Under event $\Event$:
\begin{enumerate}
\item For all $k \in \mathcal{U}_1$, $\Gminushatk{k}(\qoff_k) \geq \rho + \width_k$. Hence, $\Uest \neq \emptyset$, and $\qalg = \inf\left\{x \ \mid \ \widehat{G}(x) \geq \rho\right\} < \min_{k\in\Uest} \qoff_k$, where the strict inequality holds because $\Gminushatk{k}(\qoff_k) > \rho$ for all $k \in \Uest$. 
\item For all $k \in \Uest$, $\Gminus(\qoff_k) \geq \rho$. Therefore, $\qopt_G < \min_{k \in \Uest} \qoff_k$. 
\end{enumerate}
\end{enumerate}

Since $\Gminus(\lambda) \geq \rho$, we can apply \Cref{prop:worst-case-for-identifiable}, obtaining:
\begin{align}\label{eq:plus-likely-identifiable}
    \Regret(\qalg) &= C_G(\qalg)-C_G(\qopt_G).
\end{align}

Suppose first that $\qalg \leq \qopt_G$. By the same arguments as those used to establish \eqref{eq:first-case-a}, we have:
\begin{align*}
C_G(\qalg)-C_G(\qopt_G) &\leq (b+h)(\qopt_G-\qalg)(\rho-G(\qalg)) \\
&\leq (b+h)\lambda(\widehat{G}(\qalg)-G(\qalg)),
\end{align*}
where we use the fact that $\widehat{G}(\qalg) \geq \rho$ by definition, and moreover $\max\{\qopt_G, \qalg\} \leq \lambda \implies \qopt_G - \qalg \leq \lambda$. 

Now, since  $\qalg \leq \qopt_G$ and $\qopt_G < \min_{k \in \Uest}\qoff_k$, $\qalg < \min_{k \in \Uest}\qoff_k$. Hence, $\Ind{\soff_{ki} \leq \qalg} = \Ind{\doff_{ki} \leq \qalg}$ for all $k \in \Uest$, $i \in [N_k]$. Therefore, $\widehat{G}(\qalg) = \offeCDFtrue(\qalg)$. Plugging this into the above, we obtain:
\begin{align*}
C_G(\qalg)-C_G(\qopt_G) &\leq (b+h)\lambda(\offeCDFtrue(\qalg)-G(\qalg)).
\end{align*}
Applying the DKW inequality, we have that with probability at least $1-\delta$, 
\begin{align*}
C_G(\qalg)-C_G(\qopt_G) &\leq (b+h)\lambda \sqrt{\frac{\log(2/\delta)}{2\Nest}}.
\end{align*}

Suppose now that $\qalg > \qopt_G$. In this case, by the same arguments as those used to establish \eqref{eq:use-for-rcn-plus}, we have:
\begin{align*}
C_G(\qalg)-C_G(\qopt_G) &\leq (b+h)(\qalg-\qopt_G)(\Gminus(\qalg)-\rho) \\
&\leq (b+h)(\qalg-\qopt_G)\left(\lim_{x\to\qalg}{\offeCDFtrue}(x)+\sqrt{\frac{\log(2/\delta)}{2\Nest}}-\rho\right)\\
&\leq (b+h)(\qalg-\qopt_G)\left(\lim_{x\to\qalg}{\widehat{G}}(x)+\sqrt{\frac{\log(2/\delta)}{2\Nest}}-\rho\right) \\
&\leq (b+h)\lambda\sqrt{\frac{\log(2/\delta)}{2\Nest}},
\end{align*}
where the second inequality applies the DKW inequality, the third inequality uses the fact that $\Ind{\soff_{ki} \leq x} \geq \Ind{\doff_{ki} \leq x}$ for all $x$, the fourth inequality uses the fact that $\lim_{x\to\qalg}{\widehat{G}}(x) < \rho$ by definition of $\qalg$, and the final inequality uses the fact that $\qalg \leq \lambda$.

Using the fact that $k \in \mathcal{U}_1 \implies k \in {\Uest}$, we have that $\Nest \geq \sum_{k\in\mathcal{U}_1}N_k$. Applying this bound to the above concludes the proof of the first case. 

\smallskip 

\paragraph{Case II: $\Gminus(\lambda) < \rho - 2 \width_K$.}
In this case, by \Cref{thm:minimax-risk-identifiable}, $\qrisk = \qcrit_G$. Moreover, under event $\Event$, our algorithm outputs $\qalg = \qcrithat$. To see why this is the case, observe that, under $\Event$, $\Gminushatk{K}(\lambda) < \rho-\width_K$. It remains to establish that $\Gminushatk{k}(\qoff_k) < \rho + \width_k$ for all $k \leq K-1$. Suppose for contradiction that $\Gminushatk{k}(\qoff_k) \geq \rho + \width_k$ for some $k \leq K-1$. Then, under $\Event$, we must have $\Gminus(\qoff_k) \geq \rho$, a contradiction, since $\Gminus(\lambda) < \rho-2\width_K$ by assumption.

Hence, the same analysis as that of Case II in the proof of \Cref{thm:minmax_regret} applies, and we obtain:
\begin{align}\label{eq:final-case-2a-plus}
\Regret(\qalg) &\leq \risk + \frac{\max\{b,h\}(M-\lambda)}{1-\rho}\width_K,
\end{align}
completing the proof of the second case.

\smallskip 

Having analyzed the two extreme cases, we proceed to the ``hard'' regime, i.e., $\Gminus(\lambda) \in [\rho-2\width_K, \rho+2\width_K)$.

\smallskip

\paragraph{Case III: $\Gminus(\lambda)\in[\rho,\rho+2\width_K)$, $\Gminus(\qoff_k) < \rho + 2\width_k$ for all $k \leq K-1$.} By \Cref{thm:minimax-risk-identifiable}, $\qrisk = \qopt_G$, and $\risk= 0$.  Moreover, given $\Event$, \mbox{$\Gminushatk{k}(\qoff_k) \in {(}\rho - \width_k, \rho + 3 \width_k)$} for all $k \in \mathcal{U}_2$, and $\Gminushat(\qoff_k) < \rho + \width_k$ for all $k \not\in\mathcal{U}_2$.
\begin{enumerate}[a.]
\item Suppose first that $\Gminushatk{k}(\qoff_k) \in [\rho+\width_k,\rho+3\width_k)$ for some $k \in \mathcal{U}_2$. Then, $\Uest \neq \emptyset$, and identical arguments as those used for Case I yield:
\begin{align}\label{eq:case2-saa-bound-plus}
\Regret(\qalg)& \leq (b+h)\lambda\sqrt{\frac{\log(2/\delta)}{2\Nest}}.
\end{align}
Using the fact that $\Nest \geq \min_{k\in\mathcal{U}_2} N_k$ in this case,
we obtain:
\begin{align}\label{eq:case3-saa-bound-plus}
\Regret(\qalg)& \leq (b+h)\lambda\sqrt{\frac{\log(2/\delta)}{2\min_{k\in\mathcal{U}_2}N_k}}.
\end{align}
\item Suppose now that $\Gminushatk{k}(\qoff_k) \in {(}\rho - \width_k, \rho + \width_k)$ for all $k \in \mathcal{U}_2$. In this case, $\qalg = \lambda$. Using the same arguments as those used in Case III of the proof of \Cref{thm:minmax_regret}, we have:
\begin{align*}
 \Regret(\qalg) &\leq (b+h)(\lambda-\qopt_G)(\Gminus(\lambda)-\rho).
\end{align*}
Using the fact that $\qopt_G \in [0,\lambda)$ for $\Gminus(\lambda) \geq \rho$, and $\Gminus(\lambda) \leq \rho + 2\width_K$ by assumption, we obtain the final bound of: 
\[\Regret(\qalg) \leq 2\lambda(b+h)\width_K.\]
\end{enumerate}
\smallskip 
\paragraph{Case IV: $\Gminus(\lambda) \in [\rho-2\width_K,\rho)$.} By \Cref{thm:minimax-risk-identifiable}, $\qrisk = \qcrit_G$. Moreover, under $\Event$, \mbox{$\Gminushatk{K}(\lambda) \in {(}\rho - 3 \width_K, \rho + \width_K)$}, and $\Gminushatk{k}(\qoff_k) < \rho + \width_k$ for all $k \leq K-1$.
\begin{enumerate}[a.]
\item Suppose first that $\Gminushatk{K}(\lambda) < \rho-\width_K$.
In this case, the algorithm outputs $\qalg=\qcrithat$. By the same arguments as those used in Case II, we obtain:
\begin{align*}
\Regret(\qalg) {\leq \risk + \frac{\max\{b,h\}(M-\lambda)}{(1-\rho)}\width_K}.
\end{align*}
\item Suppose now that $\Gminushatk{K}(\lambda) \geq \rho-\width_K$. In this case, $\qalg = \lambda$. By the same arguments as those used in Case IV of the proof of \Cref{thm:minmax_regret}, we have:
\begin{align*}
\Regret(\qalg)\leq \risk + \max\{b,h\}\frac{2(M-\lambda)}{1-\rho}\width_K.
\end{align*}
\end{enumerate}
\hfill\Halmos
\end{proof}
}

\newpage 
\section{Computational Experiments: Additional Details}
\label{sec:experiment_details}

\subsection{Simulation Information}
\label{app:simulation_details}

\subsubsection*{Computing Infrastructure:} The experiments were conducted on a personal computer with an Apple M2, 8-core processor and 16.0GB of RAM.

\medskip 

\subsection{Synthetic Experiments: Supplemental Results}

\subsubsection{Impact of the Observable Boundary.\label{apx:observe-boundary}} 
\Cref{tab:apx-rel-vs-lam-75,tab:apx-rel-vs-lam-98} show the impact of $\lambda$ on policy performance, for $\rho \in \{ 0.75, 0.98\}$. Trends are identical to those identified for $\rho = 0.9$. Note that, for $\rho = 0.98$, all non-robust benchmarks exhibit extremely high relative regret for small values of $\lambda$, as the worst-case distribution forces these policies to incur exorbitant underage costs in the unidentifiable regime.

\begin{table}
\setlength\tabcolsep{3pt}
\small
\centering
\begin{subtable}[b]{\linewidth}
\centering
\begin{tabular}{lrrr!{\vrule width 1.5pt}rrrrr}
\toprule
$\lambda$ & 44.50 & 57.21 & 69.93 & 82.64 & 95.36 & 108.07 & 120.79 & 133.50 \\
\midrule
\TrueSAA &  &  &  & 0.0 (0.0\%) & 0.0 (0.1\%) & 0.0 (0.1\%) & 0.0 (0.1\%) & 0.0 (0.1\%) \\
\midrule
\ALG & 4.5 (3.0\%) & 5.4 (5.0\%) & 11 (26\%) & 0.2 (0.4\%) & 0.1 (0.2\%) & 0.1 (0.2\%) & 0.1 (0.1\%) & 0.1 (0.2\%) \\
\CensoredSAA & 184 (120\%) & 74 (68\%) & 8.8 (21\%) & 0.1 (0.2\%) & 0.1 (0.1\%) & 0.1 (0.2\%) & 0.1 (0.1\%) & 0.0 (0.1\%) \\
\KM & 184 (120\%) & 74 (68\%) & 8.8 (21\%) & 0.0 (0.1\%) & 0.0 (0.1\%) & 0.1 (0.1\%) & 0.0 (0.1\%) & 0.0 (0.1\%) \\
\IgnorantSAA & 184 (120\%) & 80 (74\%) & 21 (49\%) & 12 (31\%) & 9.0 (24\%) & 7.3 (19\%) & 5.6 (15\%) & 4.8 (13\%) \\
\SubsampleSAA & 209 (137\%) & 97 (89\%) & 27 (63\%) & 10.0 (27\%) & 3.9 (10\%) & 2.7 (7.2\%) & 2.3 (6.2\%) & 2.3 (6.2\%) \\
\bottomrule
\end{tabular}
\caption{Uniform, $\qopt_G = 74$}
\label{tab:uniform_75}       
\end{subtable}

\medskip
\begin{subtable}[b]{\linewidth}
\centering
\begin{tabular}{lr!{\vrule width 1.5pt}rrrrrrr}
\toprule
$\lambda$ & 92.07 & 118.38 & 144.68 & 170.99 & 197.30 & 223.60 & 249.91 & 276.22 \\
\midrule
\TrueSAA & & 0.1 (0.1\%) & 0.1 (0.1\%) & 0.1 (0.1\%) & 0.1 (0.1\%) & 0.1 (0.1\%) & 0.1 (0.1\%) & 0.1 (0.1\%) \\
\midrule
\ALG & 11 (23\%) & 0.4 (0.3\%) & 0.7 (0.7\%) & 0.3 (0.3\%) & 0.2 (0.2\%) & 0.3 (0.3\%) & 0.2 (0.2\%) & 0.2 (0.2\%) \\
\CensoredSAA & 13 (27\%) & 0.2 (0.2\%) & 0.2 (0.2\%) & 0.2 (0.2\%) & 0.2 (0.2\%) & 0.1 (0.1\%) & 0.2 (0.2\%) & 0.1 (0.1\%) \\
\IgnorantSAA & 34 (70\%) & 19 (17\%) & 13 (12\%) & 7.6 (6.9\%) & 7.5 (6.7\%) & 5.5 (5.0\%) & 3.3 (2.9\%) & 1.4 (1.3\%) \\
\KM & 13 (27\%) & 0.2 (0.2\%) & 0.2 (0.2\%) & 0.2 (0.2\%) & 0.2 (0.1\%) & 0.2 (0.1\%) & 0.2 (0.1\%) & 0.1 (0.1\%) \\
\SubsampleSAA & 46 (95\%) & 26 (23\%) & 18 (17\%) & 13 (12\%) & 11 (9.7\%) & 8.3 (7.5\%) & 6.4 (5.7\%) & 4.5 (4.1\%) \\
\bottomrule
\end{tabular}
\caption{Exponential, $\qopt_G = 110.90$}
\label{tab:exponential_75}       
\end{subtable}

\medskip

\begin{subtable}[b]{\linewidth}
\centering
\begin{tabular}{lrrrr!{\vrule width 1.5pt}rrrr}
\toprule
$\lambda$ & 46 & 59.14 & 72.29 & 85.43 & 98.57 & 111.71 & 124.86 & 138 \\
\midrule
\TrueSAA & &  &  &  & 0.0 (0.1\%) & 0.0 (0.2\%) & 0.0 (0.1\%) & 0.0 (0.2\%) \\
\midrule
\ALG & 0.0 (0.0\%) & 0.5 (0.2\%) & 2.4 (1.4\%) & 1.0 (6.8\%) & 0.0 (0.2\%) & 0.0 (0.3\%) & 0.0 (0.3\%) & 0.0 (0.3\%) \\
\CensoredSAA & 628 (300\%) & 590 (297\%) & 380 (219\%) & 0.6 (4.4\%) & 0.0 (0.2\%) & 0.0 (0.3\%) & 0.0 (0.2\%) & 0.0 (0.2\%) \\
\KM & 628 (300\%) & 590 (297\%) & 380 (219\%) & 0.6 (4.4\%) & 0.0 (0.2\%) & 0.0 (0.3\%) & 0.0 (0.2\%) & 0.0 (0.3\%) \\
\IgnorantSAA & 628 (300\%) & 590 (297\%) & 380 (219\%) & 3.6 (26\%) & 3.0 (26\%) & 2.8 (25\%) & 2.2 (19\%) & 2.1 (18\%) \\
\SubsampleSAA & 628 (300\%) & 593 (298\%) & 383 (221\%) & 2.6 (18\%) & 0.1 (0.5\%) & 0.1 (1.1\%) & 0.2 (1.6\%) & 0.2 (1.5\%) \\
\bottomrule
\end{tabular}
\caption{Poisson, $\qopt_G = 86$}
\label{tab:poisson_75}       
\end{subtable}
\caption{
Impact of $\lambda$ on policy performance. Here, $\rho = 0.75$. Values to the left of the thick vertical line correspond to the unidentifiable regime, where we report $\Regret(q^\pi) - \risk$ and $\mathcal{R}^{ui}(q^\pi)\%$; values to the right of the thick vertical line correspond to the identifiable regime, where we report $\UncensoredRegret(q^\pi)$ and $\mathcal{R}^{id}(q^\pi)\%$.}\label{tab:apx-rel-vs-lam-75}
\end{table}

\begin{table*}[!t]
\setlength\tabcolsep{2pt} 
\small
\centering

\begin{subtable}[b]{0.95\linewidth}
\centering
\resizebox{\linewidth}{!}{%
\begin{tabular}{lrrrrr!{\vrule width 1.5pt}rrr}
\toprule
$\lambda$ & 44.50 & 57.21 & 69.93 & 82.64 & 95.36 & 108.07 & 120.79 & 133.50 \\
\midrule
\TrueSAA &  &  &  &  & & 0.0 (0.0\%) & 0.0 (0.0\%) & 0.0 (0.1\%) \\
\midrule
\ALG & 4.7 (1.7\%) & 6.9 (2.7\%) & 6.9 (2.9\%) & 11 (5.4\%) & 115 (100\%) & 0.0 (0.1\%) & 0.1 (0.1\%) & 0.1 (0.3\%) \\
\CensoredSAA & 7165 (2651\%) & 5102 (2001\%) & 3335 (1401\%) & 1605 (751\%) & 116 (101\%) & 1.2 (2.4\%) & 0.9 (1.9\%) & 2.5 (5.1\%) \\
\KM & 7165 (2651\%) & 5102 (2001\%) & 3335 (1401\%) & 1605 (751\%) & 116 (101\%) & 0.0 (0.1\%) & 0.0 (0.1\%) & 0.1 (0.2\%) \\
\IgnorantSAA & 7165 (2651\%) & 5102 (2001\%) & 3335 (1401\%) & 1605 (751\%) & 116 (101\%) & 1.2 (2.4\%) & 1.3 (2.7\%) & 1.3 (2.7\%) \\
\SubsampleSAA & 7207 (2666\%) & 5135 (2014\%) & 3377 (1419\%) & 1629 (762\%) & 121 (105\%) & 0.4 (0.9\%) & 0.7 (1.3\%) & 0.7 (1.5\%) \\
\bottomrule
\end{tabular}}
\caption{Uniform, $\qopt_G = 97$}
\label{tab:uniform_98}
\end{subtable}

\vspace{0.5em}

\begin{subtable}[b]{0.95\linewidth}
\centering
\resizebox{\linewidth}{!}{%
\begin{tabular}{lrrrrrrrr}
\toprule
$\lambda$ & 92.07 & 118.38 & 144.68 & 170.99 & 197.30 & 223.60 & 249.91 & 276.22 \\
\midrule
\TrueSAA &  &  & & & & & & \\
\midrule
\ALG & 6.2 (2.8\%) & 7.6 (4.0\%) & 7.6 (4.8\%) & 7.7 (6.0\%) & 69 (71\%) & 111 (163\%) & 48 (118\%) & 11 (59\%) \\
\CensoredSAA & 3232 (1482\%) & 1956 (1038\%) & 1138 (720\%) & 626 (490\%) & 316 (324\%) & 140 (206\%) & 51 (124\%) & 13 (71\%) \\
\IgnorantSAA & 3232 (1482\%) & 1956 (1038\%) & 1138 (720\%) & 626 (490\%) & 316 (324\%) & 140 (206\%) & 55 (136\%) & 30 (169\%) \\
\KM & 3232 (1482\%) & 1956 (1038\%) & 1138 (720\%) & 626 (490\%) & 316 (324\%) & 140 (206\%) & 49 (120\%) & 11 (61\%) \\
\SubsampleSAA & 3326 (1525\%) & 2053 (1090\%) & 1246 (788\%) & 737 (576\%) & 424 (435\%) & 246 (362\%) & 144 (352\%) & 101 (564\%) \\
\bottomrule
\end{tabular}}
\caption{Exponential, $\qopt_G = 312.96$}
\label{tab:exponential_98}
\end{subtable}

\vspace{0.5em}

\begin{subtable}[b]{0.95\linewidth}
\centering
\resizebox{\linewidth}{!}{%
\begin{tabular}{lrrrrr!{\vrule width 1.5pt}rrr}
\toprule
$\lambda$ & 46 & 59.14 & 72.29 & 85.43 & 98.57 & 111.71 & 124.86 & 138 \\
\midrule
\TrueSAA &  &  &  &  & & 0.1 (0.6\%) & 0.2 (0.7\%) & 0.1 (0.5\%) \\
\midrule
\ALG & 0.0 (0.0\%) & 0.5 (0.2\%) & 2.5 (1.0\%) & 8.1 (3.7\%) & 3.1 (16\%) & 0.2 (1.0\%) & 0.2 (1.1\%) & 0.2 (1.0\%) \\
\CensoredSAA & 13397 (4900\%) & 12652 (4857\%) & 9578 (3887\%) & 2717 (1226\%) & 3.3 (16\%) & 0.6 (2.5\%) & 0.9 (4.2\%) & 1.9 (8.4\%) \\
\KM & 13397 (4900\%) & 12652 (4857\%) & 9578 (3887\%) & 2717 (1226\%) & 3.2 (16\%) & 0.2 (1.0\%) & 0.2 (1.1\%) & 0.2 (0.9\%) \\
\IgnorantSAA & 13397 (4900\%) & 12652 (4857\%) & 9578 (3887\%) & 2717 (1226\%) & 4.8 (24\%) & 1.7 (7.8\%) & 1.7 (7.5\%) & 1.5 (6.7\%) \\
\SubsampleSAA & 13397 (4900\%) & 12680 (4868\%) & 9589 (3892\%) & 2722 (1229\%) & 5.2 (26\%) & 0.3 (1.4\%) & 0.4 (1.7\%) & 0.5 (2.4\%) \\
\bottomrule
\end{tabular}}
\caption{Poisson, $\qopt_G = 99$}
\label{tab:poisson_98}
\end{subtable}
\caption{
Impact of $\lambda$ on policy performance. Here, $\rho = 0.98$. Values to the left of the thick vertical line correspond to the unidentifiable regime, where we report $\Regret(q^\pi) - \risk$ and $\mathcal{R}^{ui}(q^\pi)\%$; values to the right of the thick vertical line correspond to the identifiable regime, where we report $\UncensoredRegret(q^\pi)$ and $\mathcal{R}^{id}(q^\pi)\%$.}
\label{tab:apx-rel-vs-lam-98}
\end{table*}

\clearpage

\subsubsection{Impact of Demand Variability.\label{apx:demand-var}}  In \cref{tab:normal_full} we illustrate the impact of demand variability for additional values of $\lambda$, fixing $\rho = 0.9$.  (Results are identical for $\rho \in \{0.75,0.98\}$.)

When $\lambda = 59.23$, the problem is easily unidentifiable across all values of $\sigma$. Our algorithm achieves less than 1.5\% relative regret in this regime, with all other algorithms achieving over 598\% relative regret.  When $\lambda = 177.69$, on the other hand, the problem is easily identifiable for all values of $\sigma$, with \ALG achieving a relative regret of less than 0.5\% across all such values.

\begin{table}[!t]
    \setlength\tabcolsep{3pt}
    \small
\resizebox{\textwidth}{!}{%
\begin{tabular}{lrrrrr!{\vrule width 2pt}rrrrr}
\toprule
$\lambda$ & \multicolumn{5}{c}{59.23} & \multicolumn{5}{c}{177.69} \\
\midrule
$\sigma$ & 20 & 25 & 30 & 35 & 40 & 20 & 25 & 30 & 35 & 40 \\
\midrule
\TrueSAA &  &  &  &  &  & 0.1 (0.2\%) & 0.1 (0.1\%) & 0.1 (0.1\%) & 0.1 (0.2\%) & 0.1 (0.2\%) \\
\midrule
\ALG & 2.1 (0.9\%) & 2.6 (1.1\%) & 3.1 (1.4\%) & 3.2 (1.4\%) & 3.0 (1.3\%) & 0.1 (0.4\%) & 0.1 (0.3\%) & 0.2 (0.3\%) & 0.2 (0.4\%) & 0.2 (0.3\%) \\
\CensoredSAA & 1760 (751\%) & 1620 (697\%) & 1512 (656\%) & 1428 (623\%) & 1363 (598\%) & 0.1 (0.3\%) & 0.1 (0.3\%) & 0.3 (0.5\%) & 0.3 (0.5\%) & 0.3 (0.4\%) \\
\KM & 1760 (751\%) & 1620 (697\%) & 1512 (656\%) & 1428 (623\%) & 1363 (598\%) & 0.1 (0.3\%) & 0.1 (0.2\%) & 0.2 (0.3\%) & 0.2 (0.3\%) & 0.2 (0.2\%) \\
\IgnorantSAA & 1760 (751\%) & 1620 (697\%) & 1512 (656\%) & 1428 (623\%) & 1363 (598\%) & 2.6 (7.4\%) & 3.8 (8.5\%) & 4.3 (8.1\%) & 5.4 (8.9\%) & 7.5 (11\%) \\
\SubsampleSAA & 1773 (756\%) & 1636 (704\%) & 1533 (665\%) & 1453 (635\%) & 1393 (612\%) & 1.0 (2.8\%) & 1.4 (3.2\%) & 1.4 (2.7\%) & 2.5 (4.1\%) & 3.3 (4.7\%) \\
\bottomrule
\end{tabular}
}
        \caption{
        Impact of $\sigma$ on policy performance. Here, $\rho = 0.9$. At $\lambda = 59.23$, the problem is unidentifiable for all values of $\sigma$; in this case we report $\Regret(q^\pi) - \risk$ and $\mathcal{R}^{ui}(q^\pi)\%$. At $\lambda=177.69$, the problem is identifiable for all values of $\sigma$; in this case we report $\UncensoredRegret(q^\pi)$ and $\mathcal{R}^{id}(q^\pi)\%$.}
    \label{tab:normal_full}
\end{table}

\clearpage

\subsubsection{\minedit{Vanilla Regret Comparison in the ``Clearly Unidentifiable'' Regime}.}
\label{app:synthetic_vanilla_regret}
While \ALG achieves the optimal worst-case regret guarantee in the unidentifiable regime, a natural question is whether its strong performance extends to {\it vanilla regret} (see \cref{eq:uncensored-regret}) in this regime, evaluated with respect to the true distribution $G$. In this section we provide numerical support for the idea that vanilla regret is not the meaningful regret metric in the unidentifiable regime. In particular, our results will illustrate the intuitive fact that the performance of all non-robust algorithms is poor in the unidentifiable regime, but improves as $\lambda$ increases. By contrast, the performance of \ALG is highly dependent on the gap between $\qopt_G$ and $\qcrit_G$, which in turn depends on $\lambda$ and the tightness of the upper bound $M$.

\paragraph{Setup.} We consider the {uniform}, exponential, and Poisson distributions described in \Cref{sec:experiments_synthetic}. Across all experiments, we let $h = 1$ and $b = 9$ (i.e., $\rho = 0.9$). We moreover let $M = 100$ in the case of the uniform distribution, and $M = 325$ for the exponential and Poisson distributions. \minedit{The sampling of $\soff$ and the order levels is identical to that of \cref{sec:experiment_setup}, where we set $N = 500$.} \minedit{We report the vanilla regret of the five policies of interest in the ``clearly unidentifiable'' regime, where $\lambda \ll \qopt_G$, in \Cref{tab:vanilla-rel-vs-lam-09}. We study the knife-edge regime, where $\lambda \approx \qopt_G$ in Appendix \ref{apx:knife-edge}. Finally, we omit the ``clearly identifiable'' regime, where $\lambda \gg \qopt_G$, as this regime is already covered by \Cref{tab:rel-vs-lam-09}.} 

\paragraph{Results.} For the uniform distribution (\Cref{tab:uniform_vanilla}), it is easy to show that $\qopt_G = \qcrit_G$ {across all values of $\lambda$}. Since \ALG outputs an empirical estimate of $\qcrit_G$ in this setting, it achieves near-zero vanilla regret across all values of $\lambda$. In the worst case, its additive vanilla regret is 0.9 (corresponding to a relative vanilla regret of 2.1\%) when $\lambda = 82.64$.  In contrast, the other policies perform very poorly for small values of $\lambda$, with a relative vanilla regret of 300\% in the worst case, achieved by \SubsampleSAA. Since \CensoredSAA, \KM and \IgnorantSAA output $\lambda$ in almost all replications in the unidentifiable regime, their performance improves as $\lambda$ increases, thereby approaching $\qopt_G$. Similarly, the performance of \SubsampleSAA improves, since the fraction of samples for which $d_{ki} < \qoff_k$ increases with $\lambda$.

Under the exponential distribution (\Cref{tab:exponential_vanilla}), \ALG continues to perform well in the unidentifiable regime, despite the fact that $\qopt_G \neq \qcrit_G$ {($\qopt_G = 184.21$ and $\qcrit_G \in \{251.32, 234.17, 214.81, 194.24\}$ for increasing unidentifiable values of $\lambda$)}. Its cost exceeds the optimal newsvendor cost by 12\% in the worst case, with improved performance as $\lambda$ increases and $\qcrit_G$ approaches $\qopt_G$. By contrast, the remaining policies exhibit significantly higher regret (up to $79\%$ for \SubsampleSAA), as for the uniform distribution.

Under the Poisson distribution (\Cref{tab:poisson_vanilla}), the difference between $\qopt_G$ and $\qcrit_G$ is significantly higher than in the previous two cases {($\qopt_G = 92$ versus $\qcrit_G \in \{297.10, 298.14, 293.32, 234.82\}$ for increasing unidentifiable $\lambda$)}. The gap between $\qcrit_G$ and $\qopt_G$ --- explained by the very loose upper bound $M$ on $\qopt_G$ --- results in the poor performance of \ALG; in the best case ($\lambda = 85.43$), its cost is over 9 times the optimal newsvendor cost. 
Finally, we note that when $\lambda = 46$, our algorithm --- which outputs an estimate of $\qcrit_G$ --- outperforms all baselines. This is despite the fact that $\qcrit_G$ is over three times as high as $\qopt_G$ for this value of $\lambda$, and can be explained by the fact that $\rho = 0.9$, meaning that lost sales are far more heavily penalized than overage. Consequently, over-ordering relative to $\qopt_G$ mitigates expected cost more effectively than under-ordering by outputting $\lambda$, which the baselines do. As $\lambda$ increases and approaches $\qopt_G$, this advantage diminishes, since $\lambda$ approaches $\qopt_G$ more rapidly than $\qcrit_G$.

\begin{table}[!ht]
\setlength\tabcolsep{3pt}
\small
\begin{subtable}[b]{\linewidth}
\centering
\begin{tabular}{lrrrr}
\toprule
$\lambda$ & 44.50 & 57.21 & 69.93 & 82.64 \\
\midrule
\TrueSAA & 0.1 (0.2\%) & 0.1 (0.2\%) & 0.1 (0.2\%) & 0.1 (0.2\%) \\
\midrule
\ALG & 0.0 (0.1\%) & 0.1 (0.1\%) & 0.1 (0.1\%) & 0.9 (2.1\%) \\
\CensoredSAA & 101 (225\%) & 52 (115\%) & 19 (41\%) & 2.1 (4.7\%) \\
\KM & 101 (225\%) & 52 (115\%) & 19 (41\%) & 2.1 (4.7\%) \\
\IgnorantSAA & 101 (225\%) & 52 (115\%) & 19 (41\%) & 5.2 (12\%) \\
\SubsampleSAA & 135 (300\%) & 85 (188\%) & 45 (100\%) & 19 (42\%) \\
\bottomrule
\end{tabular}
\caption{Uniform, $\qopt_G = 89$}
\label{tab:uniform_vanilla}       
\end{subtable}
\medskip
\begin{subtable}[b]{\linewidth}
\centering
\begin{tabular}{lrrrr}
\toprule
$\lambda$ & 92.07 & 118.38 & 144.68 & 170.99 \\
\midrule
\TrueSAA & 0.6 (0.3\%) & 0.6 (0.3\%) & 0.5 (0.3\%) & 0.5 (0.3\%) \\
\midrule
\ALG & 21 (12\%) & 13 (7.0\%) & 9.6 (5.2\%) & 1.3 (0.7\%) \\
\CensoredSAA & 81 (44\%) & 37 (20\%) & 12 (6.4\%) & 1.6 (0.9\%) \\
\KM & 81 (44\%) & 37 (20\%) & 12 (6.4\%) & 1.5 (0.8\%) \\
\IgnorantSAA & 81 (44\%) & 37 (20\%) & 24 (13\%) & 26 (14\%) \\
\SubsampleSAA & 146 (79\%) & 103 (56\%) & 74 (40\%) & 54 (29\%) \\
\bottomrule
\end{tabular}
\caption{Exponential, $\qopt_G = 184.21$}
\label{tab:exponential_vanilla}       
\end{subtable}
\medskip
\begin{subtable}[b]{\linewidth}
\centering
\begin{tabular}{lrrrr}
\toprule
$\lambda$ & 46 & 59.14 & 72.29 & 85.43 \\
\midrule
\TrueSAA & 0.0 (0.2\%) & 0.0 (0.1\%) & 0.0 (0.2\%) & 0.0 (0.2\%) \\
\midrule
\ALG & 201 (1251\%) & 202 (1258\%) & 197 (1228\%) & 138 (857\%) \\
\CensoredSAA & 290 (1805\%) & 172 (1069\%) & 63 (389\%) & 4.6 (29\%) \\
\KM & 290 (1805\%) & 172 (1069\%) & 63 (389\%) & 4.6 (29\%) \\
\IgnorantSAA & 290 (1805\%) & 172 (1069\%) & 63 (389\%) & 4.6 (29\%) \\
\SubsampleSAA & 290 (1806\%) & 179 (1114\%) & 65 (405\%) & 8.5 (53\%) \\
\bottomrule
\end{tabular}
\caption{Poisson, $\qopt_G = 92$}
\label{tab:poisson_vanilla}       
\end{subtable}
\centering
\caption{Comparison of policies' vanilla regret in the \minedit{``clearly unidentifiable''} regime. Here, $\rho = 0.9$.}
\label{tab:vanilla-rel-vs-lam-09}
\end{table}

\clearpage

\minedit{\subsubsection{Algorithm Performance in the ``Knife-Edge'' Regime.}\label{apx:knife-edge}
We next closely study the impact of $N$ and $\lambda$ on the performance of \ALG, \CensoredSAA, and \KM in the knife-edge regime, where $\lambda$ lies near $\qopt_G$. (We omit \SubsampleSAA and \IgnorantSAA, given their clear suboptimality.) We consider four equally spaced values of $\lambda \in [0.95\qopt_G, 1.05\qopt_G]$. For each value of $\lambda$, we report algorithm performance for $N \in \{100, 250, 500\}$.

\smallskip 

\paragraph{Comparison of Worst-Case Regret.} We first compare the worst-case regret of our three policies in \Cref{tab:combined_ke_results}.  Recall, for $\lambda < \qopt_G$, we report  $\Regret(q^\pi)-\risk$, while in the identifiable regime we report $\UncensoredRegret(q^\pi)$.

\smallskip 

\noindent\textbf{Impact of $\lambda$.} We first fix $N = 500$ and compare algorithm performance as $\lambda$ increases. In the case of the uniform and exponential distributions (\Cref{tab:uniform_ke_worst_case,tab:exponential_ke_worst_case}), we observe that \ALG slightly outperforms \KM and \CensoredSAA for $\lambda < \qopt_G$, though the difference in regret is quite small. This is because, for all of these values of $\lambda$, \ALG classifies the problem as knife-edge (thereby outputting $\lambda$) in the vast majority of replications. \KM and \CensoredSAA similarly output $\lambda$ in the majority of replications. \ALG obtains a slight edge due to the minority of replications in which it correctly classifies the problem as unidentifiable, in which case it outputs an estimate of $\qcrit_G$. (For instance, this occurs in 30\% of replications when $\lambda = 84.55$ in the case of the uniform distribution.) We moreover observe an improvement in all three policies as they approach the boundary. This is because the distance between $\qcrit_G$ and $\lambda$ is decreasing in $\lambda$ when $\Gminus(\lambda) < \rho$; hence outputting $\lambda$ in the majority of replications is less costly. While we observe an improvement in all policies as $\lambda$ increases in the case of the Poisson distribution as well (\Cref{tab:poisson_ke_worst_case}), for $\lambda = 87.40$ we remark that the regret of \ALG is less than one-tenth that of \KM and \CensoredSAA. This follows from the fact that our algorithm correctly classifies the problem as unidentifiable in 99\% of replications.

To the right of the observable boundary, while \KM and \CensoredSAA are able to correctly output an estimate of $\qopt_G$, \ALG continues to output $\lambda$ in a majority of replications, for both the uniform and exponential distributions. However, given that $\lambda$ is close to $\qopt_G$ in this regime, it still exhibits strong performance, with a slight decrease as $\lambda$ increases, and the inability to classify the instance becomes more costly as a result. The Poisson distribution exhibits slightly different behavior as $\lambda$ increases. In this case, when $\lambda = 96.60$ the algorithm correctly classifies the problem as identifiable in 100\% of replications, resulting in a {\it decrease} in vanilla regret as $\lambda$ increases.

\smallskip 

\noindent\textbf{Impact of $N$.} For fixed $\lambda$, the worst-case regret of \ALG weakly decreases in $N$ in this regime. In the case of the uniform and exponential distributions, this is driven by the improvement in classification rates when $\lambda < \qopt_G$. For instance, under the uniform distribution, when $\lambda = 84.55$, the fraction of replications in which \ALG correctly classifies the problem as unidentifiable increases from 4\% when $N = 250$ to 30\% when $N = 500$. The improvement is less pronounced when $\lambda > \qopt_G$; the reason for this is that an increase in $N$ does not have a significant impact on the classification rate of \ALG in this regime. These results contrast with the Poisson distribution, where, at both ``knife-edge extremes'' of $\lambda \in \{87.4,96.6\}$, an increase in $N$ has a pronounced impact on the classification rate. In particular, for $\lambda = 87.4$, \ALG goes from correctly classifying the problem as unidentifiable in 18\% of replications when $N = 100$, to 99\% of replications when $N = 500$. Similarly, when $\lambda = 96.6$ the classification rate improves from 0\% to 100\%; here, however, there is no correspondingly large decrease in vanilla regret, since outputting $\lambda$ already achieves relatively low regret. The improvement in classification rate is not as significant for $\lambda \in \{90.35,93.53\}$; here, we conjecture that we would have to increase $N$ significantly (e.g., $N = 1000$) in order to observe a decrease in regret. Finally, we briefly remark on the impact of $N$ on the performance of \KM and \CensoredSAA when $\lambda < \qopt_G$. (When $\lambda > \qopt_G$, the fact that an increase in sample size improves vanilla regret is an expected and well-studied phenomenon.) This improvement follows from the fact that, as $N$ increases, the fraction of replications in which both algorithms output $\lambda$ increases (as $\Gminushat(\lambda) > \rho$ becomes an increasingly rare event). Since $\qcrit_G > \lambda$ in this regime, outputting $\lambda$ more frequently results in a performance improvement.

In summary, these results show that, while the knife-edge regime is the most difficult in terms of classification, all algorithms exhibit strong performance in the majority of tested instances. The reason for this is that these algorithms frequently output $\lambda$ in this setting (albeit for different reasons); by definition of the knife-edge regime, this quantity is guaranteed to be close to the optimum (be it $\qcrit_G$ or $\qopt_G$). Our results for the Poisson distribution, on the other hand, establish that the strong performance of \KM and \CensoredSAA in this regime is not robust to the true underlying demand distribution.

\begin{table}[t]
\centering
\small
\setlength\tabcolsep{3pt}
\minedit{
\begin{subtable}[b]{\linewidth}
\begin{tabular}{ll|rr!{\vrule width 1.5pt}rr}
\toprule
Algorithm & $N$ & 84.55 & 87.51 & 90.48 & 93.45 \\
\midrule

\multirow{3}{*}{\ALG}
& 100  & 2.6 (50\%) & 0.4 (19\%) & 0.0 (0.1\%) & 0.8 (1.7\%)  \\
& 250  & 2.6 (50\%) & 0.4 (19\%) & 0.0 (0.1\%) & 0.8 (1.7\%)  \\
& 500  & 2.1 (42\%) & 0.4 (19\%) & 0.0 (0.1\%) & 0.7 (1.6\%)  \\

\midrule

\multirow{3}{*}{\KM}
& 100  & 2.7 (52\%) & 0.6 (27\%) & 0.3 (0.7\%) & 0.5 (1.0\%)  \\
& 250  & 2.6 (50\%) & 0.5 (22\%) & 0.1 (0.2\%) & 0.2 (0.3\%)  \\
& 500  & 2.6 (50\%) & 0.4 (20\%) & 0.0 (0.1\%) & 0.1 (0.1\%)  \\

\midrule

\multirow{3}{*}{\CensoredSAA}
& 100  & 2.7 (52\%) & 0.7 (31\%) & 0.5 (1.1\%) & 0.6 (1.2\%)  \\
& 250  & 2.6 (50\%) & 0.5 (22\%) & 0.1 (0.2\%) & 0.2 (0.4\%)  \\
& 500  & 2.6 (50\%) & 0.4 (20\%) & 0.1 (0.1\%) & 0.1 (0.1\%) \\

\bottomrule
\end{tabular}
\centering
\caption{Uniform, $\qopt_G = 89$}
\label{tab:uniform_ke_worst_case}
\end{subtable}

\begin{subtable}[b]{\linewidth}
\begin{tabular}{ll|rr!{\vrule width 1.5pt}rr}
\toprule
Algorithm & $N$ & 174.96 & 181.10 & 187.24 & 193.38  \\
\midrule

\multirow{3}{*}{\ALG}
& 100  & 2.8 (17\%) & 1.1 (20\%) & 0.0 (0.0\%) & 0.4 (0.2\%)  \\
& 250  & 2.1 (13\%) & 0.3 (5.5\%) & 0.0 (0.0\%) & 0.4 (0.2\%)  \\
& 500  & 2.1 (13\%) & 0.3 (5.5\%) & 0.0 (0.0\%) & 0.4 (0.2\%)  \\

\midrule

\multirow{3}{*}{\KM}
& 100  & 3.2 (20\%) & 1.9 (35\%) & 1.2 (0.7\%) & 2.1 (1.1\%)  \\
& 250  & 2.4 (15\%) & 0.8 (15\%) & 0.5 (0.3\%) & 0.6 (0.4\%)  \\
& 500  & 2.3 (14\%) & 0.6 (12\%) & 0.2 (0.1\%) & 0.3 (0.2\%)  \\

\midrule

\multirow{3}{*}{\CensoredSAA}
& 100  & 3.8 (23\%) & 3.1 (59\%) & 1.9 (1.0\%) & 2.8 (1.5\%)  \\
& 250  & 2.7 (16\%) & 1.0 (19\%) & 0.8 (0.4\%) & 0.8 (0.4\%)  \\
& 500  & 2.3 (14\%) & 0.7 (13\%) & 0.3 (0.2\%) & 0.4 (0.2\%)  \\

\bottomrule
\end{tabular}
\centering
\caption{Exponential, $\qopt_G = 184.21$}
\label{tab:exponential_ke_worst_case}
\end{subtable}

\begin{subtable}[b]{\linewidth}
\begin{tabular}{ll|rr!{\vrule width 1.5pt}rr}
\toprule
Algorithm & $N$ & 87.40 & 90.46 & 93.53 & 96.60  \\
\midrule

\multirow{3}{*}{\ALG}
& 100  & 101 (86\%) & 9.6 (23\%) & 0.3 (2.1\%) & 1.9 (12\%)  \\
& 250  & 30 (26\%) & 8.7 (21\%) & 0.3 (2.1\%) & 0.7 (4.5\%)  \\
& 500  & 9.7 (8.2\%) & 9.6 (23\%) & 0.3 (2.0\%) & 0.0 (0.3\%)  \\

\midrule

\multirow{3}{*}{\KM}
& 100  & 117 (99\%) & 8.8 (21\%) & 0.2 (1.1\%) & 0.3 (1.7\%)  \\
& 250  & 117 (99\%) & 8.8 (21\%) & 0.1 (0.6\%) & 0.1 (0.5\%)  \\
& 500  & 117 (99\%) & 8.7 (21\%) & 0.0 (0.3\%) & 0.0 (0.3\%)  \\

\midrule

\multirow{3}{*}{\CensoredSAA}
& 100  & 117 (99\%) & 8.8 (21\%) & 0.2 (1.1\%) & 0.3 (1.8\%)  \\
& 250  & 117 (99\%) & 8.8 (21\%) & 0.1 (0.6\%) & 0.1 (0.5\%)  \\
& 500  & 117 (99\%) & 8.7 (21\%) & 0.1 (0.3\%) & 0.0 (0.3\%)  \\

\bottomrule
\end{tabular}
\centering
\caption{Poisson, $\qopt_G = 92$}
\label{tab:poisson_ke_worst_case}
\end{subtable}
}
\caption{\minedit{Impact of $\lambda$ and $N$ on policy performance in the knife-edge regime. Here, $\rho = 0.9$.  Values to the left of the thick vertical line correspond to the unidentifiable regime, where we report $\Regret(q^\pi) - \risk$ and $\mathcal{R}^{ui}(q^\pi)\%$; values to the right of the thick vertical line correspond to the identifiable regime, where we report $\UncensoredRegret(q^\pi)$ and $\mathcal{R}^{id}(q^\pi)\%$.}}
\label{tab:combined_ke_results}
\end{table}

\smallskip 

\paragraph{Comparison of Vanilla Regret.} We next compare the vanilla regret of all algorithms in \Cref{tab:vanilla-rel-vs-lam-ke}. The results are intuitive, given the above discussion. (We do not discuss results for $\lambda > \qopt_G$ below, given that they are identical to those reported in \Cref{tab:combined_ke_results}.) 

Fixing $N = 500$, all algorithms output $\lambda$ in the majority of replications when $\lambda < \qopt_G$ (except in the case of the Poisson distribution, for $\lambda = 87.4$). Since $\lambda$ is close to $\qopt_G$ in these settings, all algorithms exhibit low vanilla regret; this performance only improves as $\lambda$ increases and approaches $\qopt_G$. The exception to this is $\lambda = 87.4$ under the Poisson distribution where, as discussed above, \ALG correctly classifies the problem as unidentifiable in the majority of replications. For this specific instance, $\qcrit_G \approx 203$; hence, $|\qcrit_G-\qopt_G| \gg |\qopt_G-\lambda|$, which explains why the vanilla regret of \ALG is higher than that of \KM and \CensoredSAA.

Finally, the vanilla regret of \ALG weakly decreases as $N$ increases (again, except for the isolated instance under the Poisson distribution). This result is slightly unintuitive, since \ALG outputs an estimate of $\qcrit_G$ {\it more} frequently as $N$ increases. However, for these instances, $\qcrit_G$ is actually {\it closer} to $\qopt_G$ than $\lambda$ is; hence, worst-case and vanilla regret of \ALG are directionally aligned. This is not the case for $\lambda = 87.4$ under the Poisson distribution, where $\qcrit_G \gg \qopt_G$, as discussed above.

\begin{table}[t]
\centering
\small
\minedit{
\setlength\tabcolsep{3pt}
\begin{subtable}[b]{\linewidth}
\begin{tabular}{ll|rr!{\vrule width 1.5pt}rr}
\toprule
Algorithm & $N$ & 84.55 & 87.51 & 90.48 & 93.45 \\
\midrule

\multirow{3}{*}{\ALG}
& 100  & 1.2 (2.7\%) & 0.2 (0.4\%) & 0.0 (0.1\%) & 0.8 (1.7\%) \\
& 250  & 1.2 (2.6\%) & 0.2 (0.4\%) & 0.0 (0.1\%) & 0.8 (1.7\%) \\
& 500  & 0.9 (1.9\%) & 0.2 (0.4\%) & 0.0 (0.1\%) & 0.7 (1.6\%) \\

\midrule

\multirow{3}{*}{\KM}
& 100  & 1.3 (2.9\%) & 0.4 (0.8\%) & 0.3 (0.7\%) & 0.5 (1.0\%)  \\
& 250  & 1.2 (2.7\%) & 0.2 (0.5\%) & 0.1 (0.2\%) & 0.2 (0.3\%)  \\
& 500  & 1.2 (2.7\%) & 0.2 (0.5\%) & 0.0 (0.1\%) & 0.1 (0.1\%)  \\

\midrule

\multirow{3}{*}{\CensoredSAA}
& 100  & 1.3 (3.0\%) & 0.4 (1.0\%) & 0.5 (1.1\%) & 0.6 (1.2\%)  \\
& 250  & 1.2 (2.7\%) & 0.3 (0.6\%) & 0.1 (0.2\%) & 0.2 (0.4\%)  \\
& 500  & 1.2 (2.7\%) & 0.2 (0.5\%) & 0.1 (0.1\%) & 0.1 (0.1\%)  \\

\bottomrule
\end{tabular}
\centering
\caption{Uniform, $\qopt_G = 89$}
\label{tab:uniform_vanilla_ke}
\end{subtable}

\begin{subtable}[b]{\linewidth}
\begin{tabular}{ll|rr!{\vrule width 1.5pt}rr}
\toprule
Algorithm & $N$
 & 174.96 & 181.10 & 187.24 & 193.38  \\
\midrule

\multirow{3}{*}{\ALG}
& 100  & 0.7 (0.4\%) & 0.2 (0.1\%) & 0.0 (0.0\%) & 0.4 (0.2\%)  \\
& 250  & 0.4 (0.2\%) & 0.0 (0.0\%) & 0.0 (0.0\%) & 0.4 (0.2\%)  \\
& 500  & 0.4 (0.2\%) & 0.0 (0.0\%) & 0.0 (0.0\%) & 0.4 (0.2\%)  \\

\midrule

\multirow{3}{*}{\KM}
& 100  & 1.5 (0.8\%) & 1.5 (0.8\%) & 1.2 (0.7\%) & 2.1 (1.1\%)  \\
& 250  & 0.7 (0.4\%) & 0.4 (0.2\%) & 0.5 (0.3\%) & 0.6 (0.4\%)  \\
& 500  & 0.6 (0.3\%) & 0.2 (0.1\%) & 0.2 (0.1\%) & 0.3 (0.2\%)  \\

\midrule

\multirow{3}{*}{\CensoredSAA}
& 100  & 2.1 (1.1\%) & 2.7 (1.5\%) & 1.9 (1.0\%) & 2.8 (1.5\%)  \\
& 250  & 0.9 (0.5\%) & 0.6 (0.3\%) & 0.8 (0.4\%) & 0.8 (0.4\%)  \\
& 500  & 0.6 (0.3\%) & 0.3 (0.2\%) & 0.3 (0.2\%) & 0.4 (0.2\%)  \\

\bottomrule
\end{tabular}
\centering
\caption{Exponential, $\qopt_G = 184.14$}
\label{tab:exponential_vanilla_ke}
\end{subtable}

\begin{subtable}[b]{\linewidth}
\begin{tabular}{ll|rr!{\vrule width 1.5pt}rr}
\toprule
Algorithm & $N$ & 87.40 & 90.46 & 93.53 & 96.60  \\
\midrule

\multirow{3}{*}{\ALG}
& 100  & 26 (164\%) & 1.4 (8.9\%) & 0.3 (2.1\%) & 1.9 (12\%)  \\
& 250  & 93 (576\%) & 0.1 (0.8\%) & 0.3 (2.1\%) & 0.7 (4.5\%)  \\
& 500  & 108 (675\%) & 1.9 (12\%) & 0.3 (2.0\%) & 0.0 (0.3\%)  \\

\midrule

\multirow{3}{*}{\KM}
& 100  & 2.0 (12\%) & 0.2 (1.3\%) & 0.2 (1.1\%) & 0.3 (1.7\%)  \\
& 250  & 2.0 (12\%) & 0.2 (0.9\%) & 0.1 (0.6\%) & 0.1 (0.5\%)  \\
& 500  & 2.0 (12\%) & 0.1 (0.8\%) & 0.0 (0.3\%) & 0.0 (0.3\%)  \\

\midrule

\multirow{3}{*}{\CensoredSAA}
& 100  & 2.0 (12\%) & 0.2 (1.4\%) & 0.2 (1.1\%) & 0.3 (1.8\%)  \\
& 250  & 2.0 (12\%) & 0.2 (0.9\%) & 0.1 (0.6\%) & 0.1 (0.5\%)  \\
& 500  & 2.0 (12\%) & 0.1 (0.8\%) & 0.1 (0.3\%) & 0.0 (0.3\%)  \\

\bottomrule
\end{tabular}
\centering
\caption{Poisson, $\qopt_G = 92$}
\label{tab:poisson_vanilla_ke}
\end{subtable}
}
\caption{\minedit{Impact of $\lambda$ and $N$ on policy performance for vanilla regret in the knife-edge regime. Here, $\rho = 0.9$. Across all values of $\lambda$ we report $\UncensoredRegret(q^\pi)$ and $\mathcal{R}^{id}(q^\pi)\%$.}}
\label{tab:vanilla-rel-vs-lam-ke}
\end{table}
}

\clearpage

\subsubsection{\minedit{Robustifying the Kaplan-Meier Estimator}}
\label{app:sims_robust_km}

\begin{table}[!ht]
\minedit{
\setlength\tabcolsep{3pt}
\small
\begin{subtable}[b]{\linewidth}
\centering
\begin{tabular}{lrrrr!{\vrule width 1.5pt}rrrr}
\toprule
$\lambda$ & 44.50 & 57.21 & 69.93 & 82.64 & 95.36 & 108.07 & 120.79 & 133.50 \\
\midrule
\ALG & 4.0 (1.7\%) & 6.7 (3.3\%) & 7.7 (4.5\%) & 27 (27\%) & 0.1 (0.2\%) & 0.1 (0.2\%) & 0.1 (0.1\%) & 0.1 (0.2\%) \\
\KM & 1033 (450\%) & 653 (320\%) & 341 (200\%) & 70 (70\%) & 0.1 (0.1\%) & 0.1 (0.1\%) & 0.0 (0.1\%) & 0.1 (0.1\%) \\
\RKM & 4.0 (1.7\%) & 6.7 (3.3\%) & 7.7 (4.5\%) & 27 (27\%) & 0.1 (0.2\%) & 0.1 (0.1\%) & 0.0 (0.1\%) & 0.1 (0.1\%)
\\
\bottomrule
\end{tabular}
\caption{Uniform, $\qopt_G = 89$}
\label{tab:rkm_uniform}       
\end{subtable}
\medskip
\begin{subtable}[b]{\linewidth}
\centering
\begin{tabular}{lrrrr!{\vrule width 1.5pt}rrrr}
\toprule
$\lambda$ & 92.07 & 118.38 & 144.68 & 170.99 & 197.30 & 223.60 & 249.91 & 276.22 \\
\midrule
\ALG & 6.7 (4.2\%) & 6.9 (6.0\%) & 21 (29\%) & 4.3 (19\%) & 1.0 (0.5\%) & 7.4 (4.0\%) & 3.9 (2.1\%) & 0.8 (0.5\%) \\
\KM & 345 (216\%) & 148 (128\%) & 45 (64\%) & 4.5 (19\%) & 0.6 (0.3\%) & 0.6 (0.3\%) & 0.5 (0.3\%) & 0.5 (0.3\%) \\
\RKM & 6.7 (4.2\%) & 6.9 (6.0\%) & 21 (29\%) & 4.3 (19\%) & 1.0 (0.5\%) & 6.8 (3.7\%) & 2.9 (1.6\%) & 0.5 (0.5\%) \\
\bottomrule
\end{tabular}
\caption{Exponential, $\qopt_G = 184.21$}
\label{tab:tkm_exponential}       
\end{subtable}
\medskip
\begin{subtable}[b]{\linewidth}
\centering
\begin{tabular}{lrrrr!{\vrule width 1.5pt}rrrr}
\toprule
$\lambda$ & 46 & 59.14 & 72.29 & 85.43 & 98.57 & 111.71 & 124.86 & 138 \\
\midrule
\ALG & 0.0 (0.0\%) & 0.5 (0.2\%) & 2.4 (1.1\%) & 7.1 (4.7\%) & 0.1 (0.3\%) & 0.0 (0.3\%) & 0.1 (0.3\%) & 0.1 (0.4\%) \\
\KM & 2260 (900\%) & 2131 (891\%) & 1542 (697\%) & 247 (165\%) & 0.1 (0.3\%) & 0.0 (0.3\%) & 0.1 (0.3\%) & 0.1 (0.4\%) \\
\RKM & 0.0 (0.0\%) & 0.5 (0.2\%) & 2.4 (1.1\%) & 7.1 (4.7\%) & 0.1 (0.3\%) & 0.0 (0.3\%) & 0.1 (0.3\%) & 0.1 (0.4\%) \\
\bottomrule
\end{tabular}
\caption{Poisson, $\qopt_G = 92$}
\label{tab:rkm_poisson}       
\end{subtable}
\centering
\caption{
\minedit{Comparison of \ALG, \KM, and \RKM algorithms for $\rho = 0.9$.  Values to the left of the thick vertical line correspond to the unidentifiable regime, where we report $\Regret(q^\pi) - \risk$ and $\mathcal{R}^{ui}(q^\pi)\%$; values to the right of the thick vertical line correspond to the identifiable regime, where we report $\UncensoredRegret(q^\pi)$ and $\mathcal{R}^{id}(q^\pi)\%$.}}
\label{tab:rkm_comparison}
}
\end{table}

\minedit{In this section we demonstrate the flexibility of our robust framework by implementing \RKM, a ``robust'' variant of the \KM algorithm.  \RKM is identical to \ALG in all regimes except the ``likely identifiable'' regime, where it outputs the $\rho$-quantile of the KM estimator, instead of the SAA estimate used by \ALG. 

We compare the performance of \KM, \RKM, and \ALG in \Cref{tab:rkm_comparison} for various values of $\lambda$, as described in \Cref{sec:experiment_setup}, and letting $\rho = 0.9$ and $N = 500$. 
For the five smallest values of $\lambda$, \RKM and \ALG coincide across all demand families, far outperforming \KM. This behavior is expected, since both algorithms output the same classification (either ``knife-edge'' or ``likely unidentifiable'') in these regimes. Moreover, conditional on an instance being classified in either of these two ways, the algorithms output the same quantity. For larger values of $\lambda$, we observe that \RKM slightly outperforms \ALG, with the most pronounced improvement in the exponential case, for $\lambda = 223.60$ and $\lambda = 249.91$. The reason for this is that, for these values of $\lambda$, while the classification of the two algorithms may be the same, when the instance is classified as ``likely identifiable,'' \RKM outputs the $\rho$-quantile of the KM estimator, which slightly outperforms the censored SAA estimate in the identifiable regime. Finally, we remark that, in the case of the exponential distribution, \KM outperforms \RKM throughout the identifiable regime. This is the cost of robustness: in this regime, \RKM makes classification errors in which it outputs $\lambda$ instead of the KM estimator.

Overall, these results highlight the flexibility of our framework as a ``robustification'' tool to adapt optimal algorithms for the uncensored setting to the censored setting.
}

\section{On The Value of Information: The Case of Globally Well-Separated Distributions}
\label{app:well_separated}

In this section we study how to leverage additional distributional information to improve decision-making in the face of censored data. We do so by considering the setting where the demand distribution is {\it globally well-separated}, and the decision-maker knows a lower bound $\gamma > 0$ on its pdf. For this setting, our goal is to quantify (i) the {\em value of information}, measured as the reduction in minimax risk when additional well-separatedness information is available, and (ii) the resulting {\em sample complexity improvements} of a modified version of \ALG that explicitly exploits this information to detect whether the problem instance is identifiable or not. In the remainder of this section, we let $\riskws$ and $\qriskws$ respectively denote the minimax risk and minimax optimal ordering quantity under the well-separatedness condition.

\subsection{Characterizing the Minimax Risk}

We first characterize $\riskws$ and $\qriskws$ in this setting. Under this additional distributional information, the ambiguity set is given by:
\begin{equation}
\label{eq:well_separated_ambig}
\ambig{\lambda} =\left\{F \in \F \mid F(x) = G(x) \ \forall \ x < \lambda, \,\, f(x) \geq \gamma \text{ on its support}, \qopt_F \leq M \right\}.
\end{equation}

The following lemma will be useful in our derivation of $\qriskws$. We defer its proof to Appendix \ref{apx:ws-useful-facts}.
\begin{lemma}\label{lem:ws-useful-facts}
For all $F\in\ambig{\lambda}$, $F$ is a distribution with bounded support, with $f(x) = 0$ for all $x > \frac1\gamma$. Moreover, \minedit{if $\Gminus(\lambda) \leq \rho$,}  $\qopt_F \leq \lambda + \frac{\rho - \Gminus(\lambda)}{\gamma}$.
\end{lemma}

Noting that $\Gminus(\lambda) \geq \lambda\gamma \implies \lambda + \frac{\rho-\Gminus(\lambda)}{\gamma} \leq \frac\rho\gamma \leq \frac1\gamma$, we let $\Msep = \min\left\{M,\lambda + \frac{\rho - \Gminus(\lambda)}{\gamma}\right\}$ and modify the definition of $\ambig{\lambda}$ in \eqref{eq:well_separated_ambig} to be such that $\qopt_F \leq \Msep$ \minedit{when $\Gminus(\lambda) \leq \rho$}. 

\begin{theorem}
\label{thm:risk_well_separated}
Consider a data-driven censored newsvendor problem with well-separated demand distribution $G$ and observable boundary $\lambda$. Then, given $\gamma > 0$:
\begin{enumerate}
\item If $\Gminus(\lambda) \geq \rho$, $\qriskws = \qopt_G < \lambda$, and $\riskws = 0$.
\item If $\Gminus(\lambda) < \rho$, 
\[\qriskws =  \lambda + \frac{1-\Gminus(\lambda)-\sqrt{(1-\Gminus(\lambda))^2 + (\rho-\Gminus(\lambda)-\gamma(\Msep-\lambda))^2-(\rho-\Gminus(\lambda))^2}}{\gamma} \geq \lambda.\]
Moreover, \[\riskws = (b+h)(1-\rho)\cdot\frac{1-\Gminus(\lambda)-\sqrt{(1-\Gminus(\lambda))^2 + (\rho-\Gminus(\lambda)-\gamma(\Msep-\lambda))^2-(\rho-\Gminus(\lambda))^2}}{\gamma}.\]
\end{enumerate}
\end{theorem}

To gain further insight into the value of distributional information in this setting, in \Cref{fig:del-diff} we plot $\risk-\riskws$ as a function of $\gamma$, when $D \sim \text{Unif}[0,1]$ and {$b = 9, h = 1$ so $\rho = 0.9$}.  We observe that $\risk-\riskws$ is increasing in $\gamma$, for fixed $\lambda$. This phenomenon is intuitive, since the size of the ambiguity set decreases as $\gamma$ increases; this supports the interpretation of $\gamma$ as the quality of the decision-maker's additional information. We additionally observe that, for fixed $\gamma$, $\risk-\riskws$ is concave decreasing in $\lambda$. This shows that the decision-maker has the most to gain from additional distributional information when the dataset is highly censored and therefore relatively uninformative. 

\Cref{fig:q-diff} further illustrates that $\gamma$ not only affects the worst-case regret of the decision-maker, but also her minimax optimal ordering quantity. In particular, we observe that $\qriskws$ is {\it lower} than $\qrisk$ for all values of $\lambda$ and $\gamma$. Moreover, for fixed $\lambda$, the difference between $\qrisk$ and $\qriskws$ is increasing in $\gamma$. The reason for this is also intuitive: as $\gamma$ increases, the decision-maker is guaranteed that a larger amount of mass resides close to $\lambda$, thereby reducing the likelihood of large realizations of $D$ and lowering the minimax optimal ordering quantity as a result. 
\begin{figure}[t]
    \centering
        \begin{subfigure}[b]{0.48\textwidth}
        \centering
        \includegraphics[width=\textwidth]{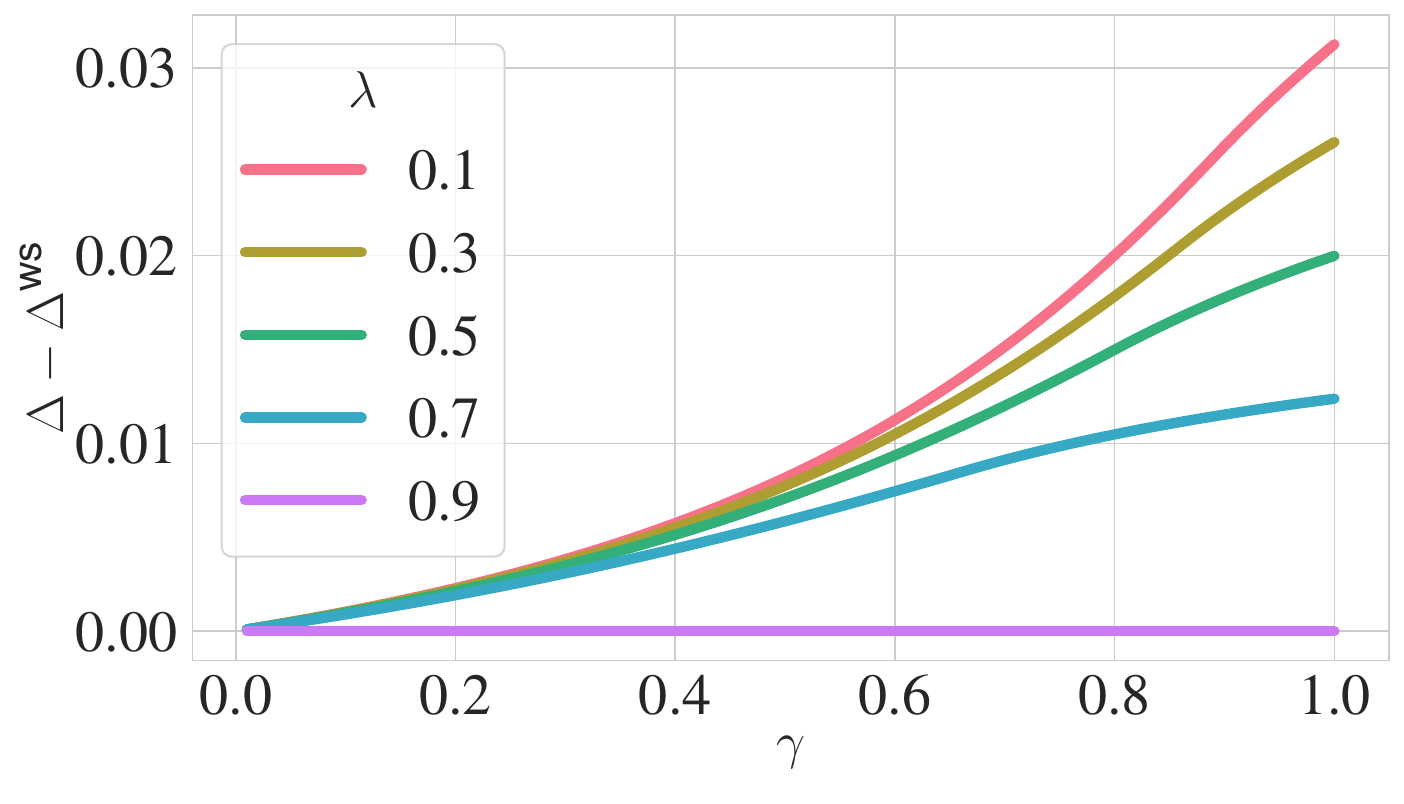}
        \caption{$\risk-\riskws$ vs. $\gamma$}
        \label{fig:del-diff}
    \end{subfigure}
    \hfill
    \begin{subfigure}[b]{0.48\textwidth}
        \centering
        \includegraphics[width=\textwidth]{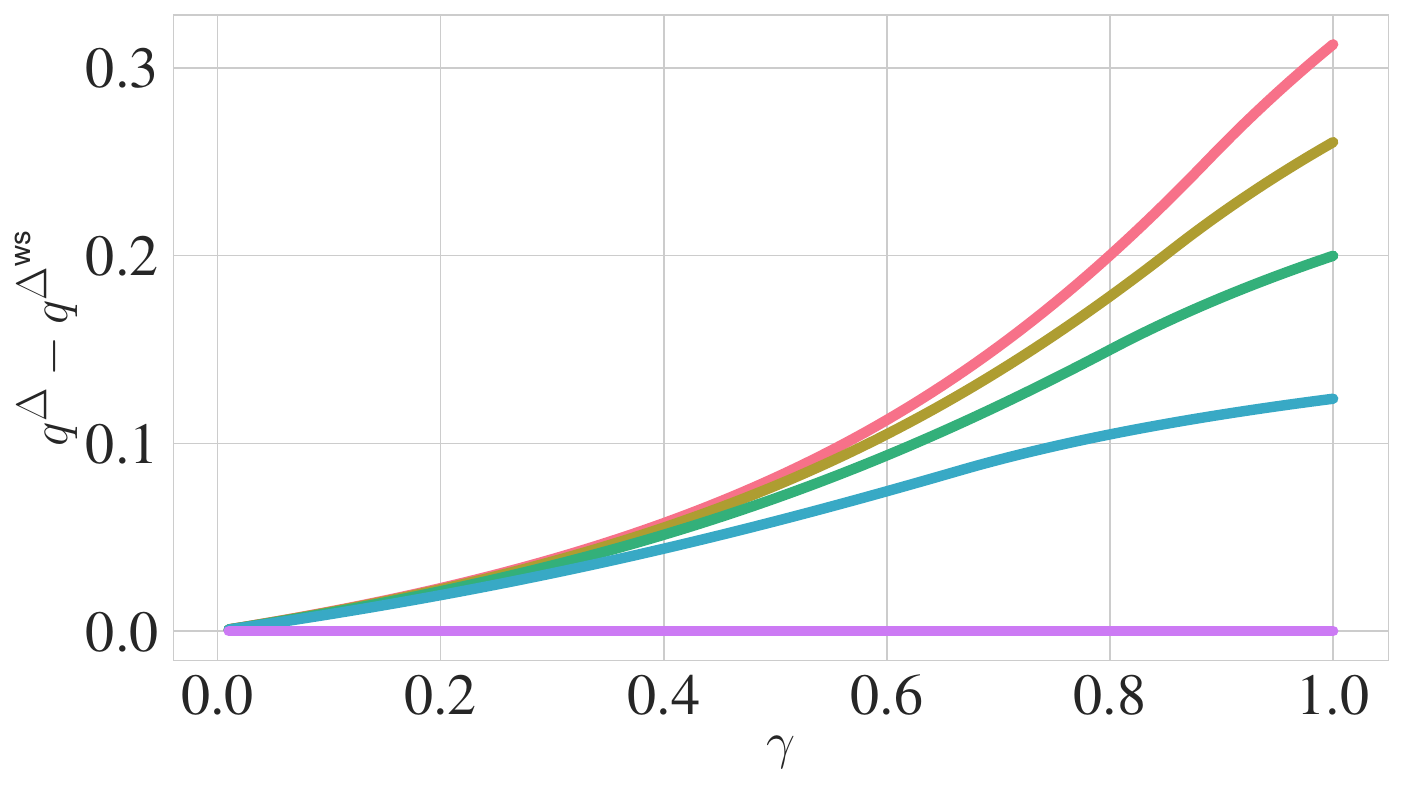}
        \caption{$\qrisk-\qriskws$ vs. $\gamma$}
        \label{fig:q-diff}
    \end{subfigure}
    \caption{Dependence of $\qrisk-\qriskws$ and $\risk-\riskws$ on $\gamma$, for $D\sim\text{Unif}[0,1]$. Here, we assume $M = 1$, $b = 9$ and $h = 1$.}
    \label{fig:well-sep}
\end{figure}

\begin{proof}{\it Proof of \Cref{thm:risk_well_separated}.}
The proof of the first fact is identical to the more general setting we consider in \Cref{thm:minimax-risk-identifiable}; we omit it as such. In the remainder of the proof we focus on the case where $\Gminus(\lambda) < \rho$. The proof proceeds in the same fashion as that of \Cref{thm:minimax-risk-identifiable}. Namely, we first provide closed-form expressions for the worst-case regret over $\ambig{\lambda}$, given $q \leq \Msep$.

\begin{lemma}\label{lem:reg-well-sep}
\begin{align*}
\Regret(q) = \begin{cases}
(b+h)\left(\rho(\Msep-q) + \E_G\left[(q-D)\Ind{D \leq q}\right] - \E_G\left[(\Msep-D)\Ind{D < \lambda}\right] - \frac\gamma2(\Msep-\lambda)^2\right) \quad \text{if } q < \lambda \\
(b+h)\max\left\{(1-\rho)(q-\lambda), (\Msep-q)\rho-\Gminus(\lambda)-\gamma(q-\lambda))-\frac\gamma2(\Msep-q)^2\right\}  \quad \text{if } q \geq \lambda. 
\end{cases}
\end{align*}
\end{lemma}

We defer the proof of the lemma to Appendix \ref{apx:reg-well-sep}. We use these expressions to compute the minimax optimal ordering quantity $\qcrit_G$ in this case. For all $q < \lambda$, we have:
\begin{align*}
\frac{d}{dq}\left[\rho(\Msep-q) + \E_G\left[(q-D)\Ind{D \leq q}\right] - \E_G\left[(\Msep-D)\Ind{D < \lambda}\right] - \frac\gamma2(\Msep-\lambda)^2\right] = -\rho + G(q) < 0, 
\end{align*}
since $\Gminus(\lambda) < \rho$. Therefore:
\begin{align*}
\inf_{q < \lambda}\Regret(q) &= (b+h)\left(\rho(\Msep-\lambda) + \E_G\left[(\lambda-D)\Ind{D \leq \lambda}\right] - \E_G\left[(\Msep-D)\Ind{D < \lambda}\right] - \frac\gamma2(\Msep-\lambda)^2\right)\\
&= (b+h)\left(\rho(\Msep-\lambda) - (\Msep-\lambda)\Gminus(\lambda) - \frac\gamma2(\Msep-\lambda)^2\right)\\
&= (b+h)\left((\Msep-\lambda)(\rho-\Gminus(\lambda)) - \frac\gamma2(\Msep-\lambda)^2\right) = \Regret(\lambda).
\end{align*}
Thus we have, 
$\qcrit_G = \arg \min_{q \geq \lambda}\max\left\{(1-\rho)(q-\lambda), (\Msep-q)(\rho-\Gminus(\lambda)-\gamma(q-\lambda))-\frac\gamma2(\Msep-q)^2\right\}$. Equivalently, $\qcrit_G$ is the smallest solution to:
\begin{align*}
&(1-\rho)(q-\lambda) = (\Msep-q)(\rho-\Gminus(\lambda)-\gamma(q-\lambda))-\frac\gamma2(\Msep-q)^2\\ \iff &\qcrit_G = \lambda + \frac{1-\Gminus(\lambda)-\sqrt{(1-\Gminus(\lambda))^2 + (\rho-\Gminus(\lambda)-\gamma(\Msep-\lambda))^2-(\rho-\Gminus(\lambda))^2}}{\gamma} \geq \lambda,
\end{align*}
where the final solution simply follows from algebra; we omit it as such.
\hfill\Halmos
\end{proof}

\subsubsection{Proof of \Cref{lem:ws-useful-facts}}\label{apx:ws-useful-facts}

\begin{proof}{\it Proof.}
To see that $f(x) = 0$ for all $x > \frac1\gamma$, note that $F(x) \geq \gamma x$ for all $x$ in the support of $F$. Therefore, $F(x) = 1$ for some $x \leq \frac1\gamma$, which proves the claim.

For the lower bound on $\qopt_F$, \minedit{since $\qopt_F \geq \lambda$ whenever $\Gminus(\lambda) \leq \rho$, we have:}
\[
\qopt_F = \inf \{q > 0 : F(q) \geq \rho \} \leq \inf \{ q > 0 : \Gminus(\lambda) + \gamma(q - \lambda) \geq \rho\}.
\]
Rearranging the above inequality yields the result.

\hfill\Halmos
\end{proof}

\subsubsection{Proof of \Cref{lem:reg-well-sep}}\label{apx:reg-well-sep}

\begin{proof}{\it Proof.}
As in the proof of \Cref{lem:sup_expression_identifiability}  {(see~\cref{eq:worst-case-rewrite})}, it suffices to analyze:
\begin{align}\label{eq:come-back-to-this}
&\sup_{F \in \ambig{\lambda}} C_F(q) - C_F(\qopt_F) \notag \\
&=\sup_{\tilde{q}\in[\lambda,\Msep]}b(\tilde{q} - q) + (b+h)\left(\sup_{\substack{F\in \ambig{\lambda}:\\\qopt_F=\tilde{q}}} \E_F\left[(q-D)\Ind{D \leq q} - (\tilde{q}-D)\Ind{D \leq \tilde{q}}\right]\right) \notag \\
&= (b+h)\left(\sup_{\tilde{q}\in[\lambda,\Msep]}\rho(\tilde{q} - q) + \left(\sup_{\substack{F\in \ambig{\lambda}:\\\qopt_F=\tilde{q}}} \E_F\left[(q-D)\Ind{D \leq q} - (\tilde{q}-D)\Ind{D \leq \tilde{q}}\right]\right) \right)
\end{align}

Note that, for $\tilde{q} = q$, $\eqref{eq:come-back-to-this} = 0$. Therefore, we partition our analysis into two cases: $\tilde{q} < q$ and $\tilde{q} > q$.

\begin{enumerate}[(i)]
\item $q > \tilde{q} \geq \lambda$: In this case, by the proof of \Cref{lem:sup_expression_identifiability} {(\cref{eq:worst-case-1})}, the worst-case distribution $F$ that has $\tilde{q}$ as its newsvendor solution maximizes:
\begin{align}\label{eq:sep-1}
\E_F\left[(q-\tilde{q})\Ind{D \leq \tilde{q}} + ({q}-D)\Ind{\tilde{q} <D \leq q}\right]. 
\end{align}
Observe that $q-D < q-\tilde{q}$ for all $D \in (\tilde{q},q]$. Hence, any worst-case distribution $F$ should minimize $\Pr_F(D \in (\tilde{q},q])$, which occurs by letting $\tilde{q}$ be the maximum of the support of the distribution. Therefore:
\begin{align*}
\eqref{eq:sep-1} &= (q-\tilde{q})\\
\implies \Regret(q) &= (b+h)\sup_{\tilde{q} \in [\lambda, \Msep]}(1-\rho)(q-\tilde{q}) = (b+h)(1-\rho)(q-\lambda).
\end{align*}

\item $q < \lambda \leq \tilde{q}$. As in the proof of \Cref{lem:sup_expression_identifiability}  {(\cref{eq:worst_case_other_setting})}, for a fixed $\tilde{q}$, it suffices to find the distribution that maximizes:
\begin{align}\label{eq:sep-2}
\sup_{\substack{F\in\ambig{\lambda}:\\\qopt_F=\tilde{q}}} \E_G\left[(q-D)\Ind{D \leq q}\right] - \E_G\left[(\tilde{q}-D)\Ind{D < \lambda}\right] - \E_F\left[(\tilde{q}-D)\Ind{\lambda \leq D < \tilde{q}}\right].
\end{align}
Suppose first that $\tilde{q} > \lambda$.
Since $\tilde{q}-D > 0$ for all $D \in [\lambda,\tilde{q})$, a similar argument as above establishes that any worst-case distribution $F \in \ambig{\lambda}$ necessarily sets \minreplace{$\Pr_F(D \in [\lambda,\tilde{q})) = \gamma$}{$f(x) = \gamma$ for all $x \in [\lambda,\tilde{q})$}. Using this fact above, we obtain:
\begin{align*}
\eqref{eq:sep-2} &= \E_G\left[(q-D)\Ind{D \leq q}\right] - \E_G\left[(\tilde{q}-D)\Ind{D < \lambda}\right] - \gamma\int_{\lambda}^{\tilde{q}}(\tilde{q}-x)dx \\
&= \E_G\left[(q-D)\Ind{D \leq q}\right] - \E_G\left[(\tilde{q}-D)\Ind{D < \lambda}\right] - \frac\gamma2(\tilde{q}-\lambda)^2.
\end{align*}
Plugging this back into \eqref{eq:come-back-to-this}, we have:
\begin{align*}
\eqref{eq:come-back-to-this} = (b+h)\left(\sup_{\tilde{q}\in[\lambda,\Msep]} \rho(\tilde{q}-q) + \E_G\left[(q-D)\Ind{D \leq q}\right] - \E_G\left[(\tilde{q}-D)\Ind{D < \lambda}\right] - \frac\gamma2(\tilde{q}-\lambda)^2\right).
\end{align*}
By the first-order condition, the unconstrained supremum of the above is achieved at the smallest value of $\tilde{q}$ such that:
\begin{align*}
&\rho-\Gminus(\lambda)-\gamma(\tilde{q}-\lambda) \leq 0
\iff \tilde{q} \geq \lambda + \frac{\rho-\Gminus(\lambda)}{\gamma}.
\end{align*}
Since $\tilde{q} \leq \Msep \leq \frac{\rho-\Gminus(\lambda)}{\gamma}$, the constrained supremum is attained at $\Msep$, and
\begin{align}\label{eq:q-greater-lam}
\Regret(q) = (b+h)\left(\rho(\Msep-q) + \E_G\left[(q-D)\Ind{D \leq q}\right] - \E_G\left[(\Msep-D)\Ind{D < \lambda}\right] - \frac\gamma2(\Msep-\lambda)^2\right).
\end{align}

Now, if $\tilde{q} = \lambda$:
\begin{align}\label{eq:q-equal-lam}
\Regret(q) = (b+h)\Big(\rho(\lambda-q) + \E_G\left[(q-D)\Ind{D \leq q}\right] - \E_G\left[(\lambda-D)\Ind{D < \lambda}\right]\Big).
\end{align}

Comparing \eqref{eq:q-greater-lam} and \eqref{eq:q-equal-lam}, setting $\tilde{q} > \lambda$ is optimal if and only if:
\begin{align*}
(\rho-\Gminus(\lambda))(\Msep-\lambda) - \frac\gamma2(\Msep-\lambda)^2 \geq 0 \iff \rho-\Gminus(\lambda)-\frac\gamma2(\Msep-\lambda)\geq 0.
\end{align*}
The above holds, since by definition $\Msep \leq \lambda+\frac{\rho-\Gminus(\lambda)}{\gamma}$, and $\rho-\Gminus(\lambda) > 0$. 

\item $\lambda \leq q < \tilde{q}$: In this case, by the proof of \Cref{lem:sup_expression_identifiability}  {(\cref{eq:last_case_for_worst_case_distribution})}, it suffices to find a distribution $F$ that maximizes:
\begin{align}\label{eq:sep-3}
-(\tilde{q}-q)\Gminus(\lambda)-(\tilde{q}-q)\Pr_F(\lambda\leq D \leq q)-\E_F\bigg[(\tilde{q}-D)\Ind{{q} <D \leq \tilde{q}}\bigg]. 
\end{align}
Since $0 < \tilde{q}-D < \tilde{q}-q$ for all $D \in (q,\tilde{q})$, the supremum of the above is achieved for $F$ such that \mbox{$f(x) = \gamma$ for all $x \in [\lambda, \tilde{q})$}. In this case, we have:
\begin{align*}
\eqref{eq:sep-3} &= -(\tilde{q}-q)\Gminus(\lambda)-\gamma(\tilde{q}-q)(q-\lambda)-\gamma\int_{q}^{\tilde{q}}(\tilde{q}-x)dx \\
&= -(\tilde{q}-q)\left(\Gminus(\lambda)+\gamma(q-\lambda)\right)-\frac\gamma2(\tilde{q}-q)^2.
\end{align*}
Plugging this back into \eqref{eq:come-back-to-this}, we have:
\begin{align*}
\eqref{eq:come-back-to-this} = (b+h)\left(\sup_{\tilde{q}\in[\lambda,\Msep]}\rho(\tilde{q}-q) -(\tilde{q}-q)\left(\Gminus(\lambda)+\gamma(q-\lambda)\right)-\frac\gamma2(\tilde{q}-q)^2\right).
\end{align*}
By the first-order condition, the unconstrained supremum of the above is attained at the smallest $\tilde{q}$ such that:
\begin{align*}
\rho-(\Gminus(\lambda)+\gamma(q-\lambda))-\gamma(\tilde{q}-q) \leq 0
\iff \tilde{q} \geq q + \frac{\rho-(\Gminus(\lambda)+\gamma(q-\lambda))}{\gamma} = \lambda + \frac{\rho-\Gminus(\lambda)}{\gamma}.
\end{align*}
Again, using the fact that $\tilde{q} \leq \Msep \leq \lambda+\frac{\rho-\Gminus(\lambda)}{\gamma}$, we obtain that the constrained supremum is attained at $\tilde{q} = \Msep$, and
\[\Regret(q) = (b+h) \left((\Msep-q)(\rho-(\Gminus(\lambda)+\gamma(q-\lambda))-\frac\gamma2(\Msep-q)^2\right).\]
\end{enumerate}
Putting these three cases together, we obtain the result.
\hfill\Halmos
\end{proof}

\subsection{Practical Performance Improvements in the Knife-Edge Regime}
\label{app:sims_globally_well_separated}

We next numerically investigate how the well-separated condition can generate sample complexity gains in the knife-edge regime, where testing for identifiability is the hardest. To do so, we introduce a new algorithm, which we denote \ALGWS. Similar to \ALG, \ALGWS proceeds hierarchically, by first testing the identifiability regime, and then estimating the appropriate $\qriskws$. However, its identifiability test leverages the well-separatedness condition via the following observations:
\begin{enumerate}
    \item If $\gamma \lambda \geq \rho$, by the well-separatedness condition, $\Gminus(\lambda) \geq \lambda \gamma \geq \rho,$
    and the problem is guaranteed to be identifiable. In this case, no statistical test is needed.
    \item If $\Gminushat(\lambda) > \rho$, $\qopt_{\offeCDF} < \lambda$. Then, by the well-separatedness condition, with high probability:
\[
\Gminus(\lambda) \geq \Gminus(\qopt_{\offeCDF}) + \gamma(\lambda - \qopt_{\offeCDF}) \geq \Gminushatk{K}(\qopt_{\offeCDF}) - \width_K + \gamma(\lambda - \qopt_{\offeCDF}) \geq \rho - \width_K + \gamma(\lambda - \qopt_{\offeCDF}).
\]
Thus, if $\gamma(\lambda - \qopt_{\offeCDF}) \geq \width_K$, $\Gminus(\lambda) \geq \rho$, and the problem is identifiable.
    \item If $\Gminushatk{k}(\qoff_k) \geq \rho + \width_k - \gamma(\lambda - \qoff_k)$ for some $k \in [K]$, we similarly have:
    \[
    \Gminus(\lambda) \geq \Gminus(\qoff_k) + \gamma(\lambda - \qoff_k) \geq \Gminushatk{k}(\qoff_k) - \width_k + \gamma(\lambda - \qoff_k) \geq \rho,
    \]
    which again guarantees identifiability.
\end{enumerate} 
We use these three observations in our modified algorithm, presented in \Cref{alg:newsvendor_ws}.

\begin{algorithm}[!t]
\DontPrintSemicolon 
\KwIn{Censored demand samples $\soff$,  confidence terms $\width_k, k \in [K]$, well-separated parameter $\gamma$}
\KwOut{Ordering quantity $\qalg$}

For all $k \in [K]$, compute censored SAA of $\Gminus(\qoff_k)$: $\Gminushatk{k}{(\qoff_k)} = \frac{1}{N_k} \sum_{i \in [N_k]} \Ind{\soff_{ki} < \qoff_k}.$

Define $\Gminushat(\lambda) = \Gminushatk{K}(\qoff_K)$.

Compute censored SAA of $\qopt_G$: $\qopt_{\offeCDF} = \inf \left\{x \mid \offeCDF(x) \geq \rho \right\},$
where  $\offeCDF(x) = \frac{1}{N_K}\sum_{i \in [N_K]} \Ind{\soff_{iK} \leq x}.$

\uIf(\tcp*[h]{Likely identifiable}){$\gamma\lambda \geq \rho$ \emph{or} $\Gminushat(\lambda) \geq \rho+\width_K$ \emph{or} $\Gminushatk{k}(\qoff_k) \geq \rho + \width_k - \gamma(\lambda - \qoff_k)$ \emph{for some} $k \in [K]$}{
Let $\qalg = \qopt_{\offeCDF}$.
}
\uElseIf(\tcp*[h]{Borderline: use separation test}){$\Gminushat(\lambda) > \rho$}{
\uIf(\tcp*[h]{Separation overcomes noise}){$\gamma(\lambda - \qopt_{\offeCDF}) \geq \width_K$}{
    Let $\qalg = \qopt_{\offeCDF}$. \tcp*[h]{Likely identifiable}
    }
    \uElse{
        Let $\qalg = \lambda$. \tcp*[h]{Knife-edge}
    }
}
\uElseIf(\tcp*[h]{Knife-edge}){$\Gminushat(\lambda) \in (\rho-\zeta,\;\rho]$}{
    Let $\qalg = \lambda$.
}

\uElse(\tcp*[h]{Likely {unidentifiable}}){
    Compute empirical estimate of $\qcrit_G$:
    \begin{align}
    \label{eq:qcrithat_ws}
    \qcrithat \;=\; \frac{\Msep(\rho - \Gminushat(\lambda) + \gamma\lambda) + \lambda(1- \rho) - \frac\gamma2(\Msep^2 + \lambda^2)}{1 - \Gminushat(\lambda)} \quad \text{ where } \Msep = \min\left\{M, \frac{1}{\gamma}, \lambda + \frac{\rho - \Gminushat(\lambda)}{\gamma}\right\}.
    \end{align}
    Let $\qalg = \qcrithat$.
}
\Return{$\qalg$}
\caption{\textsf{Robust Censored Well-Separated Newsvendor} (\ALGWS)}
\label{alg:newsvendor_ws}
\end{algorithm}

\subsubsection{Experimental setup.} Across all experiments we let $h = 1$ and $b = 9$, so that $\rho = 0.9$.  We consider the following demand distributions:
\begin{itemize}
    \item {\em Continuous Uniform}: Demand is drawn from a continuous uniform distribution over $[0,100]$. We let $M = 100$, and $\gamma =  1/100$.
    \item {\em Truncated Exponential}: Demand is drawn from a truncated exponential distribution with rate $1 / 80$ with support over $[0,325]$. We let $M = 325$, and $\gamma = \frac{1}{80}e^{-325/80}$.
\end{itemize}
We moreover consider $K = 2$ historical ordering quantities and vary $\lambda$ as in \cref{sec:experiment_setup}.

\subsubsection{Results.}
We demonstrate the improvements in classifying the problem as identifiable by comparing the performance of \ALGWS to that of \ALG in the knife-edge regime where $\lambda \in [\qopt_G, 1.05\qopt_G]$. Additionally, in order to isolate the {\it sample complexity} gains of \ALGWS, we modify \ALG to instead estimate the value of $\qriskws$ associated with the well-separated setting. Our results can be found in \Cref{tab:rel-vs-lam-09_ws}. The performance of the two algorithms in the likely unidentifiable and likely identifiable regimes is similar, and omitted as such.

For the Continuous Uniform distribution (\Cref{tab:uniform_ws}), \ALGWS yields modest but consistent performance gains over \ALG across all tested values of $\lambda$. Specifically, for $\lambda \in \{90.89, 91.79, 92.69\}$, \ALGWS achieves a $0\%$ misclassification rate thanks to the ``extra conditions'' for identifiability, and therefore correctly outputs an estimate of $\qopt_G$. \ALG, on the other hand, misclassifies the problem as {\em knife-edge} in $100\%$ of replications, therefore outputting $\lambda$. For $\lambda = 93.59$ and $\lambda = 94.49$, \ALGWS applies the identifiability condition in $90\%$ and $60\%$ of runs respectively (again with no misclassifications), while \ALG continues to misclassify the problem at those same rates.

For the Truncated Exponential distribution (\Cref{tab:exponential_ws}), \ALG and \ALGWS achieve identical performance. This behavior arises because the chosen value of $\gamma$ severely underestimates the local probability density around $\lambda$, rendering the additional identifiability conditions in \ALGWS inactive. This highlights the importance of having tight bounds on $\gamma$ in order to achieve sample complexity gains from this additional distributional information.

\begin{table}[!t]
\setlength\tabcolsep{3pt}
\small
\begin{subtable}[b]{\linewidth}
\centering
\begin{tabular}{lrrrrr}
\toprule
$\lambda$ & 90.89 & 91.79 & 92.69 & 93.59 & 94.49 \\ \midrule
\TrueSAA &  0.1 (0.2\%) & 0.1 (0.2\%) & 0.1 (0.1\%) & 0.1 (0.2\%) & 0.1 (0.2\%) \\
\midrule
\ALG & 0.1 (0.2\%) & 0.2 (0.4\%) & 0.4 (0.9\%) & 0.7 (1.5\%) & 0.7 (1.5\%) \\
\ALGWS & 0.1 (0.3\%) & 0.1 (0.2\%) & 0.1 (0.2\%) & 0.1 (0.3\%) & 0.1 (0.3\%) \\
\bottomrule
\end{tabular}
\caption{Uniform, $\qopt_G = 89$}
\label{tab:uniform_ws}       
\end{subtable}
\medskip
\begin{subtable}[b]{\linewidth}
\centering
\begin{tabular}{lrrrrr}
\toprule
$\lambda$ & 186.69 & 190.63 & 194.58 & 198.52 & 202.47 \\ \midrule
\midrule
\TrueSAA & 0.2 (0.1\%) & 0.2 (0.1\%) & 0.3 (0.2\%) & 0.3 (0.2\%) & 0.2 (0.1\%) \\
\midrule
\ALG & 0.1 (0.0\%) & 0.2 (0.1\%) & 0.5 (0.3\%) & 1.1 (0.6\%) & 1.8 (1.1\%) \\
\ALGWS & 0.1 (0.0\%) & 0.2 (0.1\%) & 0.5 (0.3\%) & 1.1 (0.6\%) & 1.8 (1.1\%) \\
\bottomrule
\end{tabular}
\caption{Truncated Exponential, $\qopt_G = 184.14$}
\label{tab:exponential_ws}       
\end{subtable}
\centering
\caption{{$\UncensoredRegret(q^\pi)$ vs. $\lambda$. The relative vanilla regret $\mathcal{R}^{id}(q^\pi)\%$ is shown in parentheses.}}
\label{tab:rel-vs-lam-09_ws}
\end{table}

\end{APPENDICES}







\end{document}